\documentclass{article}
\usepackage{hyperref}
\usepackage{graphicx}
\usepackage[all]{xy} 
\usepackage{amsmath}
\usepackage{authblk}
\usepackage{amsthm}
\usepackage{array}  
\usepackage{amssymb}
\usepackage{enumerate}
\usepackage{calc}
\usepackage{txfonts}
\usepackage{textcomp}
\usepackage{tabularx} 
\usepackage[table]{xcolor}
\usepackage{lscape}
\usepackage{rotating}
\usepackage{tensor}
\usepackage{arydshln} 
\usepackage{bbm} 
\usepackage{appendix} 
\usepackage{placeins} 
\usepackage[T1]{fontenc} 
\usepackage{titletoc}
\usepackage{pgf}
\usepackage{tikz}
\usetikzlibrary{positioning}
\usetikzlibrary{matrix,arrows}
\usetikzlibrary{patterns}
\usetikzlibrary{shapes,decorations,shadows}
\usetikzlibrary{calc}
\usepackage{young}
\usepackage[enableskew]{youngtab}

\DeclareMathOperator{\End}{End} 
\DeclareMathOperator{\Ext}{Ext}

\DeclareMathOperator{\add}{add}

\DeclareMathOperator{\Hom}{Hom} 
\DeclareMathOperator{\rad}{rad} 
\DeclareMathOperator{\inj}{inj} 

\DeclareMathOperator{\soc}{soc}

\DeclareMathOperator{\id}{id} 
 
\DeclareMathOperator{\Lie}{Lie} 
 
\DeclareMathOperator{\rep}{rep}

\DeclareMathOperator{\GL}{GL}

\DeclareMathOperator{\codim}{codim} 
 
\DeclareMathOperator{\Hilb}{Hilb} 
\definecolor{lblue}{rgb}{0.3,0.0,4.4} 

\newcommand*{\punkte}{\dots\unkern}
\newcolumntype{C}[1]{>{\centering\arraybackslash}p{#1}}
\newcommand{\A}{\mathcal{A}} 
\newcommand{\Fa}{\mathcal{F}} 
 
\newcommand{\Ha}{\mathcal{H}} 
\newcommand{\N}{\mathcal{N}} 
\newcommand{\Orb}{\mathcal{O}} 
 
\newcommand{\Q}{\mathcal{Q}}

\newcommand{\bN}{\mathbb N}

\newcommand{\dimv}{\underline{\dim}}
\newcommand{\dff}{\underline{d}}
\newcommand{\dfp}{\underline{d}_{P}} 
\newcommand{\df}{\underline{d}} 
\newcommand{\pp}{\mathfrak{p}}
\newcommand{\gl}{\mathfrak{gl}}

\newcommand{\NN}{\mathcal{N}}

\newcommand{\bv}{\underline{b}}

\newcommand{\bfp}{\underline{b}_P} 
 
\newcommand{\Ker}{\ker}

\newcommand{\bolda}{\boldsymbol{\lambda}} 
\newcommand{\boldmu}{\boldsymbol{\mu}} 
\newcommand{\Hil}[1]{\Hilb_0^{[#1]}(\mathbb{A}^2)}
\definecolor{lightblue}{rgb}{0.6,0.6,0.7}
\definecolor{llllightblue}{rgb}{0.75,0.75,0.85}
\newtheorem{theorem}{Theorem}[section]
\newtheorem*{Mtheorem}{Main Theorem}
\newtheorem{lemma}[theorem]{Lemma}
\newtheorem{proposition}[theorem]{Proposition}
\newtheorem{corollary}[theorem]{Corollary}
\newtheorem{remark}[theorem]{Remark}

\newtheorem{example}[theorem]{Example}
\newtheorem{claim}[theorem]{Claim}
\newtheorem*{thm:fin_cases}{Theorem} 
\newtheorem*{thm:class_max}{Theorem}
\newtheorem*{prop:class_inf}{Proposition} 
\theoremstyle{definition}
\newtheorem{definition}[theorem]{Definition}
\newtheorem{tools}[theorem]{Tools}

\author{Magdalena Boos \thanks{Ruhr-Universit\"at Bochum, Faculty of Mathematics,  D - 44780 Bochum, Germany. Magdalena.Boos-math@ruhr-uni-bochum.de}~~ and~ Michael Bulois \thanks{Univ Lyon, Universit\'{e} Jean Monnet, CNRS UMR 5208, Institut Camille Jordan, Maison de l'Uni-versit\'{e}, 10 rue Tr\'{e}filerie, CS 82301, 42023 Saint-Etienne Cedex 2, France. michael.bulois@univ-st-etienne.fr}}
\affil{}

\makeatletter
\long\def\nnfoottext#1{\insert\footins{\footnotesize
    \interlinepenalty\interfootnotelinepenalty
    \splittopskip\footnotesep
    \splitmaxdepth \dp\strutbox \floatingpenalty \@MM
    \hsize\columnwidth \@parboxrestore
   \edef\@thefnmark{}
   \edef\@currentlabel{}\@makefntext
    {\rule{\z@}{\footnotesep}\ignorespaces
      #1\strut}}}
\makeatother

\begin{document}
\makeatletter
\let\@fnsymbol\@arabic
\makeatother
\title{\bf Parabolic Conjugation and Commuting Varieties}

\nnfoottext{Keywords: parabolic subalgebra, nilpotent cone, modality, quiver with relations, covering theory, Delta-filtered modules

AMS Classification 2010: 16G20, 14R20, 14C05, 17B08}

\date{}
\maketitle

\begin{abstract}
We consider the conjugation-action of an arbitrary upper-block parabolic subgroup of the general linear group on the variety of nilpotent matrices in its Lie algebra. Lie-theoretically, it is natural to wonder about the number of orbits of this action. We translate the setup to a representation-theoretic one and obtain a finiteness criterion which classifies all actions with only a finite number of orbits over an arbitrary infinite field. These results are applied to commuting varieties and nested punctual Hilbert schemes. 

\end{abstract}
\tableofcontents
\section{Introduction}\label{intro}
The Lie-theoretical question whether an action of an algebraic group on an affine variety admits only finitely many orbits, is a very natural and basic one. For instance, the conjugation-action of the general linear group $\GL_n$ on all square-sized nilpotent matrices is finite in this way and representatives of the orbits are given by so-called Jordan normal forms \cite{Jo1}.

Many further actions have been examined in detail, some involving a parabolic subgroup $P$ of $\GL_n$. For example, the action of $P$ on the nilradical $\mathfrak{n}_{\pp}$ of its Lie algebra $\pp$ \cite{HiRoe}; or on varieties of nilpotent matrices of a certain nilpotency degree (and, in particular, on the nilpotent cone $\N$ of $\GL_n$) \cite{B2}.

Let $K$ be an arbitrary infinite field. In this work, we fix an upper-block parabolic subgroup $P$ of $\GL_n(K)$ of block sizes $\bfp:=(b_1,\punkte,b_p)$ which acts on its Lie algebra $\pp$ and on the irreducible affine variety $\N_\pp:=\pp\cap\N$ via conjugation.

The main aim of this article is to prove Theorem \ref{thm:fin_cases} and Proposition \ref{prop:class_inf} which classify all parabolic subgroups $P$ which act with only a finite number of orbits on $\N_{\pp}$. This gives

\begin{Mtheorem}
The parabolic subgroup $P$ acts finitely on $\N_{\pp}$ if and only if its block size vector appears (up to symmetry) in the diagram \ref{app:fin_case_diag}. The complementary cases to diagram \ref{app:fin_case_diag} are displayed in diagram \ref{app:inf_case_diag}. 

In particular, $P$ acts infinitely if $P$ has at least $6$ blocks, if $P$ has at least $3$ blocks of size at least $2$ or if $P$ has at least $2$ blocks of size at least $6$. 
\end{Mtheorem}

We also consider the same questions for the actions of a Levi subgroup $L_P$ of $P$ on the nilpotent cone $\N_{\pp}$ and on the nilradical $\mathfrak{n}_{\pp}$. Answers are given in Section \ref{ssect:Levi_result}.

Unless otherwise specified, we assume that $K$ is algebraically closed. We explain how to drop this hypothesis in Section \ref{sect:field}.

The first step in order to prove our  main theorem is to translate the Lie-theoretic setup to a setup in the representation theory of finite-dimensional algebras in Section \ref{sect:act_quiv}. Thus, we define a quiver with relations and a certain subcategory of its representation category such that the isomorphism classes in this subcategory correspond bijectively to the $P$-orbits in $\N_{\pp}$. One difficulty of this correspondence is that we have to look for the number of isomorphism classes for fixed dimension vectors.

The proof of the main theorem is approached from two directions, then. On the one hand, we use covering techniques \cite{BoGa,Ga3} in Section \ref{sect:covering}. This leads us to the study of a subcategory of representations of an acyclic quiver. Several ad-hoc infinite families of representations of this covering quiver are pointed out which allow us to find the infinite cases of our original problem in Proposition \ref{prop:class_inf}. We advance the theory of representations of our covering quiver, especially via the notion of $\Delta$-filtered representations \cite{DlR}. This yields the partial results of Proposition \ref{prop:reptypeGeneral} and  \ref{prop:classifFDelta}.  

On the other hand, we show that every remaining case is finite in Section \ref{sect:finite_cases}. The main idea is to reduce the problem to only four cases by reduction techniques (Section \ref{ssect:reductions}). These four cases are proved by change-of-basis-methods which make use of the representation-theoretic context and can be visualized nicely by combinatorial data.

Our main motivation for classification results such as our main theorem originates in the study of commuting varieties and Hilbert schemes. There is a well-known connection \cite{Nak} between these two varieties, the latter one being a GIT quotient of the former one. The commuting variety corresponding to the Hilbert scheme of nested (non-reduced) subschemes of the plane supported at $(0,0)$ is the nilpotent commuting variety of a parabolic subalgebra of $\mathfrak{gl}_n$ \cite{BE}. The dimension of the latter variety is controled by the modality of the action of $P$ on $\N_{\pp}$. In particular, when there are finitely many orbits, the dimension of the nilpotent commuting variety (hence, that of the corresponding nested punctual Hilbert scheme) is equal to the dimension of its \emph{principal component}. In Section \ref{sect:hilb_comm}, we discuss this in details and point out how bigger components may arise in the infinite cases.

Our method also provides valuable information in other contexts such as the (whole) commuting variety of a parabolic subalgebra. In this setting, we explain how big components can arise, even when the considered parabolic is maximal in $GL_n$. This is linked with the recent work \cite{GoGo}. Whether such results can be pushed forward to the nested punctual Hilbert scheme of the whole plane is an open question.

\medskip

{\bf Acknowledgments:} The authors would like to thank K. Bongartz for his valuable ideas for approaching and visualizing the proof of the finite case. M. Reineke is being thanked for helpful discussions concerning the methods of this work. We thank J. K\"ulshammer and U. Thiel for debating the possible use of bocs calculation with us.

\section{Theoretical background}\label{sect:theory}
We include some facts about the representation theory of finite-dimensional algebras \cite{ASS}. 
 Let $K$ be an algebraically closed field. For a fixed integer $n\in\textbf{N}$, we denote by $\GL_n\coloneqq\GL_n(K)$ the general linear group  
 and by $\gl_n$ its Lie algebra. 

Let  $\Q$  be a \textit{finite quiver}, that is, a directed graph $\Q=(\Q_0,\Q_1,s,t)$ of finitely many \textit{vertices} $i\in\Q_0$  and finitely many \textit{arrows} $(\alpha\colon s(\alpha)\rightarrow t(\alpha))\in\Q_1$ with \textit{source} map $s: \Q_1\rightarrow \Q_0$ and \textit{target} map $t: \Q_1\rightarrow \Q_0$. A \textit{path} in $\Q$ is defined to be a sequence of arrows $\omega=\alpha_l\punkte\alpha_1$, such that $t(\alpha_{k})=s(\alpha_{k+1})$ for all $k\in\{1,\punkte,l-1\}$; formally we include a path $\varepsilon_i$ of length zero for each $i\in \Q_0$ starting and ending in $i$.
We define the \textit{path algebra} $K\Q$ of $\Q$ to be the $K$-vector space with a basis given by the set of all paths in $\Q$. The multiplication of two paths $\omega= \alpha_l ... \alpha_1$ and $\omega' = \beta_q ... \beta_1$ is defined by 
\begin{center}
 $\omega\cdot\omega'=\left\{\begin{array}{ll}\omega\omega'&~\textrm{if}~t(\beta_q)=s(\alpha_1),\\
0&~\textrm{otherwise,}\end{array}\right.$\end{center}
where $\omega\omega'$ is the  concatenation of paths.

Let $\rad(K\Q)$ be the \textit{path ideal} of $K\Q$ which is the (two-sided) ideal generated by all paths of positive length. An ideal $I\subseteq K\Q$  is called \textit{admissible} if there exists an integer $s$ with $\rad(K\Q)^s\subset I\subset\rad(K\Q)^2$.

Let us denote by $\rep(\Q)$ the abelian $K$-linear category of finite-dimensional \textit{$K$-representations} of $\Q$, that is, tuples
\[((M_i)_{i\in \Q_0},(M_\alpha\colon M_i\rightarrow M_j)_{(\alpha\colon i\rightarrow j)\in \Q_1}),\] 
of $K$-vector spaces $M_i$  and $K$-linear maps $M_{\alpha}$. A \textit{morphism of representations} $M=((M_i)_{i\in \Q_0},(M_\alpha)_{\alpha\in \Q_1})$ and
 \mbox{$M'=((M'_i)_{i\in \Q_0},(M'_\alpha)_{\alpha\in \Q_1})$} consists of tuples of $K$-linear maps $(f_i\colon M_i\rightarrow M'_i)_{i\in \Q_0}$, such that $f_jM_\alpha=M'_\alpha f_i$ for every arrow $\alpha\colon i\rightarrow j$ in $\Q_1$.
 
For a representation $M$ and a path $\omega$ in $\Q$ as above, we denote $M_\omega=M_{\alpha_s}\cdot\punkte\cdot M_{\alpha_1}$. A representation $M$ is called \textit{bound by $I$} if $\sum_\omega\lambda_\omega M_\omega=0$ whenever $\sum_\omega\lambda_\omega\omega\in I$. We denote by $\rep(\Q,I)$ the category of representations of $\Q$ bound by $I$, which is equivalent to the category of finite-dimensional $K\Q/I$-representations.

Given a representation $M$ of $\Q$, its \textit{dimension vector} $\dimv M\in\mathbf{N}\Q_0$ is defined by $(\dimv M)_{i}=\dim_K M_i$ for $i\in \Q_0$. For a fixed dimension vector $\df\in\mathbf{N}\Q_0$, we denote by $\rep(\Q,I)(\df)$ the full subcategory of $\rep(\Q,I)$ of representations of dimension vector $\df$.


For certain classes of finite-dimensional algebras, a convenient tool for the classification of the indecomposable representations is the \textit{Auslander-Reiten quiver} $\Gamma(\Q,I)$ of $\rep(\Q,I)$. Its vertices $[M]$ are given by the isomorphism classes of indecomposable representations of $\rep(\Q,I)$; the arrows between two such vertices $[M]$ and $[M']$ are parametrized by a basis of the space of \textit{irreducible maps} $f\colon M\rightarrow M'$. One standard technique to calculate the Auslander-Reiten quiver for certain algebras is the \textit{knitting process} (see, for example, \cite[IV.4]{ASS}). In some cases, large classes of representations or even the whole Auslander-Reiten quiver $\Gamma(\Q,I)$ can be calculated by using \textit{covering techniques}: results about the connection between representations of the universal covering quiver (with relations) of $K\Q/I$ and the representations of $K\Q/I$ are available by P. Gabriel \cite{Ga3} and others.

 
 A finite-dimensional $K$-algebra $\A:=K\Q/I$ is called \textit{of finite representation type}, if the number of isomorphism classes of indecomposable representations is finite; otherwise it is of \textit{infinite representation type}.
%
The minimal quiver algebras (with relations) of infinite representation type have been discussed by K. Bongartz; and by D. Happel and D. Vossieck, which lead to the famous Bongartz-Happel-Vossieck list (abbreviated by BHV-list), see for example \cite{HaVo}. If a quiver with relations contains one of the listed quivers as a subquiver, then the corresponding algebra is of infinite representation type; and the given so-called \textit{nullroots} determine concrete one-parameter families of these dimension vectors.


 
For a fixed dimension vector  $\df\in\mathbf{N}\Q_0$, we define the affine space 
\[R_{\df}(\Q):= \bigoplus_{\alpha\colon i\rightarrow j}\Hom_K(K^{d_i},K^{d_j}).\] 
Its points $m$ naturally correspond to representations $M\in\rep(\Q)(\df)$ with $M_i=K^{d_i}$ for $i\in \Q_0$.  Via this correspondence, the set of representations bound by $I$ corresponds to a closed subvariety $R_{\df}(\Q,I)\subset R_{\df}(\Q)$. The group $\GL_{\df}=\prod_{i\in \Q_0}\GL_{d_i}$ acts on $R_{\df}(\Q)$ and on $R_{\df}(\Q,I)$ via base change, furthermore the $\GL_{\df}$-orbits $\Orb_M$ of this action are in bijection with the isomorphism classes of representations $M$ in $\rep(\Q,I)(\df)$.

\medskip

The following fact on associated fibre bundles sometimes makes it possible to translate an algebraic group action into another algebraic group action that is easier to understand (see, \cite{Se} amongst others).
\begin{theorem}\label{assocfibrebundles}
Let $G$ be a linear algebraic group with a closed subgroup $H$ and let $X$ be a $G-$variety. Assume that $\pi\colon X \rightarrow G/H$ is a $G$-equivariant morphism and set $F\coloneqq \pi^{-1} (eH)$. \\
Then $H$ acts on $F$ and $X$ is isomorphic to the associated fibre bundle $G\times^HF$. Furthermore, the embedding $\phi\colon F \hookrightarrow X$ induces a bijection $\Phi$ between the $H$-orbits in $F$ and the $G$-orbits in $X$ preserving orbit closure relations,
 dimensions of stabilizers (of single points), codimensions and types of singularities.
\end{theorem}
Given a $G$-variety $X$, we say that $G$ \textit{acts infinitely} on $X$, if the number of orbits of the action is infinite; and \textit{finitely}, otherwise.
We also speak about \textit{infinite} or \textit{finite actions}.

\section{Actions in the quiver context}\label{sect:act_quiv}

Fix an upper-block parabolic subgroup $P$ of $\GL_n$ of block sizes $\bfp:=(b_1,\punkte,b_p)$. We denote by $L_P$ the Levi subgroup of $P$ and by $\pp := \mathrm{Lie}(P)$ its Lie algebra. Given $x\leq n$, we define $\N_\pp^{(x)}$ as the variety of so-called $x$-nilpotent matrices in $\pp$, that is, of matrices $N\in\pp$, such that $N^x=0$. As a special case, we obtain $\N_\pp$ for $x=n$, which is the irreducible variety of nilpotent matrices in $\pp$. Define $\mathfrak{n}_{\pp}$ to be the nilradical of $\pp$.  The groups $L_P$ and $P$ act on $\N_\pp^{(x)}$ and on $\mathfrak{n}_{\pp}$ via conjugation. 

 Our main aim is to answer the following question:
\begin{center}
 \textquotedblleft For which $P$ is the number of $P$-orbits in $\N_{\pp}$ finite? \textquotedblright
\end{center} 
It is natural to look at a broader context and we will also consider  the Levi subgroup $L_P$ as the acting group; and the nilradical $\mathfrak{n}_\pp$ as the variety on which our groups act. Every such action is translated to a representation-theoretic setup by defining a suitable finite-dimensional algebra in \ref{ssect:transl}. We prove certain reduction methods in \ref{ssect:reductions} and proceed  in \ref{ssect:Levi_result} by classifying the finite actions in the Levi case.

\subsection{Translations to a quiver settings}\label{ssect:transl}
\subsubsection[The P-action on Np]{The $P$-action on $\N_{\pp}$}\label{sssect:transl_PNp}
Consider the quiver
\begin{center}\small\begin{tikzpicture}
\matrix (m) [matrix of math nodes, row sep=0.01em,
column sep=1.5em, text height=0.5ex, text depth=0.1ex]
{\Q_p\colon & \bullet & \bullet &  \bullet & \cdots  & \bullet & \bullet  & \bullet \\ & 1 & 2 &  3 & &   p-2 &  p-1  & p \\ };
\path[->]
(m-1-2) edge node[above=0.025cm] {$\alpha_1$} (m-1-3)
(m-1-3) edge  node[above=0.025cm] {$\alpha_2$}(m-1-4)
(m-1-4) edge  node[above=0.025cm] {$\alpha_3$}(m-1-5)
(m-1-5) edge  node[above=0.025cm] {$\alpha_{p-3}$}(m-1-6)
(m-1-6) edge  node[above=0.025cm] {$\alpha_{p-2}$}(m-1-7)
(m-1-7) edge node[above=0.025cm] {$\alpha_{p-1}$} (m-1-8)
(m-1-2) edge [loop above=0.05cm] node{$\beta_1$} (m-1-2)
(m-1-3) edge [loop above=0.05cm] node{$\beta_2$} (m-1-3)
(m-1-4) edge [loop above=0.05cm] node{$\beta_3$} (m-1-4)
(m-1-6) edge [loop above=0.05cm] node{$\beta_{p-2}$} (m-1-6)
(m-1-7) edge [loop above=0.05cm] node{$\beta_{p-1}$} (m-1-7)
(m-1-8) edge [loop above=0.05cm] node{$\beta_p$} (m-1-8);\end{tikzpicture}\end{center} 
together with the admissible ideal \[I_x:= (\beta_{j}^x,~1\leq j\leq p;~~ \beta_{i+1}\alpha_{i}-\alpha_i\beta_{i},~ 1\leq i\leq p-1), \qquad (x\in \mathbb{N}_{\geqslant 2})\] that is, the ideal generated by all commutativity relations and a nilpotency condition at each loop. The corresponding path algebra $\A(p,x):= K \Q_p/I_x$  with relations is finite-dimensional. Let us fix the dimension vector 
\[\dfp:=(d_1,\punkte,d_p):=(b_1,b_1+b_2, \punkte, b_1+...+b_p)\]
 and formally set $b_0=0$. As explained in Section \ref{sect:theory}, the algebraic group $\GL_{\dfp}$ acts on $R_{\dfp}(\Q_p,I_x)$; the orbits of this action are in bijection with the isomorphism classes of representations in $\rep(\Q_p,I_x)(\dfp)$.
 
Let us define $\rep^{\inj}(\Q_p,I_x)(\dfp)$ to be the full subcategory of $\rep(\Q_p,I_x)(\dfp)$ consisting of representations $((M_i)_{1\leq i\leq p},(M_{\rho})_{\rho\in (\Q_p)_1})$, such that $M_{\alpha_i}$ is injective for every $i\in\{1,\punkte, p-1\}$. Corresponding to this subcategory, there is an open subset $R_{\dfp}^{\inj}(\Q_p,I_x)\subset R_{\dfp}(\Q_p,I_x)$, which is stable under the $\GL_{\dfp}$-action.
\begin{lemma} \label{bijection}
There is an isomorphism $R_{\dfp}^{\inj}(\Q_p,I_x)\cong \GL_{\dfp}\times^{P}\N_{\pp}^{(x)}$. Thus, there exists a bijection $\Phi$ between the set of $P$-orbits in $\N_{\pp}^{(x)}$ and the set of $\GL_{\dfp}$-orbits in $R_{\dfp}^{\inj}(\Q_p,I_x)$, which sends an orbit $P.N\subseteq \N_{\pp}^{(x)}$ to the isomorphism class of the representation
\begin{center}\small\begin{tikzpicture}[descr/.style={fill=white,inner sep=2.5pt}]
\matrix (m) [matrix of math nodes, row sep=0.05em,
column sep=2em, text height=1.5ex, text depth=0.2ex]
{ K^{d_1} & K^{d_2} & K^{d_3} & \cdots  & K^{d_{p-2}} & K^{d_{p-1}}  & K^{n}\\ };
\path[->]
(m-1-1) edge node[above=0.05cm] {$\epsilon_1$} (m-1-2)
(m-1-2) edge  node[above=0.05cm] {$\epsilon_2$}(m-1-3)
(m-1-3) edge  (m-1-4)
(m-1-4) edge  (m-1-5)
(m-1-5) edge  node[above=0.05cm] {$\epsilon_{p-2}$}(m-1-6)
(m-1-6) edge node[above=0.05cm] {$\epsilon_{p-1}$} (m-1-7)
(m-1-1) edge [loop above=0.05cm] node{$N_1$} (m-1-1)
(m-1-2) edge [loop above=0.05cm] node{$N_2$} (m-1-2)
(m-1-3) edge [loop above=0.05cm] node{$N_3$} (m-1-3)
(m-1-5) edge [loop above=0.05cm] node{$N_{p-2}$} (m-1-5)
(m-1-6) edge [loop above=0.05cm] node{$N_{p-1}$} (m-1-6)
(m-1-7) edge [loop right] node{$N$,} (m-1-7);\end{tikzpicture}\end{center}
where $N_i$ is the $d_i\times d_i$-submatrix of $N$ of the first $d_i$ rows and columns and $\epsilon_i\colon K^{d_i}\hookrightarrow K^{d_{i+1}}$ are the natural embeddings. This bijection preserves orbit closure relations,
 dimensions of stabilizers (of single points), codimensions and types of singularities.
\end{lemma}
\begin{proof}
 Consider the subquiver $\widetilde{\Q_p}$ of $\Q_p$ with $(\widetilde{\Q_p})_0=(\Q_p)_0$ and $(\widetilde{\Q_p})_1=\{\alpha_1,...,\alpha_{p-1}\}$. 
 Define the $\widetilde{\Q_p}$-representation
\[y_0:= K^{d_1} \xrightarrow{\epsilon_{1}} K^{d_2} \xrightarrow{\epsilon_{2}} \cdots \xrightarrow{\epsilon_{p-2}} K^{d_{p-1}} \xrightarrow{\epsilon_{p-1}} K^{n},\]
 with $\epsilon_j$ being the canonical embedding of $K^{d_{j}}$ into $K^{d_{j+1}}$ and and denote its stabilizer in $\GL_{\dfp}$ by $H$ (this is a closed subgroup).  Then the variety $R_{\dfp}^{\inj}(\widetilde{\Q_p})$ consists of tuples of injective maps and, in particular, $R_{\dfp}^{\inj}(\widetilde{\Q_p})\cong\GL_{\dfp}/H$ equals the orbit of $y_0$.\\[1ex]
 We have a natural $\GL_{\dfp}$-equivariant projection $\pi\colon R_{\dfp}^{\inj}(\Q_p,I)\rightarrow R_{\dfp}^{\inj}(\widetilde{\Q_p})$.  
Then $H$ is isomorphic to $P$ and the fibre of $\pi$ over $y_0$ is isomorphic to $\N_{\pp}^{(x)}$. Thus, $R_{\dfp}^{\inj}(\Q_p,I_x)$ is isomorphic to the associated fibre bundle $\GL_{\dfp}\times^{P}\N_{\pp}^{(x)}$ by Theorem \ref{assocfibrebundles}, yielding the claimed bijection $\Phi$.
\end{proof}

\begin{remark}
The conjugation-action of $P$ on its nilradical has been classified by L. Hille and G. R\"ohrle \cite{HiRoe}. In particular the number of $P$-orbits on $\mathfrak{n}_{\pp}$ is shown to be finite if and only if
$p\leq 5$.\\[1ex]
The result is proved by translating the setup to a quiver-theoretic one, as well. The authors consider the quiver
\begin{center}\begin{tikzpicture}
\matrix (m) [matrix of math nodes, row sep=0.05em,
column sep=2em, text height=1.5ex, text depth=0.2ex]
{\Q'_p\colon & \bullet & \bullet &  \bullet & \cdots  & \bullet & \bullet  & \bullet  \\ };
\path[->]
(m-1-2) edge [bend left=20] node[above=0.05cm] {\begin{footnotesize}$\alpha_1$\end{footnotesize}} (m-1-3)
(m-1-3) edge [bend left=20] node[above=0.05cm] {\begin{footnotesize}$\alpha_2$\end{footnotesize}}(m-1-4)
(m-1-6) edge [bend left=20] node[above=0.05cm] {\begin{footnotesize}$\alpha_{p-2}$\end{footnotesize}}(m-1-7)
(m-1-7) edge [bend left=20] node[above=0.05cm] {\begin{footnotesize}$\alpha_{p-1}$\end{footnotesize}} (m-1-8)
(m-1-3) edge [bend left=20] node[below=0.05cm] {\begin{footnotesize}$\beta_1$\end{footnotesize}} (m-1-2)
(m-1-4) edge [bend left=20] node[below=0.05cm] {\begin{footnotesize}$\beta_2$\end{footnotesize}}(m-1-3)
(m-1-7) edge [bend left=20] node[below=0.05cm] {\begin{footnotesize}$\beta_{p-2}$\end{footnotesize}}(m-1-6)
(m-1-8) edge [bend left=20] node[below=0.05cm] {\begin{footnotesize}$\beta_{p-1}$\end{footnotesize}} (m-1-7);\end{tikzpicture}\end{center} 
 together with the relations $\beta_1\alpha_1=0$ and $\beta_i\alpha_i=\alpha_{i-1}\beta_{i-1}$ for $i\in\{2,\punkte,p-1\}$ which generate an ideal $I'_p$. They prove that the orbits of the action are in bijection with certain isomorphism classes of $K\Q'_p/I'_p$-representations and classify the latter.
\end{remark}
\subsubsection{The Levi-action on the nilradical}\label{sssect:Levi_nilr}
Consider the quiver $\Q'_{L,p}$ of $p$  vertices, such that there is an arrow $i\rightarrow j$, whenever $i<j$. For example, $\Q'_{L,5}$ is given by

\begin{center}\begin{tikzpicture}
\matrix (m) [matrix of math nodes, row sep=0.05em,
column sep=2em, text height=1.5ex, text depth=0.2ex]
{\Q'_{L,5}\colon & \bullet & \bullet &  \bullet & \bullet  & \bullet  \\ & \mathrm{1} & \mathrm{2} &  \mathrm{3} &   \mathrm{4}  & \mathrm{5} \\};
\path[->]
(m-1-2) edge   (m-1-3)
(m-1-2) edge [bend right=10] (m-1-4)
(m-1-2) edge [bend left=20] (m-1-5)
(m-1-2) edge [bend right=20]  (m-1-6)
(m-1-3) edge (m-1-4)
(m-1-3) edge [bend right=20] (m-1-5)
(m-1-3) edge [bend left=20] (m-1-6)
(m-1-4) edge   (m-1-5)
(m-1-4) edge [bend right=10] (m-1-6)
(m-1-5) edge   (m-1-6)
;\end{tikzpicture}\end{center} 
We define $\A'_{L,p}:= K \Q'_{L,p}$  to be the corresponding finite-dimensional algebra.

As explained in Section \ref{sect:theory}, the algebraic group $L_P\cong \GL_{\bfp}$ acts on $R_{\bfp}(\Q'_{L,p})$; the orbits of this action are in bijection with the isomorphism classes of representations in $\rep(\Q'_{L,p})(\bfp)$. These are in bijection with the $L_P$-orbits in $\mathfrak{n}_\pp$.

\subsubsection[The Levi-action on Np]{The Levi-action on $\N_{\pp}$}\label{sssect:Levi_Np}
Consider the quiver $\Q'_{L,p}$ defined above and add a loop $\beta_i$ at each vertex $1\leq i\leq p$; we denote the resulting quiver by $\Q_{L,p}$. Define the ideal $I$ to be generated by the relations $\beta_i^n$ for all $i$. Then the algebra $\A_{L,p}:= K\Q_{L,p}/I$ is finite-dimensional. 

As explained in Section \ref{sect:theory} and similarly to the previous case, the algebraic group $L_P\cong \GL_{\bfp}$ acts on $R_{\bfp}(\Q_{L,p},I)$ and the orbits of this action are in bijection with the isomorphism classes of representations in $\rep(\Q_{L,p},I)(\bfp)$. These are in bijection with the $L_P$-orbits in $\N_\pp$.

Note that these last constructions are easily generalized to $x$-nilpotent matrices.
\subsection{Reductions}\label{ssect:reductions}
Here, we prove three lemmas in order to compare actions of different parabolics or Levis. That is, we show three kinds of classical reductions. Analogues of these statements are available for the $P$-action on $\mathfrak{n}_\pp$ in \cite{Roe}.

 Given two tuples $(b_1,...,b_p)$ and $(b_1',...,b_q')$, we define  $(b_1,...,b_p)\leq_c(b_1',...,b_q')$ if and only if 
there is an increasing sequence $i_1<...<i_p$, such that $b_j\leq b_{i_j}'$ for all $j$.

\begin{lemma}\label{lem:red_induction}
Let $P$ and $P'$ be respective parabolic subgroups of $\GL_n$ and $\GL_{n'}$ with respective block sizes $\bv_P$ and $\bv_{P'}$ such that $\bv_P\leq_c\bv_{P'}$. \\Assume that $P$ acts infinitely on $\N_\pp$ (or $L_P$ acts infinitely on $\N_\pp$ or $\mathfrak{n}_\pp$, respectively). 
Then $P'$ acts infinitely on $\N_{\pp'}$ (or $L_{P'}$ acts infinitely on $\N_{\pp'}$ or $\mathfrak{n}_{\pp'}$, respectively). 
\end{lemma}
\begin{proof}
Denote $\bv_P$ by $(b_1,...,b_p)$ and $\bv_{P'}$ by $(b_1',...,b_q')$.
As seen before, the orbits of each action translate to certain isomorphism classes of representations. Let $(M_{t})_{t\in I}$ be an infinite family of non-isomorphic such representations.

Assume first that the acting group is $L_P$. Let $M'_t$ be the corresponding $\Q_{L,q}$- (or \mbox{$\Q'_{L,q}$-)} representation with dimension vector $\underline{c}\in \mathbb{N}^q$, where $c_{i_j}=b_j$ and $c_k=0$, otherwise. Denote by $S_i$ the simple module supported at the vertex $i$. Then the family $(N_t)_{t\in I}$, where $N_t=M'_t\oplus \bigoplus_{i=1}^p S_i^{b'_{i}-c_i}$ contains pairwise non-isomorphic representations.

Consider now the action of $P$ on $\N_{\pp}$. Formally set $i_{p+1}:=q+1$. Let $(M'_{t})_{t\in I}$ be the naturally induced family of pairwise non-isomorphic  representations in $\rep^{\inj}(\Q_q,I_n)$ defined by
\[(M'_{t})_i:=\left\{\begin{array}{ll}(M_{t})_j &\textrm{if $i_j\leqslant i< i_{j+1}$ }\\
 0 &\textrm{if $i<i_1$,}
\end{array}\right.\]
together with the induced maps $\beta'_{i,t}:=\beta_{j,t}$  ($i_j\leqslant i< i_{j+1}$) and $\alpha'_{i_{j+1}-1, t}:=\alpha_{j,t}$, further $\alpha'_{i,t}:=0$ if $i<i_1$ and $\alpha'_{i,t}$ the obvious isomorphism, otherwise.

For $1\leqslant i\leqslant q$, define a representation $U_i$ in  $\rep^{\inj}(\Q_q,I_n)$ via $(U_i)_k=K^{\delta_{k\geqslant i}}$ with injective $\alpha$ and zero $\beta$.
Then the representations $N_{t}:=M'_{t}\oplus\bigoplus_{i\neq i_j} U_i^{b'_i}\oplus\bigoplus_{j} U_{i_j}^{b'_{i_{j}}-b_j}$ ($t\in I$) form an infinite family of non-isomorphic representations. Hence $P'$ acts infinitely on $\N_{\pp'}$.\qedhere

\end{proof}
We define the transposition ${}^t(\cdot)$ to be the anti-involution of the Lie algebra $\gl_n$ which is induced by the permutation $(1,n)(2,n-2)...$ . It sends the parabolic subgroup $P$ (resp. subalgebra $\pp$) to a parabolic subgroup ${}^t\!P$ (resp. subalgebra ${}^t\pp$), such that $\dff_{({}^t\!P)}=(n-d_{p-1}, \dots , n-d_1, n)$, and $\bv_{({}^t\!P)}=(b_p, \dots , b_1)$. Hence we have
\begin{lemma}\label{lem:red_symmetry}
Let $P$ and $P'$ be parabolics of respective block sizes $\bv_P=(b_1,\dots, b_p)$ and $\bv_{P'}=(b_p,\dots,b_1)$. Then $P$ acts infinitely on $\N_\pp$ (or $L_P$ acts infinitely on $\N_\pp$ or $\mathfrak{n}_\pp$, respectively) if and only if $P'$ acts infinitely on $\N_{\pp'}$ (or $L_{P'}$ acts infinitely on $\N_{\pp'}$ or $\mathfrak{n}_{\pp'}$, respectively). 
\end{lemma}


\begin{lemma}\label{lem:red_subgroup}
Let $P$ and $P'$ be parabolic subgroups of $\GL_n$. Assume that $P$ acts infinitely on $\N_\pp$ and that $P'\subset P$. Then $P'$ acts infinitely on $\N_{\pp'}$, where $\pp'=\Lie P'$. 
\end{lemma}
\begin{proof}
Given a $P'$-orbit $P'.x$ in $\mathfrak p'$, we can associate a $P$-orbit in $\mathfrak p$ via $P'.x\mapsto P.x$. Since any $P$-orbit meets the Borel subalgebra, this map is surjective. The result follows. 
\end{proof}

\subsection{Results for Levi-actions}\label{ssect:Levi_result}
With standard techniques from quiver-representation theory, we classify the cases in which $L_P$ acts finitely on the nilpotent radical $\mathfrak{n}_\pp$ and on the nilpotent cone $\N_\pp$. 
\begin{lemma}\label{classLnilr}
 $L_P$ acts with finitely many orbits on $\mathfrak{n}_{\pp}$ if and only if $p\in\{1,2\}$, that is, if $P$ has at most two blocks.
\end{lemma}
\begin{proof}
For $p=1$, we obtain the action of $\GL_n$ on $\{0\}$ which is clearly finite.

Let $p=2$, the problem amounts to classify matrices up to equivalence. This is known to be finite.

Let $p=3$, then \begin{center}\begin{tikzpicture}[descr/.style={fill=white,inner sep=2.5pt}]
\matrix (m) [matrix of math nodes, row sep=0.05em,
column sep=2em, text height=1.5ex, text depth=0.2ex]
{M_{t}\colon & K & K &  K \\\\};
\path[->]
(m-1-2) edge node[descr] {\footnotesize{id}}  (m-1-3)
(m-1-2) edge [bend left=20]  node[above] {\footnotesize{$t\cdot \mathrm{id}$}} (m-1-4)
(m-1-3) edge  node[descr] {\footnotesize{id}} (m-1-4)
;\end{tikzpicture}\end{center} 
are pairwise non-isomorphic representations of dimension vector $(1,1,1)$. Thus, an infinite family of non-$L_P$-conjugate matrices is given by
\[\left(
\begin{array}{lll}
0 & 1 & t \\ 
0 & 0 & 1 \\ 
0 & 0 & 0
        \end{array}
 \right) \]
This induces an infinite family of non-conjugate representations for every remaining case by Lemma \ref{lem:red_induction}.
\end{proof}

\begin{lemma}\label{classLNp}
 $L_P$ acts with finitely many orbits on $\N_{\pp}$ if and only if $P=\GL_n$ or $P$ is of block sizes $(1,n-1)$ or $(n-1,1)$.
\end{lemma}
\begin{proof}
Whenever $p\geq 3$, infinitely many orbits are obtained from Lemma \ref{classLnilr}, since $\mathfrak{n}_\pp \subseteq \N_\pp$.

Let $p=2$, then an infinite family of pairwise non-isomorphic representations of dimension vector $(2,2)$ is induced by
\begin{center}\begin{tikzpicture}[descr/.style={fill=white,inner sep=2.5pt}]
\matrix (m) [matrix of math nodes, row sep=2.01em,
column sep=3.5em, text height=1.5ex, text depth=0.1ex]
{ K^2 &K^2\\};

\path[->]
(m-1-1) edge node[above=0.05cm] {\scalebox{0.6}{$\begin{pmatrix}1&0\\0&t\end{pmatrix}$}} (m-1-2)
(m-1-1) edge [loop above=0.05cm] node{\scalebox{0.6}{$\begin{pmatrix}0&1\\0&0\end{pmatrix}$}} (m-1-1)
(m-1-2) edge [loop above=0.05cm] node{\scalebox{0.6}{$\begin{pmatrix}0&1\\0&0\end{pmatrix}$}} (m-1-2);
\end{tikzpicture}\end{center} 
and gives an infinite family of pairwise non-$L_P$-conjugate matrices in $\N_{\pp}$ right away. Thus, an infinite family is induced whenever two blocks are at least of size $2$ by Lemma \ref{lem:red_induction}.

Let $P$ be of block sizes $(1,n-1)$. The corresponding representations of $\Q_{L,2}$ are of the form
\begin{center}\begin{tikzpicture}[descr/.style={fill=white,inner sep=2.5pt}]
\matrix (m) [matrix of math nodes, row sep=2.01em,
column sep=3.5em, text height=1.5ex, text depth=0.1ex]
{ K &K^{n-1}\\};

\path[->]
(m-1-1) edge node[above=0.05cm] {\scalebox{0.6}{$f$}} (m-1-2)
(m-1-1) edge [loop above=0.05cm] node{\scalebox{0.6}{$\begin{pmatrix}0\end{pmatrix}$}} (m-1-1)
(m-1-2) edge [loop above=0.05cm] node{\scalebox{0.6}{$g$}} (m-1-2);
\end{tikzpicture}\end{center} 
where $g$ is nilpotent. We can decompose $K^{n-1}=\bigoplus_i V_i$ where $g$ acts on each $V_i$ as a regular nilpotent element. Let $x=f(1)$ and decompose $x=\sum_i x_i$ with $x_i\in V_i$. Assuming that $x_i\neq 0$, let $m_i$ be the maximal index such that $x_i\in g^{m_i}(V_i)$. Then, we can construct a basis of $V_i$ with $m_i$-th element $x_i$ such that the matrix of $g_{|V_i}$ is in Jordan canonical form, in the classical way. In other words, if $\lambda_1\geqslant \dots\geqslant \lambda_k$ is the partition corresponding to the Jordan decomposition of $g$, there are, up to isomorphism of representations of $\Q_{L,2}$, at most $\prod_i (\lambda_i+1)$ choices for $f$. Finiteness follows.


If $p=1$, then $P=\GL_n$ and the finite set of nilpotent Jordan normal forms classifies the orbits.
\end{proof}

\section{Covering techniques and $\Delta$-filtered representations}\label{sect:covering}
From now on, we will restrict all considerations to the actions of $P$ on $\N_{\pp}^{(x)}$ and in particular on $\N_{\pp}$. Let us call a parabolic subgroup $P$ \textit{representation-finite}, if its action on $\N_\pp$ admits only a finite number of orbits and \textit{representation-infinite}, otherwise. 

In this section, we firstly define a covering quiver $\widehat{\Q}_p$ of $\Q_p$ together with an admissible ideal. The study of this quiver provides a finiteness criterion for the whole category $\rep(\Q_p,I_x)$ in Proposition \ref{prop:reptypeGeneral}. The rest of the section is devoted to the study of analogues of $\rep^{\inj}(\Q_p,I_n)$ in the covering context.  In subsection \ref{ssect:inf_cases} this yields several infinite families of non-isomorphic indecomposable representations. In subsections \ref{delta}, \ref{deltatorsionpair} we make use of the theory of quasi-hereditary algebras and $\Delta$-filtered modules. This allows to relate some of our subcategories of modules of the form ``$\rep^{\inj}$'' to ``whole'' categories of representations of smaller quivers.

\subsection{The covering}
By techniques of Covering Theory \cite[\S 3]{Ga3}, it is useful to look at the covering algebra first. We sketch this idea and discuss first results, now. Note that our quiver algebras with relations should be thought as locally bounded $K$-categories in \cite{Ga3}.

In order to apply results of Covering Theory, we consider the infinite  universal covering quiver of $\A(p,x)$ (at the vertex $p$), which we call $\widehat{\Q}_p$:

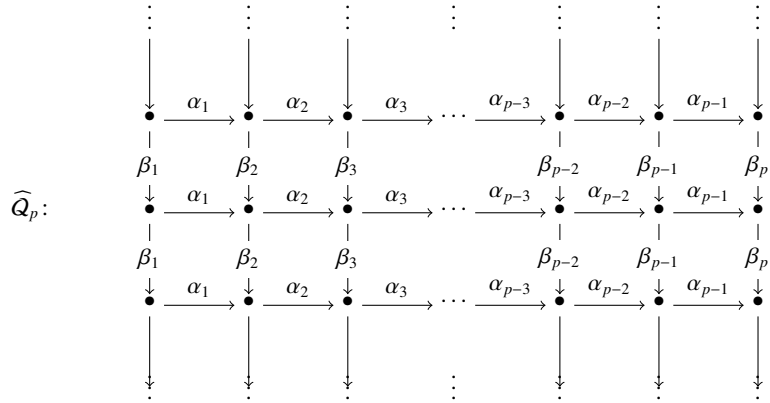
\begin{center}\small\begin{tikzpicture}[descr/.style={fill=white}]
\matrix (m) [matrix of math nodes, row sep=3em,
column sep=3em, text height=0.4ex, text depth=0.1ex]
{& \vdots & \vdots &  \vdots & \vdots  & \vdots & \vdots  & \vdots \\
& \bullet & \bullet &  \bullet & \cdots  & \bullet & \bullet  & \bullet \\
\widehat{\Q}_p\colon & \bullet & \bullet &  \bullet & \cdots  & \bullet & \bullet  & \bullet \\
& \bullet & \bullet &  \bullet & \cdots  & \bullet & \bullet  & \bullet \\
& \vdots & \vdots &  \vdots & \vdots  & \vdots & \vdots  & \vdots \\
  };
\path[->]
(m-1-2) edge  (m-2-2)
(m-1-3) edge  (m-2-3)
(m-1-4) edge  (m-2-4)
(m-1-6) edge  (m-2-6)
(m-1-7) edge  (m-2-7)
(m-1-8) edge  (m-2-8)
(m-2-2) edge node[descr] {$\beta_1$} (m-3-2)
(m-2-3) edge  node[descr] {$\beta_2$}(m-3-3)
(m-2-4) edge  node[descr] {$\beta_3$}(m-3-4)
(m-2-6) edge  node[descr] {$\beta_{p-2}$}(m-3-6)
(m-2-7) edge node[descr] {$\beta_{p-1}$} (m-3-7)
(m-2-8) edge  node[descr] {$\beta_{p}$}(m-3-8)
(m-3-2) edge node[descr] {$\beta_1$} (m-4-2)
(m-3-3) edge  node[descr] {$\beta_2$}(m-4-3)
(m-3-4) edge  node[descr] {$\beta_3$}(m-4-4)
(m-3-6) edge  node[descr] {$\beta_{p-2}$}(m-4-6)
(m-3-7) edge node[descr] {$\beta_{p-1}$} (m-4-7)
(m-3-8) edge  node[descr] {$\beta_{p}$}(m-4-8)
(m-4-2) edge (m-5-2)
(m-4-3) edge  (m-5-3)
(m-4-4) edge (m-5-4)
(m-4-6) edge  (m-5-6)
(m-4-7) edge  (m-5-7)
(m-4-8) edge (m-5-8)
(m-2-2) edge node[above] {$\alpha_1$} (m-2-3)
(m-2-3) edge  node[above] {$\alpha_2$}(m-2-4)
(m-2-4) edge  node[above] {$\alpha_3$}(m-2-5)
(m-2-5) edge  node[above] {$\alpha_{p-3}$}(m-2-6)
(m-2-6) edge  node[above] {$\alpha_{p-2}$}(m-2-7)
(m-2-7) edge node[above] {$\alpha_{p-1}$} (m-2-8)
(m-3-2) edge node[above] {$\alpha_1$} (m-3-3)
(m-3-3) edge  node[above] {$\alpha_2$}(m-3-4)
(m-3-4) edge  node[above] {$\alpha_3$}(m-3-5)
(m-3-5) edge  node[above] {$\alpha_{p-3}$}(m-3-6)
(m-3-6) edge  node[above] {$\alpha_{p-2}$}(m-3-7)
(m-3-7) edge node[above] {$\alpha_{p-1}$} (m-3-8)
(m-4-2) edge node[above] {$\alpha_1$} (m-4-3)
(m-4-3) edge  node[above] {$\alpha_2$}(m-4-4)
(m-4-4) edge  node[above] {$\alpha_3$}(m-4-5)
(m-4-5) edge  node[above] {$\alpha_{p-3}$}(m-4-6)
(m-4-6) edge  node[above] {$\alpha_{p-2}$}(m-4-7)
(m-4-7) edge node[above] {$\alpha_{p-1}$} (m-4-8);\end{tikzpicture}\end{center} 
Let $\widehat{I}_x$ be the induced ideal, generated by all nilpotency relations (for the loops at the vertices) and all commutativity relations; we see that the fundamental group is given by $\mathbf{Z}$ which acts by vertical shifts. The universal covering algebra will be denoted by $\widehat{\A}(p,x):=K\widehat{\Q}_p/\widehat{I}_x$. The special case where $x=n$ is denoted by $\widehat{\A}(p)$. In order to classify the action of $P$ on $\N_\pp$, it is useful to study the category $\rep^{\inj}(\Q_p,I_n)$ (Lemma \ref{bijection}). The corresponding representation category for the covering algebra is given by those $\widehat{\A}(p)$-representations of which all horizontal maps are injective, we call it $\rep^{\inj}(\widehat{\Q}_p,\widehat{I}_n)$. 

The quiver $\widehat{\Q}_p$ is locally bounded, so that many results of Covering Theory \cite{Ga3} can be applied. The universal covering functor $F:\widehat{\A}(p,x)\rightarrow \A(p,x)$ induces a \textquotedblleft push-down\textquotedblright-functor between the representation categories $F_{\lambda}: \rep(\widehat{\Q}_p,\widehat{I}_x) \rightarrow \rep(\Q_p,I_x)$ \cite[\S3.2]{BoGa} which has the following nice properties: 
\begin{proposition}\label{prop:F_lambda}
\begin{enumerate}
\item $F_{\lambda}$ sends indecomposable non-isomorphic  representations, which are not $\mathbf{Z}$-translates, to indecomposable non-isomorphic  representations.
\item If $M$ has dimension vector $\dff=(d_{i,j})_{i\in\mathbf{Z}, 1\leq j\leq p}$, then $F_{\lambda}(M)$ has dimension vector $\dff'=(\sum_{i\in\mathbb{Z}} d_{i,j})_{1\leq j\leq p}$.
\item $F_{\lambda}\left(\rep^{\inj}(\widehat{\Q}_p,\widehat{I}_n)\right)\subseteq \rep^{\inj}(\Q_p,I_n)$.
\item Let $\dff$ and $\dff'$ be as in 2. Then $F_{\lambda}$ induces an injective 
linear map $R_{\dff}(\widehat{\Q}_{p}, \widehat{I}_n)\rightarrow R_{\dff'}(\Q_p,I_n)$. 
\end{enumerate}
\end{proposition}

\begin{proof}
The first property follows from \cite[Lemma~3.5]{Ga3}. The other three are clear from the construction of $F_{\lambda}$ in  \cite[\S3.2]{BoGa}.
\end{proof}

By \cite[Theorem 3.6]{Ga3}, if the algebra $\widehat{\A}(p,x)$ is locally representation-finite (that is, for each vertex $y$, the number of $\widehat{\A}(p,x)$-representations $M$ with $M_y\neq \{0\}$ is finite), then the algebra $\A(p,x)$ is representation-finite. We can use the BHV-list (see, for example, \cite{HaVo}) in order to find infinitely many non-isomorphic indecomposable $\widehat{\A}(p,x)$-representations and this yields infinitely many non-isomorphic indecomposable $\A(p,x)$-representations via the functor $F_{\lambda}$.

It is also useful to define a truncated version of $\widehat{\Q}_p$.

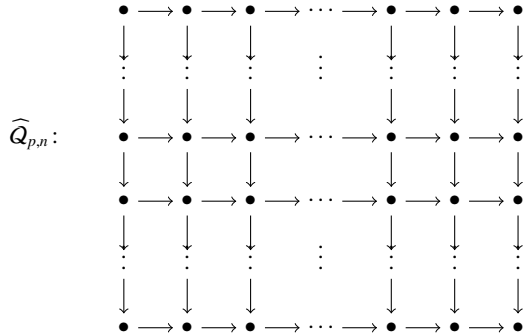
\begin{center}\small\begin{tikzpicture}[descr/.style={fill=white}]
\matrix (m) [matrix of math nodes, row sep=1.5em,
column sep=1.5em, text height=1ex, text depth=0.1ex]
{& \bullet & \bullet &  \bullet & \cdots  & \bullet & \bullet  & \bullet \\
& \vdots & \vdots &  \vdots & \vdots  & \vdots & \vdots  & \vdots \\
\widehat{\Q}_{p,n}\colon & \bullet & \bullet &  \bullet & \cdots  & \bullet & \bullet  & \bullet \\
 & \bullet & \bullet &  \bullet & \cdots  & \bullet & \bullet  & \bullet \\
& \vdots & \vdots &  \vdots & \vdots  & \vdots & \vdots  & \vdots \\
& \bullet & \bullet &  \bullet & \cdots  & \bullet & \bullet  & \bullet \\
  };
\path[->]
(m-1-2) edge  (m-2-2)
(m-1-3) edge  (m-2-3)
(m-1-4) edge  (m-2-4)
(m-1-6) edge  (m-2-6)
(m-1-7) edge  (m-2-7)
(m-1-8) edge  (m-2-8)
(m-2-2) edge  (m-3-2)
(m-2-3) edge  (m-3-3)
(m-2-4) edge  (m-3-4)
(m-2-6) edge  (m-3-6)
(m-2-7) edge  (m-3-7)
(m-2-8) edge  (m-3-8)
(m-3-2) edge  (m-4-2)
(m-3-3) edge (m-4-3)
(m-3-4) edge  (m-4-4)
(m-3-6) edge  (m-4-6)
(m-3-7) edge  (m-4-7)
(m-3-8) edge  (m-4-8)
(m-4-2) edge (m-5-2)
(m-4-3) edge  (m-5-3)
(m-4-4) edge (m-5-4)
(m-4-6) edge  (m-5-6)
(m-4-7) edge  (m-5-7)
(m-4-8) edge (m-5-8)
(m-5-2) edge (m-6-2)
(m-5-3) edge  (m-6-3)
(m-5-4) edge (m-6-4)
(m-5-6) edge  (m-6-6)
(m-5-7) edge  (m-6-7)
(m-5-8) edge (m-6-8)
(m-1-2) edge  (m-1-3)
(m-1-3) edge  (m-1-4)
(m-1-4) edge  (m-1-5)
(m-1-5) edge  (m-1-6)
(m-1-6) edge  (m-1-7)
(m-1-7) edge  (m-1-8)
(m-3-2) edge (m-3-3)
(m-3-3) edge  (m-3-4)
(m-3-4) edge  (m-3-5)
(m-3-5) edge  (m-3-6)
(m-3-6) edge (m-3-7)
(m-3-7) edge (m-3-8)
(m-4-2) edge (m-4-3)
(m-4-3) edge  (m-4-4)
(m-4-4) edge  (m-4-5)
(m-4-5) edge  (m-4-6)
(m-4-6) edge  (m-4-7)
(m-4-7) edge  (m-4-8)
(m-6-2) edge (m-6-3)
(m-6-3) edge  (m-6-4)
(m-6-4) edge  (m-6-5)
(m-6-5) edge  (m-6-6)
(m-6-6) edge  (m-6-7)
(m-6-7) edge  (m-6-8);\end{tikzpicture}\end{center} 
of $n$ rows and $p$ columns. Define $\widehat{I}_{p,n}(x)$ to be the ideal generated by all commutativity relations and by the relation that the composition of $x$ vertical maps equals zero. 
Further, define $\widehat{\A}(p,x)_n:= K\widehat{\Q}_{p,n}/\widehat{I}_{p,n}(x)$. Representations are given as tuples $(M_{k,l})_{1\leq k\leq n, 1\leq l\leq p}$ together with maps $\alpha_{k,l}: M_{k,l}\rightarrow M_{k,l+1}$ and $\beta_{k,l}: M_{k,l}\rightarrow M_{k+1,l}$ fulfilling $ \beta_{k,l+1} \circ \alpha_{k,l}=   \alpha_{k+1,l} \circ \beta_{k,l}$. If $x=n$, then define $\widehat{\A}(p)_n:=\widehat{\A}(p,x)_n$ and $\widehat{I}_{p,n}:=\widehat{I}_{p,n}(x)$.

We decide representation-finiteness concretely for $\A(p,x)$ below. In case the universal covering algebra is locally representation-finite, the algebra $\A(p,x)$ is representation-finite, as well \cite{Ga3}. If the covering algebra is representation-infinite, then $\A(p,x)$ is representation-infinite as well via the functor $F_{\lambda}$. Note that the following result has been stated in \cite{GLS}, we include it with a detailed proof for completeness.

\begin{proposition}\label{prop:reptypeGeneral}
 The algebra $\A(p,x)$ is representation-finite if and only if $p=1$ or $x=1$ or $(p,x)\in\{(2,2), (2,3), (3,2)\}$.

\end{proposition}
\begin{proof}
Let $p=1$ and $x$ be arbitrary. In this case, the indecomposable representations are (up to isomorphism) induced by the Jordan normal forms of $x$-nilpotent matrices. Therefore, there are only finitely many isomorphism classes of indecomposables, corresponding to single Jordan blocks of size at most $x$.

Let  $x=1$ and $p\geq 2$. In this case, the classification translates to an $A_p$-classification case. Thus, there are only finitely many isomorphism classes of indecomposables.

Let $p=2 $ and $x= 2$. We show that the algebra $\widehat{\A}(2,2)$ is locally representation-finite and this implies that $\A(2,2)$ is representation-finite. Let $M$ be an indecomposable finitely-generated $\widehat{\A}(2,2)$-representation. Then $M$ can be seen as a representation of $\widehat{Q}_{p,n}$ for some sufficiently large $n$. 

By knitting, we can compute the Auslander-Reiten quiver of $\widehat{\A}(2,2)_n$. It turns out that it has a middle part which is repeated in a cyclic way up to a shift by the $\mathbf{Z}$-action. Its inital, middle and terminal part are depicted in Appendix \ref{app:p2x2}. There, we denote by $M^{(i)}$ the $\widehat{\A}(2,2)_n$-representation obtained by $i$-times shifting the $\widehat{\A}(2,2)_h$-representation ($h\in\{1,2,3\}$) $M$  from bottom to top. We see that, up to $\mathbf{Z}$-action, each of the isomorphism classes of indecomposables of $\widehat{\A}(2,2)_n$ have a representant in the middle part of the Auslander Reiten quiver. 
As a consequence, the same holds for the representations of $\widehat{\A}(2,2)$, so the algebra is locally representation-finite.
%
 
Let $p=3$ and $x=2$. In the same manner as in the case $p=2$, $x=2$, one shows that  $\widehat{\A}(3,2)$ is locally representation-finite. The middle part of the Auslander-Reiten-quiver of  $\widehat{\A}(3,2)_n$ is depicted in the Appendix \ref{app:p3x2}.

Let $p=2 $ and $x=3$. Again as in the two previous cases, one shows that  $\widehat{\A}(2,3)$ is locally representation-finite. The middle part of the Auslander-Reiten-quiver of  $\widehat{\A}(2,3)_n$ is
depicted in the Appendix \ref{app:p2x3}.

For each remaining case, we find a full subquiver of the quiver $\widehat{\Q}_{p,n}$ in the BHV-list (see below for the concrete quivers), which fulfills the relations induced by $\widehat{I}_{p,n}(x)$. Via the functor $F_{\lambda}$, we obtain infinite families of non-isomorphic representations for each such algebra $\A(p,x)$. 
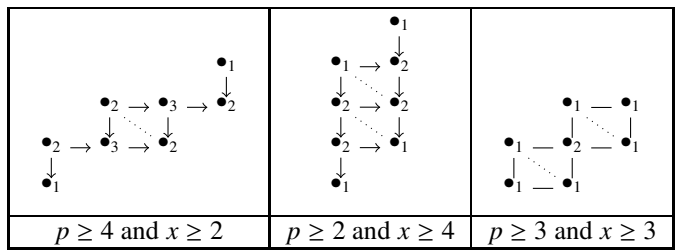
\begin{center}
\begin{tabular}{|c|c|c|}
\hline
\small\begin{tikzpicture}[descr/.style={fill=white}]
\matrix (m) [matrix of math nodes, row sep=0.9em,
column sep=0.9em, text height=0.2ex, text depth=0.1ex]
{ &  &  &    \bullet_1   \\
& \bullet_2 &  \bullet_3  & \bullet_2    \\
\bullet_2 & \bullet_3 &  \bullet_2  &     \\
\bullet_1 &  &    &     \\  };
\path[->]
(m-1-4) edge (m-2-4)
(m-2-2) edge (m-2-3)
(m-2-3) edge (m-2-4)
(m-2-2) edge  (m-3-2)
(m-2-3) edge  (m-3-3)
(m-3-1) edge  (m-3-2)
(m-3-2) edge (m-3-3)
(m-3-1) edge (m-4-1)
(m-2-2) edge[-,dotted] (m-3-3)
;\end{tikzpicture} 
&
\small\begin{tikzpicture}[descr/.style={fill=white}]
\matrix (m) [matrix of math nodes, row sep=0.9em,
column sep=0.9em, text height=0.2ex, text depth=0.1ex]
{ & \bullet_1  \\
\bullet_1 & \bullet_2  \\
\bullet_2 & \bullet_2  \\
\bullet_2 & \bullet_1  \\
\bullet_1 &   \\  };
\path[->]
(m-2-1) edge (m-2-2)
(m-3-1) edge (m-3-2)
(m-4-1) edge  (m-4-2)
(m-2-1) edge  (m-3-1)
(m-3-1) edge (m-4-1)
(m-4-1) edge (m-5-1)
(m-1-2) edge  (m-2-2)
(m-2-2) edge (m-3-2)
(m-3-2) edge (m-4-2)
(m-2-1) edge[-,dotted] (m-3-2)
(m-3-1) edge[-,dotted] (m-4-2)
;\end{tikzpicture}
&
\small\begin{tikzpicture}[descr/.style={fill=white}]
\matrix (m) [matrix of math nodes, row sep=0.9em,
column sep=0.9em, text height=0.2ex, text depth=0.1ex]
{ & \bullet_1 &  \bullet_1      \\
\bullet_1 & \bullet_2 &  \bullet_1     \\
\bullet_1 &  \bullet_1&       \\  };
\path[-]
(m-1-2) edge (m-1-3)
(m-1-2) edge (m-2-2)
(m-1-3) edge (m-2-3)
(m-3-1) edge  (m-3-2)
(m-2-2) edge  (m-2-3)
(m-2-1) edge  (m-2-2)
(m-2-1) edge (m-3-1)
(m-2-2) edge (m-3-2)
(m-1-2) edge[-,dotted] (m-2-3)
(m-2-1) edge[-,dotted] (m-3-2)
;\end{tikzpicture} 
\\
\hline
$p\geq 4$ and $x\geq 2$ & $p\geq 2 $ and $x\geq 4$ & $p\geq 3$ and $x\geq 3$\\
\hline
\end{tabular}\end{center}
\qedhere
\end{proof}

Note that this is a very general representation-theoretic approach to understand the algebra $\A(p,x)$ better. In order to solve our classification problem, it is not sufficient, since the found representations do not necessarily come up in the classification: they might not have injective maps corresponding to all arrows $\alpha_1$, ..., $\alpha_p$.

From now on, we focus on the case $x=n$.

\subsection{Infinite actions via covering}\label{ssect:inf_cases} 

Via the covering functor, every representation in $\rep^{\inj}(\widehat{\Q}_{p,n},\widehat{I}_{p,n})$ of a convenient dimension vector induces a representation in $\rep^{\inj}(\Q_p,I_n)(\dfp)$.
In order to examine the number of isomorphism classes in  $\rep^{\inj}(\Q_p,I_n)(\dfp)$, we can begin by considering isomorphism classes in  $\rep^{\inj}(\widehat{\Q}_{p,n},\widehat{I}_{p,n})$ of ``expanded'' dimension vectors, which sum up to $\dfp$, thus.

\begin{proposition}\label{prop:class_inf}
The number of $P$-orbits in $\N_{\pp}$ is infinite if $\bv_P=(b_1,\dots,b_p)$ or $\bv_{({}^t\!P)}=(b_p,\dots,b_1)$ appears in Figure \ref{app:inf_case_diag}. 
\end{proposition}
\begin{proof}
First, let us assume that $K$ is algebraically closed. We begin by proving infiniteness for the minimal cases (painted in Diagram \ref{app:inf_case_diag}). This is done by pointing out some infinite families in $\rep^{\inj}(\widehat{Q}_{p,n},\widehat{I}_{p,n})$. This provides some infinite families in $\rep^{\inj}(Q_{p}, I_{p})$ thanks to Proposition \ref{prop:F_lambda} and the result follows in these cases by Lemma \ref{bijection}.

Then the use of Lemmas \ref{lem:red_subgroup}, \ref{lem:red_induction} is depicted in Diagram \ref{app:inf_case_diag} to produce further infinite cases. The corresponding symmetric cases are infinite by Lemma \ref{lem:red_symmetry}. Induced by the quiver $\widetilde{D}_4$, we find an infinite family for block sizes $(2,2,2)$, see Figure \ref{fig:D4}.
\begin{figure}
\begin{center}
\begin{tabular}{cc}
\begin{tikzpicture}[->,>=stealth',shorten >=1pt,auto,node distance=3cm,
  thick,main node/.style={
  node distance={6ex}, minimum size=2cm,
  font=\sffamily\small\bfseries,minimum size=15mm}]
\node[main node] (1b){1};
\node[main node] (2)[below of=1b]{2};
\node[main node] (1a)[left of=2]{1}; 
\node[main node] (1c)[below of=2]{1};
\node[main node] (1d)[right of=2]{1};

\path[-, draw] (1a) -- node[above]{$a$} (2)-- node{$d$} (1d);
\path[-,draw] (1b)-- node[right]{$b$} (2) --node {$c$} (1c);
\end{tikzpicture}
&

\begin{tikzpicture}[->,>=stealth',shorten >=1pt,auto,node distance=3cm,
  thick,main node/.style={
  node distance={6ex}, minimum size=2cm,
  font=\sffamily\small\bfseries,minimum size=15mm}]
\node[main node] (1b){1};
\node[main node] (2)[below of=1b]{2};
\node[main node] (1a)[left of=2]{1}; 
\node[main node] (1c)[below of=2]{1};
\node[main node] (1c1)[left of=1c]{1};
\node[main node] (1c2)[right of=1c]{1};
\node[main node] (2p)[right of=2]{2};
\node[main node] (2pp)[right of=1b]{2};
\node[main node] (1d)[above of=2pp]{1};

\path[<-,draw] 
(1a) edge node[above]{$a$} (2)
 (2) edge node{$c$} (1c)
(1b) edge node[right]{$b$} (2)
(1b) edge node[above]{$b$} (2pp)
(1d) edge node[right]{$d$} (2pp)
(1a) edge node {$ca$} (1c1)
(2p) edge node[right]{$c$} (1c2) ;
\path[-, double,draw] 
(1c1)--(1c) --(1c2)
(2) -- (2p) --(2pp);
\end{tikzpicture}
\end{tabular}
\vspace{-1cm}
\end{center}
\caption{Infinite family induced by $\widetilde{D}_4$, $\bv=(2,2,2)$}\label{fig:D4}
\end{figure}
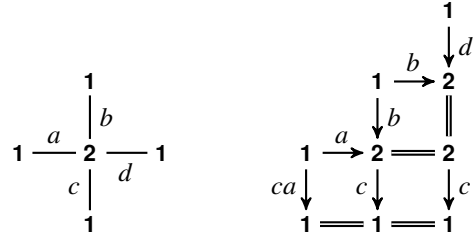
The remaining cases can be deduced from the tame quiver $\widetilde{E}_6$ and are depicted in Figure \ref{fig:E6}. For these to be admissible, there remains to make the following observation:

\begin{claim}\label{cla:inj} Whenever $a,b,c$ or $e$ are oriented from a smaller space to a bigger one in the $\widetilde{D}_4$- or $\widetilde{E}_6$-representations of Figures \ref{fig:D4} and \ref{fig:E6}, then these maps are injective maps. In particular, all the corresponding families belong to $\rep^{\inj}(\widehat{Q}_{p,n},\widehat{I}_{p,n})$.
\end{claim} 

\begin{proof}The case $1\stackrel{\alpha}{\rightarrow}2$ ($\alpha\in \{a,b,c,d,e\}$), such that no arrow ends in $1$. In this case, the kernel of $\alpha$ is a subrepresentation and any complement is a submodule. The result then follows since our one-parameter family is made of indecomposables.

The second case is $2\stackrel{b}{\rightarrow}3$ in the $\widetilde{E_6}$-case. The orientation is given by $1\stackrel{a}{\rightarrow}2\stackrel{b}{\rightarrow}3$ (note that our claim also holds true for $1\stackrel{a}{\leftarrow}2\stackrel{b}{\rightarrow}3$). By our previous considerations, $a$ is injective. Then $\Ker(b\circ a)\stackrel{a}{\rightarrow} \Ker (b)$ is a subrepresentation which has a complement being a submodule. Once again, the indecomposability allows us to conclude.\qedhere
\end{proof}

\qedhere
\end{proof}
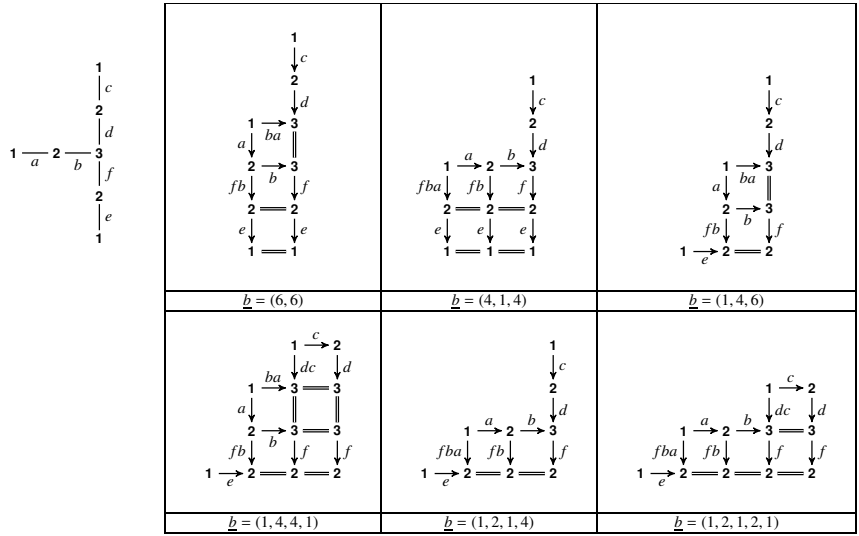
\begin{figure}
\begin{tabular}{cc}
\scalebox{0.6}{\begin{tikzpicture}[->,>=stealth',shorten >=1pt,auto,node distance=3cm,
  thick,main node/.style={
  node distance={6ex}, minimum size=2cm,
  font=\sffamily\small\bfseries,minimum size=15mm}]
\node[main node] (1b){1};
\node[main node] (2b)[below of=1b]{2};
\node[main node] (3)[below of=2b]{3}; 
\node[main node] (2a)[left of=3]{2};
\node[main node] (1a)[left of=2a]{1};
\node[main node] (2c)[below of=3]{2};
\node[main node] (1c)[below of=2c]{1};

\path[-, draw] (1a) -- node {$a$} (2a)-- node{$b$} (3)-- node {$d$}(2b)-- node {$c$} (1b);
\path[-,draw] (1c)-- node {$e$} (2c) --node {$f$} (3);
\end{tikzpicture}}
&
\scalebox{0.6}{\begin{tabular}{|c|c|c|}
\hline
\begin{tikzpicture}[->,>=stealth',shorten >=1pt,auto,
  thick,main node/.style={
 node distance={6ex}, 
  font=\sffamily\small\bfseries,minimum size=15mm}]
\node[main node] (1b){1};
\node[main node] (2b)[below of=1b]{2};
\node[main node] (3)[below of=2b]{3}; 
\node[main node] (3p)[below of=3]{3};
\node[main node] (2c)[below of=3p]{2};
\node[main node] (1c)[below of=2c]{1};
\node[main node] (2c2)[left of=2c]{2};
\node[main node] (1c2)[left of=1c]{1};
\node[main node] (1a)[left of=3]{1};
\node[main node] (2a)[left of=3p]{2};

\path[<-,draw] 
(1a) edge node {$a$} (2a)
(2a) edge node {$b$} (3p)
(2b) edge node[right] {$d$}(3)
(1b) edge node[right] {$c$} (2b)
(2c2) edge node {$e$} (1c2) 
(2c) edge node[right] {$e$} (1c) 
(2a) edge node{$fb$} (2c2)
(1a) edge node{$ba$}(3)
(3p) edge node[right] {$f$} (2c);
\path[-, double,draw] 
(2c2) -- (2c)
(3) -- (3p)
(1c2) --(1c);

\end{tikzpicture}
&
\begin{tikzpicture}[->,>=stealth',shorten >=1pt,auto,
  thick,main node/.style={
 node distance={6ex}, 
  font=\sffamily\small\bfseries,minimum size=15mm}]
\node[main node] (1b){1};
\node[main node] (2b)[below of=1b]{2};
\node[main node] (3)[below of=2b]{3}; 
\node[main node] (2a)[left of=3]{2};
\node[main node] (1a)[left of=2a]{1};
\node[main node] (2c)[below of=3]{2};
\node[main node] (1c)[below of=2c]{1};
\node[main node] (2a2)[left of=2c]{2};
\node[main node] (2a1)[left of=2a2]{2};
\node[main node] (1a2)[left of=1c]{1};
\node[main node] (1a1)[left of=1a2]{1};

\path[<-,draw] 
(1a) edge node[above] {$a$} (2a)
(2a) edge node[above] {$b$} (3)
(2b) edge node[right] {$d$}(3)
(1b) edge node[right] {$c$} (2b)
(2c) edge node {$e$} (1c) 
(3) edge node {$f$} (2c)
(1a) edge node{$fba$} (2a1)
(2a) edge node{$fb$} (2a2)
(2a1) edge node{$e$} (1a1)
(2a2) edge node{$e$} (1a2);
\path[-, double,draw] (2c)-- (2a2) --(2a1);
\path[-, double,draw] (1c)-- (1a2) --(1a1);
\end{tikzpicture} 
&
\begin{tikzpicture}[->,>=stealth',shorten >=1pt,auto,
  thick,main node/.style={
 node distance={6ex}, 
  font=\sffamily\small\bfseries,minimum size=15mm}]
\node[main node] (1b){1};
\node[main node] (2b)[below of=1b]{2};
\node[main node] (3)[below of=2b]{3}; 
\node[main node] (3p)[below of=3]{3};
\node[main node] (2c)[below of=3p]{2};
\node[main node] (2c2)[left of=2c]{2};
\node[main node] (1c)[left of=2c2]{1};
\node[main node] (1a)[left of=3]{1};
\node[main node] (2a)[left of=3p]{2};

\path[<-,draw] 
(1a) edge node {$a$} (2a)
(2a) edge node {$b$} (3p)
(2b) edge node[right] {$d$}(3)
(1b) edge node[right] {$c$} (2b)
(1c) edge node {$e$} (2c2) 
(2a) edge node{$fb$} (2c2)
(1a) edge node{$ba$}(3)
(3p) edge node[right] {$f$} (2c);
\path[-, double,draw] 
(2c2) -- (2c)
(3) -- (3p);

\end{tikzpicture}
\\
\hline
$\bv=(6,6)$ & $\bv=(4,1,4)$ & $\bv=(1,4,6)$\\
\hline
\begin{tikzpicture}[->,>=stealth',shorten >=1pt,auto,
  thick,main node/.style={
 node distance={6ex}, 
  font=\sffamily\small\bfseries,minimum size=15mm}]
\node[main node] (1b){1};
\node[main node] (2b)[right of=1b]{2};
\node[main node] (3)[below of=1b]{3}; 
\node[main node] (30)[right of=3]{3};
\node[main node] (1a)[left of=3]{1};
\node[main node] (3p)[below of=3]{3};
\node[main node] (3p0)[right of=3p]{3};
\node[main node] (2a)[left of=3p]{2};
\node[main node] (2c)[below of=3p]{2};
\node[main node] (2cp)[left of=2c]{2};
\node[main node] (2c0)[right of=2c]{2};
\node[main node] (1c)[left of=2cp]{1};

\path[<-,draw] 
(1a) edge node[left] {$a$} (2a)
(2a) edge node[below] {$b$} (3p)
(2b) edge node[right] {$d$}(30)
(1b) edge node[right] {$dc$}(3)
(1b) edge node[above] {$c$} (2b)
(1c) edge node {$e$} (2cp) 
(3p) edge node[right]{$f$} (2c)
(3p0) edge node[right]{$f$} (2c0)
(1a) edge node[above]{$ba$} (3)
(2a) edge node{$fb$} (2cp);
\path[-, double,draw] 
(3) -- (30) -- (3p0) -- (3p) -- (3)
(2cp) -- (2c) -- (2c0);
\end{tikzpicture} &
\begin{tikzpicture}[->,>=stealth',shorten >=1pt,auto,
  thick,main node/.style={
 node distance={6ex}, 
  font=\sffamily\small\bfseries,minimum size=15mm}]
\node[main node] (1b){1};
\node[main node] (2b)[below of=1b]{2};
\node[main node] (3)[below of=2b]{3}; 
\node[main node] (2a)[left of=3]{2};
\node[main node] (1a)[left of=2a]{1};
\node[main node] (2c)[below of=3]{2};
\node[main node] (2cp)[left of=2c]{2};
\node[main node] (2cpp)[left of=2cp]{2};
\node[main node] (1c)[left of=2cpp]{1};

\path[<-,draw] 
(1a) edge node[above] {$a$} (2a)
(2a) edge node[above] {$b$} (3)
(2b) edge node[right] {$d$}(3)
(1b) edge node[right] {$c$} (2b)
(1c) edge node {$e$} (2cpp) 
(3) edge node[right]{$f$} (2c)
(1a) edge node{$fba$} (2cpp)
(2a) edge node{$fb$} (2cp);
\path[-, double,draw] (2c)-- (2cp) --(2cpp);
\end{tikzpicture} 

&
\begin{tikzpicture}[->,>=stealth',shorten >=1pt,auto,
  thick,main node/.style={
 node distance={6ex}, 
  font=\sffamily\small\bfseries,minimum size=15mm}]
\node[main node] (1b){1};
\node[main node] (2b)[right of=1b]{2};
\node[main node] (3)[below of=1b]{3}; 
\node[main node] (2a)[left of=3]{2};
\node[main node] (1a)[left of=2a]{1};
\node[main node] (2c)[below of=3]{2};
\node[main node] (2cp)[left of=2c]{2};
\node[main node] (2cpp)[left of=2cp]{2};
\node[main node] (1c)[left of=2cpp]{1};
\node[main node] (2c0)[right of=2c]{2};
\node[main node] (30)[right of=3]{3};

\path[<-,draw] 
(1a) edge node[above] {$a$} (2a)
(2a) edge node[above] {$b$} (3)
(2b) edge node[right] {$d$}(30)
(1b) edge node[right] {$dc$}(3)
(1b) edge node[above] {$c$} (2b)
(1c) edge node {$e$} (2cpp) 
(3) edge node[right]{$f$} (2c)
(30) edge node[right]{$f$} (2c0)
(1a) edge node{$fba$} (2cpp)
(2a) edge node{$fb$} (2cp);
\path[-, double,draw] 
(2c0) -- (2c)-- (2cp) --(2cpp) 
(3) -- (30);
\end{tikzpicture} 
\\
\hline
$\bv=(1,4,4,1)$ &$\bv=(1,2,1,4)$ & $\bv=(1,2,1,2,1)$ \\
\hline
\end{tabular}}\end{tabular}
\caption{Infinite families induced by $\widetilde{E}_6$}\label{fig:E6}
\end{figure}

\begin{remark}
It is worth noting that our infinite families arising in Tables \ref{fig:D4} and \ref{fig:E6} all correspond to simpler families in $\rep(\widehat{Q}_{p,n-1},\widehat{I}_{p,{n-1}})$ through Theorem \ref{thm:Delta_equiv} and Proposition \ref{prop:caracHT}. Up to symmetry and bending of some arrows, they all are of the form 

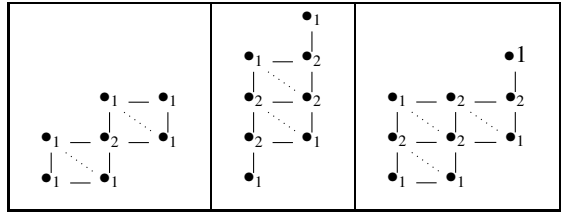
\begin{center}
\begin{tabular}{|c|c|c|}
\hline
\small\begin{tikzpicture}[descr/.style={fill=white}]
\matrix (m) [matrix of math nodes, row sep=0.9em,
column sep=0.9em, text height=0.2ex, text depth=0.1ex]
{ & \bullet_1 &  \bullet_1      \\
\bullet_1 & \bullet_2 &  \bullet_1     \\
\bullet_1 &  \bullet_1&       \\  };
\path[-]
(m-1-2) edge (m-1-3)
(m-1-2) edge (m-2-2)
(m-1-3) edge (m-2-3)
(m-3-1) edge  (m-3-2)
(m-2-2) edge  (m-2-3)
(m-2-1) edge  (m-2-2)
(m-2-1) edge (m-3-1)
(m-2-2) edge (m-3-2)
(m-1-2) edge[-,dotted] (m-2-3)
(m-2-1) edge[-,dotted] (m-3-2)
;\end{tikzpicture} 
&
\small\begin{tikzpicture}[descr/.style={fill=white}]
\matrix (m) [matrix of math nodes, row sep=0.9em,
column sep=0.9em, text height=0.2ex, text depth=0.1ex]
{ & \bullet_1  \\
\bullet_1 & \bullet_2  \\
\bullet_2 & \bullet_2  \\
\bullet_2 & \bullet_1  \\
\bullet_1 &   \\  };
\path[-]
(m-2-1) edge (m-2-2)
(m-3-1) edge (m-3-2)
(m-4-1) edge  (m-4-2)
(m-2-1) edge  (m-3-1)
(m-3-1) edge (m-4-1)
(m-4-1) edge (m-5-1)
(m-1-2) edge  (m-2-2)
(m-2-2) edge (m-3-2)
(m-3-2) edge (m-4-2)
(m-2-1) edge[-,dotted] (m-3-2)
(m-3-1) edge[-,dotted] (m-4-2)
;\end{tikzpicture}
&\small\begin{tikzpicture}[descr/.style={fill=white}]
\matrix (m) [matrix of math nodes, row sep=0.9em,
column sep=0.9em, text height=0.2ex, text depth=0.1ex]
{ &&\bullet 1\\
\bullet_1& \bullet_2 &  \bullet_2      \\
\bullet_2 & \bullet_2 &  \bullet_1     \\
\bullet_1 &  \bullet_1&       \\  };
\path[-]
(m-1-3) edge (m-2-3)
(m-2-1) edge (m-2-2)
(m-2-1) edge (m-3-1)
(m-2-2) edge (m-2-3)
(m-2-2) edge (m-3-2)
(m-2-3) edge (m-3-3)
(m-4-1) edge  (m-4-2)
(m-3-2) edge  (m-3-3)
(m-3-1) edge  (m-3-2)
(m-3-1) edge (m-4-1)
(m-3-2) edge (m-4-2)
(m-2-1) edge[-,dotted] (m-3-2)
(m-2-2) edge[-,dotted] (m-3-3)
(m-3-1) edge[-,dotted] (m-4-2)
;\end{tikzpicture} 
\\
\hline
%
%
\end{tabular}
\end{center}
\end{remark}

\subsection[Delta-filtrations]{$\Delta$-filtrations}\label{delta}
We now investigate $\rep^{\inj}(\widehat{\Q}_{p,n},\widehat{I}_{p,n})$ in more detail.
This category turns out to be a certain category of $\Delta$-filtered modules. We describe it following the constructions in \cite{DlR} and \cite{BH}.

We define $V:=\{1,...,n\}\times \{1,...,p\}$, which equals the set of vertices of $\widehat{\Q}_{p,n}$; the first entry increases from top to bottom and the second entry increases from left to right. A total ordering on $V$ is given by 
\[(i,j)\leq (k,l) \Leftrightarrow i<k~\mathrm{or}~ (i=k ~\mathrm{and} j\geq l).\] 
For $(i,j)\in V$, let $S(i,j)$ be the standard simple representation at the vertex $(i,j)$. 

The projective indecomposables $P(i,j)$ of $\widehat{\A}(p)_n$ are parametrized by $V$ and are given by
 \[P(i,j)_{k,l}=\left\lbrace
\begin{array}{ll}
K & {\rm if}~ k\geq i ~{\rm and}~ l\geq j, \\ 
0 & {\rm otherwise}.
                                                               \end{array}
 \right. \]
together with identity and zero maps, accordingly. Given a simple representation $S(i,j)$, the epimorphism  $f: P(i,j)\rightarrow S(i,j)$ is a projective cover of $S(i,j)$  by \cite[III.2.4]{ASS}.

Let $D(i,j)$ be the maximal quotient of $P(i,j)$ which admits a filtration of simple representations $S(k,l)$, such that $(k,l)\leq (i,j)$. Then 
\[D(i,j)_{k,l}=\left\lbrace
\begin{array}{ll}
K & {\rm if}~ i=k ~{\rm and}~ l\geq j, \\ 
0 & {\rm otherwise},
                                                               \end{array}
 \right. \]
and the sequence is $0\subseteq D(i,p)\subseteq \cdots \subseteq D(i,j+1) \subseteq D(i,j)$ with quotients $S(i,x)$, $j\leq x\leq p$. The module $D(i,j)$ is thus a maximal factor module of $P(i,j)$ with composition factors of
the form $S(k,l)$ with $(k,l)\leq(i,j)$. 

The representations $\Delta:=\{D(i,j)\mid (i,j)\in V\}$ are called \textit{standard} representations. We define $\Fa(\Delta)$ to be the category of all $\Delta$-filtered modules, that is, modules $M$ which admit a filtration $\{0\}=M_k\subseteq M_{k-1}\subseteq \cdots \subseteq M_{1}\subseteq M_0=M$ for some $k$ and such that for every $i$, there is a module $D\in \Delta$, such that $M_i/M_{i+1}\cong D$.

The \textit{costandard} representations $\nabla:=\{\nabla(i,j)\mid i,j\}$ and the category $\mathcal{F}(\nabla)$ are defined dually and are given by \[\nabla(i,j)_{k,l}=\left\lbrace
\begin{array}{ll}
K & {\rm if}~ k\leq i ~{\rm and}~ l=j, \\ 
0 & {\rm otherwise},
                                                               \end{array}
 \right. \]
 together with the obvious maps.
 
\begin{proposition}\label{prop:F_Delta}
\[
\begin{array}{ll}
\Fa(\Delta) &= \{M\in \rep(\widehat{\Q}_{p,n},\widehat{I}_{p,n})\mid \Hom(S(i,j),M)=0~{\rm for~}(i,j)\in V,~ j<p \} \\
 &  = \rep^{\inj}(\widehat{\Q}_{p,n},\widehat{I}_{p,n}) \\
 \Fa(\nabla) &=  \{ M\in \rep(\widehat{\Q}_{p,n},\widehat{I}_{p,n}) \mid M_{\beta_{k,l}}~{\rm is~surjective~for~all}~k,l\}
  \end{array}
 \]
\end{proposition}

\begin{proof}
We show that $\rep^{\inj}(\widehat{\Q}_{p,n},\widehat{I}_{p,n})\subset \Fa(\Delta)\subset \{M\mid \Hom(S(i,j),M)=0~{\rm for~}j<p \}\subset \rep^{\inj}(\widehat{\Q}_{p,n},\widehat{I}_{p,n})$; the proof for $\Fa(\nabla)$ is dual.

The first inclusion can be shown as follows: Let $M\in \rep^{\inj}(\widehat{\Q}_{p,n},\widehat{I}_{p,n})$. 
We show that $M$ has a $\Delta$-filtration inductively:

Define $M_0=M$. If $M_i$ has been defined, then without loss of generality, we may assume that $(M_{i})_{k,l}=K^{a_{k,l}}$ for all $k,l$. Define  $x$ to be the minimal integer, such that $(M_{i})_{x,n}\neq 0$ (then $(M_{i})_{x-1,n}= 0$) and $y$ to be the minimal integer, such that $(M_{i})_{x,y}\neq 0$ (then $(M_{i})_{x,y-1}= 0$ or $y=1$). Then define $M_{i+1}$ to be the module
 \[(M_{i+1})_{k,l}=\left\lbrace
\begin{array}{ll}
K^{a_{k,l}-1} & {\rm if}~ k=x ~{\rm and}~ l= y, \\
\alpha_{k,l-1}\circ \dots \circ \alpha_{k,y}(K^{a_{x,y}-1}) & {\rm if}~ k=x ~{\rm and}~ l> y, \\
K^{a_{k,l}} & {\rm otherwise},
                                                               \end{array}\right. \]
together with the induced natural maps. Then $M_i/M_{i+1}\cong D(x,y)$ and by induction $M$ is $\Delta$-filtered.

The second inclusion is a consequence of the fact that $\soc D(i,j)=D(i,j)_{i,p}$.

Let $M$ be a module, such that one horizontal map is not injective, say $\alpha_{i,j}$ with $i$ maximal. Since $\alpha_{i+1,j}$ is injective, it follows from the commutativity relation that $\ker(\alpha_{i,j})\subset \ker(\beta_{i,j})$. So $\ker(\alpha_{i,j})$ is a submodule of $M$ isomorphic to a sum of copies of $S(i,j)$.
Since $j<p$, the last inclusion follows.
\end{proof}

 The algebra $\Fa(\Delta)$ is strongly quasi-hereditary \cite{Ri2}, since $\End_{K\widehat{\Q}_{p,n}}(D(i,j))\cong K$ for all $(i,j)\in V$ and since the projective dimension of $D(i,j)\in\Delta$ is at most $1$. Indeed, if $i=n$, then  $D(i,j)$ is projective. Otherwise, a projective resolution is given by
\[0 \rightarrow P(i+1,j) \rightarrow P(i,j)\rightarrow D(i,j)\rightarrow 0.\]
In a similar way, the costandard modules have injective dimension at most $1$.

%
%
%
%

\subsection[Delta-filtrations via a torsion pair]{$\mathcal{F}(\Delta)$ via a torsion pair}\label{deltatorsionpair}

In the following, we translate the category $\Fa(\Delta)$ to the torsionless part of a certain torsion pair in a similar manner as in V. Dlab's and C. M. Ringel's work \cite{DlR}. We refer to \cite[Chapter VI]{ASS} for basic definitions of Tilting Theory.

Define for each $(i,j)\in V$ a module
\[T(i,j)_{k,l}=\left\lbrace
\begin{array}{ll}
K & {\rm if}~ k\leq i ~{\rm and}~ l\geq j, \\ 
0 & {\rm otherwise.}
                                                               \end{array}
 \right. \] 
 Define $T:=\bigoplus_{i,j} T(i,j)$. 
 
 \begin{lemma}
\begin{enumerate}
\item  $\mathcal{F}(\Delta)\cap\mathcal{F}(\nabla)=\mathrm{add}\;T$ equals the set of all Ext-injective modules.
\item The module $T=\bigoplus_{(i,j)\in V} T(i,j)$ is a tilting module.
\item The pair $(\Fa(\nabla), \Ha(T))$, where 
\[\Ha(T):=\{Y\in \rep(\widehat{\Q}_{p,n},\widehat{I}_{p,n})\mid \Hom(T,Y)=0\},\]
is a torsion pair.
\end{enumerate}
 \end{lemma} 
 
\begin{proof}
We know $T(i,j)\in \Fa(\Delta)\cap\mathcal{F}(\nabla)$ by Proposition \ref{prop:F_Delta}. By \cite{DlR}, this is $\Ext$-injective, while there are exactly $p\cdot n$ indecomposable $\Ext$-injective representations. Thus, 
 $\add T = \Fa(\Delta)\cap\mathcal{F}(\nabla)$. Since the projective dimension of every standard representation is at most $1$, the module $T$ is a tilting module by \cite[Lemma 4.1 ff]{DlR} and the pair $(\Fa(\nabla), \Ha(T))$ is a torsion pair by \cite[Lemma~4.2]{DlR}.
\end{proof}

 Let $\varphi$ be the endofunctor of $\rep(\widehat{\Q}_{p,n},\widehat{I}_{p,n})$ defined by $\varphi(M) = M/\eta_T(M)$, where $\eta_T(M)$ is the trace of $M$ along $T$, that is, the largest submodule of $M$ which lies in $\add T$. Let $\Fa(\Delta)/\langle T\rangle$ be the category with the same objects as $\Fa(\Delta)$, and morphisms given by residue classes of maps in $\Fa(\Delta)$:  two maps $f,g : X\rightarrow Y$ are contained in the same residue class if and only if $f - g$ factors through a direct sum of copies of $T$. 

\begin{theorem}\cite[Theorem 3]{DlR}\label{thm:Delta_equiv}
The functor $\varphi$ induces an equivalence between $\Fa(\Delta)/\langle T\rangle$ and $\Ha(T)$.
\end{theorem}

Since the indecomposable representations in $\Fa(\Delta)/\langle T\rangle$ are exactly the indecomposable representations in $\Fa(\Delta)$ except for the indecomposable representations contained in $\add T$ \cite{DlR}, we obtain the following corollary.

\begin{corollary}
The categories $\Ha(T)$ and $\Fa(\Delta)$ have the same representation type.
\end{corollary}

Thus, the knowledge of the category $\Ha(T)$ gives further insights into $\Fa(\Delta)$. We discuss the detailed structure of the former now.

\begin{proposition}\label{prop:caracHT}
$\Ha(T)\cong \rep(\widehat{\Q}_{p,n-1},\widehat{I}_{p,n-1})$
\end{proposition}

\begin{proof}
We show that the representations in $\Ha(T)$ are exactly those representations $M
$, such that $M_{1,j}=0$ for all $j$.

Assume that $M$ is a representation of $\rep(\widehat{\Q}_{p,n},\widehat{I}_{p,n})$, such that $M_{1,j}\neq 0$ for some $j$. Let $v$ be a non-zero element of $M_{1,j}$. It generates a submodule of $M$ which is a quotient of $T(n,j)$. Hence $\Hom(T,M)\neq 0$.

Let $M$ be a representation of $\rep(\widehat{\Q}_{p,n},\widehat{I}_{p,n})$, such that $M_{1,j}=0$ for all $j$. Since $\rm{Top}(T(i,j))=S(1,j)$ for all $i$, we then have $\Hom(T,M)=0$.
\end{proof}

\begin{proposition}\label{prop:classifFDelta}
There are only finitely many isomorphism classes of indecomposable representations in $\Fa(\Delta)=\rep^{\inj}(\widehat{\Q}_{p,n},\widehat{I}_{p,n})$ if and only if $p=1$ or $n\leq 2$ or $(p,n)\in\{(2,3), (3,3), (4,3), (2,4), (2,5) \}$.
\end{proposition}

\begin{proof}
Since the representation type of $\Fa(\Delta)=\rep^{\inj}(\widehat{\Q}_{p,n},\widehat{I}_{p,n})$ is the same as the one of $\Ha(T)$, the proof follows from \cite{B4}, where the representation type of $\widehat{\A}(p)_n$ is discussed. This last result is also easily recovered by arguments similar to the proof of Lemma \ref{prop:reptypeGeneral}.
\end{proof}

We end the examination of $\Ha(T)$ by considering a particular example for which we discuss $\Ha(T)$ by means of the Auslander-Reiten quiver of $\widehat{\A}(p)_n$.
\begin{example}
Let $p=2$ and $n=3$. The Auslander-Reiten quiver of $\widehat{\A}(2)_3$ is given by\\[1ex]

\setlength{\unitlength}{0.68mm}
\begin{picture}(44,20)(-90,-30)
    \put(-70,-36){\fbox{\begin{tiny}$\begin{array}{l}
0 0 \\ 
0 0 \\
0 1\end{array}$\end{tiny}}}\put(-63,-42){$\searrow$}\put(-63,-30){$\nearrow$}
    \put(-60,-24){\fbox{\begin{tiny}$\begin{array}{l}
0 0 \\ 
0 0 \\
1 1\end{array}$\end{tiny}}}\put(-53,-30){$\searrow$}
    \put(-60,-48){\fbox{\begin{tiny}$\begin{array}{l}
0 0 \\ 
0 1 \\
0 1\end{array}$\end{tiny}}}\put(-53,-42){$\nearrow$}\put(-53,-54){$\searrow$}
    \put(-50,-36){\fbox{\begin{tiny}$\begin{array}{l}
0 0 \\ 
0 1 \\
1 1\end{array}$\end{tiny}}}\put(-43,-42){$\searrow$}\put(-43,-30){$\nearrow$}\put(-43,-36){$\rightarrow$}
    \put(-50,-60){\begin{tiny}$\begin{array}{l}
\mathbf{01} \\ 
\mathbf{01} \\
\mathbf{01}\end{array}$\end{tiny}}\put(-43,-54){$\nearrow$}
    \put(-40,-24){\begin{tiny}\fbox{$\begin{array}{l}
0 0 \\ 
0 1 \\
0 0\end{array}$}\end{tiny}}\put(-33,-30){$\searrow$}
    \put(-40,-36){\fbox{\begin{tiny}$\begin{array}{l}
0 0 \\ 
1 1 \\
1 1\end{array}$\end{tiny}}}\put(-33,-36){$\rightarrow$}
    \put(-40,-48){\begin{tiny}$\begin{array}{l}
0 1 \\ 
0 1 \\
1 1\end{array}$\end{tiny}}\put(-33,-42){$\nearrow$}\put(-33,-54){$\searrow$}
\put(-30,-36){\begin{tiny}$\begin{array}{l}
0 1 \\ 
1 2\\
1 1\end{array}$\end{tiny}}\put(-23,-30){$\nearrow$}\put(-23,-42){$\searrow$}\put(-23,-36){$\rightarrow$}
    \put(-30,-60){\fbox{\begin{tiny}$\begin{array}{l}
0 0 \\ 
0 0 \\
1 0\end{array}$\end{tiny}}}\put(-23,-54){$\nearrow$}
    \put(-20,-24){\begin{tiny}$\begin{array}{l}
0 1 \\ 
1 1\\
1 1\end{array}$\end{tiny}}\put(-13,-18){$\nearrow$}\put(-13,-30){$\searrow$}
    \put(-20,-36){\begin{tiny}$\begin{array}{l}
0 1 \\ 
0 1 \\
0 0\end{array}$\end{tiny}}\put(-13,-36){$\rightarrow$}
    \put(-20,-48){\fbox{\begin{tiny}$\begin{array}{l}
0 0 \\ 
1 1 \\
1 0\end{array}$\end{tiny}}}\put(-13,-42){$\nearrow$}\put(-13,-54){$\searrow$}
  \put(-10,-12){\begin{tiny}$\begin{array}{l}
\mathbf{11} \\ 
\mathbf{11} \\
\mathbf{11}\end{array}$\end{tiny}}\put(-3,-18){$\searrow$}\put(-3,-36){$\rightarrow$}
\put(-10,-36){\begin{tiny}$\begin{array}{l}
0 1 \\ 
1 1 \\
1 0\end{array}$\end{tiny}}\put(-3,-30){$\nearrow$}\put(-3,-42){$\searrow$}
    \put(-10,-60){\fbox{\begin{tiny}$\begin{array}{l}
0 0 \\ 
1 1 \\
0 0\end{array}$\end{tiny}}}\put(-3,-54){$\nearrow$}
    \put(0,-24){\begin{tiny}$\begin{array}{l}
1 1 \\ 
1 1\\
1 0\end{array}$\end{tiny}}\put(7,-30){$\searrow$}
    \put(0,-36){\fbox{\begin{tiny}$\begin{array}{l}
0 0 \\ 
1 0 \\
1 0\end{array}$\end{tiny}}}\put(7,-36){$\rightarrow$}
    \put(0,-48){\begin{tiny}$\begin{array}{l}
0 1 \\ 
1 1 \\
0 0\end{array}$\end{tiny}}\put(7,-42){$\nearrow$}\put(7,-54){$\searrow$}

\put(10,-36){\begin{tiny}$\begin{array}{l}
1 1 \\ 
2 1 \\
1 0\end{array}$\end{tiny}}\put(17,-30){$\nearrow$}\put(17,-42){$\searrow$} \put(17,-36){$\rightarrow$}
    \put(10,-60){\begin{tiny}$\begin{array}{l}
\mathbf{01} \\ 
\mathbf{00} \\
\mathbf{00}\end{array}$\end{tiny}}\put(17,-54){$\nearrow$}
    \put(20,-24){\fbox{\begin{tiny}$\begin{array}{l}
0 0 \\ 
1 0 \\
0 0\end{array}$\end{tiny}}}\put(27,-30){$\searrow$}
    \put(20,-36){\begin{tiny}$\begin{array}{l}
\mathbf{11} \\ 
\mathbf{11} \\
\mathbf{00}\end{array}$\end{tiny}}\put(27,-36){$\rightarrow$}
    \put(20,-48){\begin{tiny}$\begin{array}{l}
1 1 \\ 
1 0 \\
1 0\end{array}$\end{tiny}}\put(27,-42){$\nearrow$}\put(27,-54){$\searrow$}
\put(30,-36){\begin{tiny}$\begin{array}{l}
1 1 \\ 
1 0 \\
0 0\end{array}$\end{tiny}}\put(37,-30){$\nearrow$}\put(37,-42){$\searrow$}
    \put(30,-60){\begin{tiny}$\begin{array}{l}
1 0 \\ 
1 0 \\
1 0\end{array}$\end{tiny}}\put(37,-54){$\nearrow$}
    \put(40,-24){\begin{tiny}$\begin{array}{l}
\mathbf{11} \\ 
\mathbf{00} \\
\mathbf{00}\end{array}$\end{tiny}}\put(47,-30){$\searrow$}
    \put(40,-48){\begin{tiny}$\begin{array}{l}
1 0 \\ 
1 0 \\
0 0\end{array}$\end{tiny}}\put(47,-42){$\nearrow$}

\put(50,-36){\begin{tiny}$\begin{array}{l}
1 0 \\ 
0 0 \\
0 0\end{array}$\end{tiny}}
\end{picture}\label{ark222}\rule{3mm}{0mm} \vspace{2.9cm}\\
The modules $T(i,j)$ are marked by bold dimension vectors; and the modules which belong to the category $\Ha(T)$ are marked by boxes.
\end{example}

\section{Finite cases}\label{sect:finite_cases}
In this section, we prove finiteness of all remaining cases which do not appear, up to symmetry, in diagram \ref{app:inf_case_diag}, cf. Proposition \ref{prop:class_inf}. 
 \begin{theorem}\label{thm:fin_cases}
The parabolic $P$ acts finitely on $\N_{\pp}$ if $\bv_{P}=(b_1,\dots,b_p)$ or $\bv_{({}^t\!P)}=(b_p,\dots,b_1)$ appears in diagram \ref{app:fin_case_diag}. 
 \end{theorem} 
Via reductions of Section~\ref{ssect:reductions}, as visualized in diagram \ref{app:fin_case_diag}, the following lemma directly proves Theorem \ref{thm:fin_cases}.
\begin{lemma}\label{lem:fin_cases}
Let $\bfp\in\{(5,k,1) , (1,3,k,1) , (3,1,k,1) , (1,1,1,k,1)|\, k\in \bN^*\}$. Then the number of isomorphism classes in $\rep^{\inj}(\Q_p,I_n)(\dfp)$ is finite.
\end{lemma}
The remainder of the section is dedicated to proving Lemma \ref{lem:fin_cases}.  Note that some of the remaining cases are known to be finite. Namely, the case $\bfp=(1,1,1,1,1)$ is proved to be finite by L. Hille and G. R\"ohrle in \cite{HiRoe}. Also, the case $\bfp=(5,k)$ has been shown to be finite by S. Murray \cite{Mur}. Independently, the cases $\bfp\in\{(1,n-1), (2,n-2)\}$ have been proved to be finite by the second author and L. Evain in \cite{BE}. 

The proof of Lemma \ref{lem:fin_cases} is structured as follows:  We begin by re-proving Murray's case $\bfp=(5,k)$ in Subsection \ref{ssect_fin_reductions} and make use of certain techniques which will be introduced in Subsections \ref{ssect_fin_notation} and \ref{ssect_base_changes}. Afterwards, we generalize these results to the four cases of Lemma \ref{lem:fin_cases} in Subsection \ref{ssect:fin_further}. 


\subsection{Notation}\label{ssect_fin_notation}

We introduce first the combinatorial data which is the central object of study in the remaining of this section.
We associate to any partition $\bolda:=(\lambda_1\geqslant \dots\geqslant\lambda_g)$ of $n$ the corresponding left-justified Young diagram with $\lambda_i$ boxes in the $i$-th row. The box in the $i$-th row and the $j$-th column is referred to as box $(i,j)$.
\begin{definition}
Given $h\in \mathbb{N}^*$, a \emph{labeled Young diagram} of $h$-tuples is a Young diagram together with an $h$- tuple $(\gamma_{i,j}^1, \dots, \gamma_{i,j}^h)$ associated to each box $(i,j)$ in the Young diagram.
\end{definition}

Given an element of $\rep^{\inj}(\Q_2, I_{n})$ of dimension $(l,n)$

\vspace{-1ex}
\begin{equation}
\begin{tikzpicture}[descr/.style={fill=white,inner sep=2.5pt}]
\matrix (m) [matrix of math nodes, row sep=0.05em,
column sep=2em, text height=1.5ex, text depth=0.2ex]
{ U & V \\ };
\draw [right hook-latex] (m-1-1) -- (m-1-2);
\path[->]
(m-1-1) edge [loop left] node{$f_{|U}$} (m-1-1)
(m-1-2) edge [loop right] node{$f$} (m-1-2)
;\end{tikzpicture}\label{fUV}
\end{equation}
\vspace{-3ex}\\
we can construct a corresponding labeled Young diagram as follows. 

We choose a basis of $V$ (resp. $U$) 
such that $f$ (resp. $f_{|U}$) is in Jordan normal form with partition $\bolda=(\lambda_1\geqslant \dots\geqslant \lambda_g)$  (resp. $\boldmu=(\mu_1\geqslant \dots \geqslant \mu_h)$) of $k$ (resp. $l$)  in a basis of $V$ (resp. $U$) of the form $(v_{i,j})_{\begin{subarray}{l}i\leqslant g\\ j\leqslant \lambda_i\end{subarray}}$ (resp. $(u_{m,t})_{\begin{subarray}{l}m\leqslant h\\ t\leqslant \mu_m\end{subarray}}$). That is, 
\[f(v_{i,j})=\left\{\begin{array}{l l}v_{i,j-1}& \textrm{if $j\geqslant 2$}\\ 0 & \textrm{else}\end{array}\right., \quad f(u_{m,t})=\left\{\begin{array}{l l}u_{m,t-1}& \textrm{if $t\geqslant 2$}\\ 0 & \textrm{else}\end{array}\right..\]
For each $i$, we set $v_i:=v_{i,\lambda_i}$ (resp. $u_m:=u_{m, \mu_m}$). We consider the decomposition $u_m=\sum_{i,j} \gamma^m_{i,j} v_{i,j}$. Then the corresponding labeled Young diagram is the Young diagram associated to $\bolda$ together with an $h$-tuple $\gamma_{i,j}:=(\gamma_{i,j}^1, \dots,\gamma_{i,j}^h)\in W:=K^h$ associated to each box $(i,j)$. 
\begin{example}\label{ex:lab_YT}
Let $\bolda:=(4,2,1)$ and $\boldmu=(2,1)$. Assume that $u_1=2v_{1,2}-3v_{1,1}-4v_{2,2}+5v_{2,1}+v_{3,1}$, $u_{2}=6v_{1,1}-7v_{2,1}+v_{3,1}$, then we obtain the labeled Young diagram
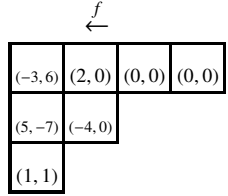
\begin{equation*}\begin{array}{c}\stackrel{f}{\leftarrow}\\
\scalebox{1.5}{
\begin{Young}
\scalebox{0.4}{$(-3,6)$}& \scalebox{0.5}{$(2,0)$}&\scalebox{0.5}{$(0,0)$}& \scalebox{0.5}{$(0,0)$}\cr
\scalebox{0.4}{$(5,-7)$}& \scalebox{0.4}{$(-4,0)$}\cr
\scalebox{0.5}{$(1,1)$}\cr
\end{Young}}\end{array}
\end{equation*}
\end{example}
The $\gamma^m_{i,j}$ are not unique in general. However, since $U=\langle f^t(u_m)\rangle_{t,m}$, they are enough to recover the isomorphism class of the original representation. We will prove in  Subsection \ref{ssect_fin_reductions} that, up to $\GL(U)\times\GL(V)$-conjugacy, the $\gamma^m_{i,j}$ can all be taken in $\{0,1\}$ whenever $l=\dim U\leqslant 5$. This implies that there are finitely many non-isomorphic representations in this case. 


\subsection{Change of basis}\label{ssect_base_changes}
The aim of this subsection is to introduce some combinatorial elementary moves on labeled Young diagrams which will be used in \ref{tool1} and \ref{tool2}. The procedure consists in \emph{reducing} the labeled Young diagram. That is, we will use some non-zero entries $\gamma_{i,j}^m$ as \emph{pivots} in order to \emph{kill} (that is, to bring to $0$ by a base change) some other entries.
The following proposition is well known \cite{Ho}.
\begin{proposition}\label{prop_adm_bc}
Performing a base change of the form
\begin{equation}\left(v_{i,j}\leftarrow \sum_{i',j'} \omega_{i,j}^{i',j'}v_{i',j'}\right)_{i,j}\label{basechange}\qquad \forall i,j,i',j': \;\omega_{i,j}^{i',j'}\in K\end{equation}
keeps $f$ in Jordan normal form if and only if
\begin{equation}\label{stab}\omega_{i,j}^{i',j'}=\left\{\begin{array}{l l}
0&\textrm{ if $j'> j$ or $\lambda_i-j>\lambda_{i'}-j'$ }\\
\omega_{i,j-1}^{i',j'-1}&\textrm{ if $2\leqslant j' \leqslant j$ and $\lambda_i-j\leqslant \lambda_{i'}-j'$}\end{array}\right.\end{equation}
\end{proposition} 

Note that the $(\omega_{i,\lambda_i}^{i',j'})_{i,i',j'}$ are enough to determine all the $(\omega_{i,j}^{i',j'})_{i,j,i',j'}$ thanks to condition \eqref{stab}. In particular, we can afford setting $\omega_{i}^{i',j'}:=\omega_{i, \lambda_i}^{i',j'}$ and writing the base change \eqref{basechange} via $\left(v_{i}\leftarrow \sum_{i',j'} \omega_{i}^{i',j'}v_{i',j'}\right)_{i}$. A bit more generally, given indices $j_i$ for each $i$, we define a base change of the form \eqref{basechange} by the formula \begin{equation*}\label{basechange2}\left(v_{i,j_i}\leftarrow \sum_{i',j'} \omega_{i}^{i',j'}v_{i',j'}\right)_{i}, \quad \textrm{ setting } \omega_{i,j}^{i',j'}:=\left\{\begin{array}{l l}\omega_i^{i', j'+j_i-j} &\textrm{ if $1\leqslant j'+j_i-j\leqslant \lambda_{i'} $}\\ 0&\textrm{else.}  \end{array}\right.\end{equation*}

Upon above notation and for any fixed $i,j$, the base changes of the following forms always satisfy condition \eqref{stab}
\begin{equation*}
M_i:\quad v_i\leftarrow \omega v_i, \quad  \textrm{ with } \omega\in K^*
\end{equation*}
\begin{equation}
C_{i,j}:\quad   v_{i,j}\leftarrow 
v_{i,j}+\!\!\!\sum_{\begin{subarray}{c}i'\leqslant i \\ j'\leqslant j\\(i',j')\neq (i,j) \end{subarray}} \!\!\!\!\omega^{i',j'} v_{i',j'}, \qquad \forall i',j': \omega^{i',j'}\in K \label{Cij}
\end{equation}
and other $v_{i'}$ ($i'\neq i$) unchanged. The same holds for the following base change for any $(i_0,j_0), (i_1,j_1)$ such that $i_0<i_1$ and $j_0>j_1$.
\begin{equation*}
B_{(i_0,j_0),(i_1,j_1)}: \quad \left(\begin{array}{c}v_{i_0,j_0}\leftarrow v_{i_0,j_0} + \omega v_{i_0,j_1} \\
v_{i_1,j_1}\leftarrow v_{i_1,j_1} - \omega v_{i_0,j_1}
\end{array}\right), \quad \omega\in K
\end{equation*}

\begin{tools}\label{tool1}
These base changes have interesting effects in our situation. We explain some of them on the coefficients $\gamma^m_{i,j}$ and pictorially on diagrams drawn as in \ref{ex:lab_YT}.
\begin{enumerate}
\item[{\bf M}] $M_i$ allows to multiply $(\gamma^m_{i,j})_{m,j}$ by $\frac{1}{\omega}$. Pictorially, this means that we can multiply row $i$ in a diagram by any non-zero scalar.
\item[{\bf C}] Assume that $\gamma^m_{i,j}=1$ for some $m,i,j$ and $\gamma_{i,j''}=0$ for any $j''>j$. A base change of the form $C_{i,j}$ allows to set $\gamma^m_{i',j'}$ to $0$ for any $(i',j')\neq (i,j)$ such that $i'\leqslant i$, $j'\leqslant j$ without modifying $\gamma_{i',j'}$ for any other couple $(i',j')$ (take $\omega^{i',j'}:=\gamma^m_{i',j'}$ in \eqref{Cij}). Moreover, if for some $m'\neq m$ we have $\gamma^{m'}_{i,j'}=0$ for any $j'$, then the $\gamma^{m'}_{i',j'}$ are not modified for any $i',j'$. 
Pictorially, this means the following. Assume that the rightmost non-zero tuple in row $i$ appears in column $j$ and is of the form $(*,*,\dots ,1,*,\dots,*)$, then the $m$-th entry of any box in the quadrant northwest to $(i,j)$ can be killed while the tuples outside this quadrant are unchanged.  Moreover, if all the $m'$-th entries (for some $m'\neq m$) are zero in row $i$, then no $m'$-th entry is modified in the labeled Young diagram. 
\item[{\bf B}] Assume that $(i_0,j_0), (i_1,j_1)$ are such that $i_0<i_1$, $j_0>j_1$, $\gamma_{i,j}=0$ for any $i,j$ such that $i=i_0$, $j\notin \{1,j_0\}$ or $i=i_1$, $j\notin \{1,j_1\}$. Assume also that $\gamma_{i_0,j_0}=\gamma_{i_1,j_1}$. A base change of the form $B_{(i_0,j_0),(i_1,j_1)}$ preserves all the tuples $\gamma_{i,j}$ except when $i=i_0$, $j=1$ where the base change implies
$\gamma_{i_0,1}\leftarrow \gamma_{i_0,1}+\omega \gamma_{i_1,1}$
\end{enumerate}
\end{tools}

Naturally, Proposition \ref{prop_adm_bc} and subsequent remarks also hold with $(u_{m,t})_{m,t}$ instead of $(v_{i,j})_{i,j}$ and $\boldmu$ instead of $\bolda$. The effect on the $\gamma^m_{i,j}$ is easier to describe. For instance, a base change of the form $\left(u_{m}\leftarrow \sum_{m'} \omega_{m}^{m'}u_{m'}\right)_{m}$ induces an action on $W$, the space of tuples, as follows: for each $(i,j)$ we get $\gamma_{i,j}\leftarrow A\gamma_{i,j}$ where $A$ is the $h\times h$-matrix whose entry in line $m$ and column $m'$ is $\omega_{m}^{m'}$. 

We describe below a few base changes of interest which always satisfy condition \eqref{stab}.
Given $s\in [\![1,h]\!]$, we define $W_s$ as the subspace of $W$ generated by the first $s$ vectors of the canonical basis $(e_1,\dots, e_h)$ and set $W_0:=\{0\}$. Given $j$, define
\begin{equation}\label{base_change_D}
 D_{j}:\left(u_{m}\leftarrow \sum_{m'} \omega_{m}^{m'}u_{m'}\right)_{m} \textrm{ with }\left\{\begin{array}{l} S=[\![m_1,\dots, m_2]\!]:=\{m\,|\,\mu_m=j\},\\ A:=(\omega_m^{m'})_{m,m'}\in \GL_{h}, \\ \forall s\notin S,\,A e_s=e_s,\; \forall s\in S,\,A e_s\in W_{m_2}. \end{array}\right.
\end{equation}
Given $m$ and $j<\mu_{m}$, we define
\begin{equation*}
E_{m,j}:\left(u_{m'}\leftarrow u_{m'}+
\omega_{m'}u_{m,j} \right)_{m'\,|\,\mu_{m'}\geqslant j}, \quad \forall m': \, \omega_{m'}\in K
\end{equation*}
and such that every remaining $u_{m'}$ is unchanged.
\begin{tools}\label{tool2}
These base changes allow the following actions on the coefficients $\gamma_{i,j}^{m}$.
\begin{enumerate}
\item[{\bf D}] 
Given $(i_{m_1},j_{m_1}),\dots, (i_{m_1+p},j_{m_1+p})$ such that the $p$ different  $(\sharp S)$-tuples $(\gamma^m_{i_s,j_s})_{m\in S}$ are linearly independant, $D_{j}$ allows to set each $(\gamma^m_{i_s,j_s})_{m\in [\![1,m_2]\!]}\in W_{m_2}$ to $e_s\in W_{m_2}$, stabilizing the $e_{s'}$ for $s'\notin S$.
\item[{\bf E}] Given $m$, assume that there exists exactly one index $i$ such that $\gamma^{m}_ {i,\mu_{m}}$ is non-zero (note that $\gamma^m_{i',j'}=0$ whenever $j'>\mu_m$). Then a base change of the form $E_{m,j}$ ($j<\mu_{m}$) allows to kill the tuple $\gamma_{i,j}$, without modifying any $\gamma_{i',j'}$ with $j'\geqslant j$ (set $\omega_{m'}:=\gamma_{i,j}^{m'}/\gamma_{i,\mu_m}^{m}$). Pictorially, if there is a single tuple with non-zero $m$-th entry in column $\mu_{m}$, then we can kill any tuple lying on the same row left of this one without modifying tuples on columns right to the annihilated one.  
\end{enumerate}
\end{tools}

\subsection{Reductions}\label{ssect_fin_reductions}
With the combinatorial tools of the previous subsection, the game is to reduce every possible diagram as in \ref{ex:lab_YT} to a diagram with coefficients in $\{0,1\}$. First note that, for our bases $(v_{i,j})_{i,j}$ and $(u_{i,j})_{i,j}$, we have $\gamma^m_{i,j}=0$ whenever $j>\mu_m$, since $u_m\in Ker(f^{\mu_m})$. In particular, only columns of index less or equal than $\mu_1$ may have non-zero tuples.  The usual procedure will consider columns from right to left and, in each column, we will proceed from bottom to top. We prove the following proposition.

\begin{proposition}\label{prop_reduced_cases}
With base changes of the forms \ref{tool1} and \ref{tool2}, we can reduce the setup to a case where the labeled Young diagram satisfies the following conditions 
\begin{enumerate}
\item In any box $(i,j)$, except possibly one, the tuple $\gamma_{i,j}$ is either zero or of the form $e_s$, for some $s$ such that $\mu_s\geqslant j$.\\ The only exception may arise if $\boldmu =(3,2)$ and $(i,j)=(i_0,1)$ for some unique row index $i_0$. Then $\gamma_{i_0,1}=(1,1)$. 
\item In each row $i$, except possibly one, there exists at most one non-zero tuple $\gamma_{i,j}$.\\
If exists, the exceptional row is denoted by $i_*$, and there exists only two non-zero tuples $\gamma_{i_*,j}$. One of these two tuples has to be in column $j=1$.
\item In each column $j\neq 1$, for each $s$, there is at most one $\gamma_{i,j}$ equal to $e_s$. For each $s$ there exists at most one index $i\neq i_*$ such that $\gamma_{i,1}$ equals $e_s$.
\end{enumerate}
\end{proposition}


At some point in the proof, we will meet some indices called $i_*$ and $i_0$. They should be seen as candidates to be the specific indices of the proposition. It might turn out that $\gamma_{i_0,1}=e_s$ or $\gamma_{i_*,1}=0$, in which case the specific index should be discarded. 
\begin{proof} 
Since $\boldmu$ is a partition of $l\leqslant 5$, we can distinguish three cases: 
\begin{enumerate}[a)]
\item $\boldmu=(2^a, 1^{l-2a})$ for some $a\geqslant 0$.
\item $\boldmu=(a,1^{l-a})$ for some $a\geqslant 3$,
\item $\boldmu=(3,2)$,
\end{enumerate}

{\bf Case a)} Set $j:=\mu_1 \in \{1,2\}$.\\[1ex]
\underline{First step of case a)}: We focus at first on column $j$.\\[1ex] 
We initialize $s$ to $0$ and apply the following iterative procedure for decreasing $i$ from $\max\{i | \lambda_i\geqslant j \} $ to $1$. During each loop, no $\gamma_{i',j}$ for $i'>i$ is modified. After each loop, we will have $\gamma_{i,j}\in \{0,e_s\}$ while $\gamma^{s'}_{i',j}=0$ for any $s'\leqslant s$ and $i'<i$. In particular, we will never begin a loop with $\gamma_{i,j}\in W_s\setminus\{0\}$.

Given $i$, if $\gamma_{i,j}\neq 0$, then $\gamma_{i,j}\notin W_s$. So a base change of the form $D_{j}$ sends $\gamma_{i,j}$ to $e_{s+1}$ and preserves $e_1, \dots e_s$. Apply then a base change of the form $C_{i,j}$ to bring any $\gamma^{s}_{i',j'}$ to $0$ for any $i'\leqslant i$, $j'\leqslant j$, $(i',j')\neq (i,j)$. If $j=2$, furthermore apply a base change of the form $E_{s,1}$ to provide $\gamma_{i,1}=0$. Set $i_{s+1}:=i$ for later use and $s\leftarrow s+1$.

If $\gamma_{i,j}=0$, we do nothing.

\medskip
We are thus left with a picture of the following form. Note that this gives the desired result when $j=1$, that is, when $\boldmu=(1^l)$.

\begin{equation}\label{221} \begin{array}{c c}
\scalebox{1.3}{\begin{Young}
\scalebox{0.4}{$(0,0,\underline{*})$}&$0$&$0$ \cr
$0$&\scalebox{0.4}{$(0,1,\underline{0})$}& $0$  \cr
 \scalebox{0.4}{$(0,\underline{*})$}& $0$ &$0$ \cr
 \scalebox{0.4}{$(0,\underline{*})$}&$0$ \cr
$0$&\scalebox{0.4}{$(1,\underline{0})$}\cr
\scalebox{0.8}{$(\underline{*})$}&$0$\cr
\scalebox{0.8}{$(\underline{*})$}\cr
\end{Young}} & \scalebox{1.4}{$\begin{array}{c}i_2\\ \\\\ i_1\\ \\ \\ \\ \\ \\ \\ \\ \end{array}$}
\end{array}
\end{equation}
\vspace{-20ex}

\underline{Second step of case a)}: All that remains to be done is to reduce the first column in the case $\mu_1=2$. This will be achieved using base changes on $U$ of the form $D_j$ where the matrix $A$ is in upper triangular form.

We initialize a variable subset $S\subseteq[\![1,h]\!]$ to $\emptyset$ and apply an iterative procedure for decreasing $i=g,\dots, 1$. After each loop the above shape \eqref{221} is preserved and no $\gamma_{i',1}$ is modified for $i'>i$. Moreover, we will have $\gamma^s_{i',1}=0$ for any $i'<i$ and $s\in S$. In particular, we will have $\gamma_{i,1}\notin W_s\setminus W_{s-1}$ for any $s\in S$ at the beginning of every loop.

Given $i$, such that $\gamma_{i,1} \neq0$:\vspace{-1.5 ex}
\begin{itemize}
\item We define $s\in [\![1,h]\!]$ via $\gamma_{i,1}\in W_{s}\setminus W_{s-1}$. \vspace{-1.5ex}
\item If $s>a$ (resp. if $s\leqslant a$), then a base change of the form $D_1$ (resp. $D_2$) brings $\gamma_{i,1}$ to $e_s$, fixing each $e_{s'}$ for $s\neq s'$. Meanwhile, among the previously fixed tuples, only $\gamma_{i_s,2}$ may have changed (if $s\leqslant a$), and the new tuple still lies in $W_s$, since the matrix $A$ in \eqref{base_change_D} of this particular base change is triangular. In this case, since $i_{s'}>i_s$ for $s'<s$, base changes of the form $C_{i_{s'},2}$ ($s'\leqslant s$) and $M_{i_s}$ allow to bring $\gamma_{i_s,2}$ back to $e_s$. \vspace{-1.5ex}
\item Applying $C_{i,1}$ allows then to set $\gamma^s_{i',1}=0$ for any $i'<i$. \vspace{-1.5ex}
\item Set $S\leftarrow S\cup \{s\}$.
\end{itemize}

If $\gamma_{i,1} =0$, we do nothing.

\medskip

{\bf Case b)}. Set $\mu_2:=0$ if $h=1$, that is, if $\boldmu=(\mu_1)$. We have $\mu_1>\mu_2$ and the only non-zero coefficients $\gamma^m_{i,j}$ with $\mu_2<j\leqslant \mu_1$ arise when $m=1$.\\[1ex]
\underline{First step of case b)}:  Iteratively for decreasing $j=\mu_1, \dots, \mu_2+1$ we apply the following algorithm. 
\vspace{-1.5ex}
\begin{itemize}
\item Consider the lowermost non-zero tuple in column $j$ (if it exists) and denote by $i_j$ the index of its row. \vspace{-1.5ex}
\item Apply a base change of the form $M_{i_j}$ to bring $\gamma_{i_j,j}$ to $e_1$. \vspace{-1.5ex}
\item Apply a base change of the form $C_{i_j,j}$ to kill the first entry of tuples in the quadrant northwest to $(i_j,j)$. In particular, $\gamma_{i',j}=0$ when $i'<i_j$. \vspace{-1.5ex}
\end{itemize}

Note that this gives Proposition \ref{prop_reduced_cases} if $h=1$.

From now on, let $h\geqslant 2$, that is, $\mu_2=1$. All that remains to be done is to modify the entries on the first column. 
Some difficulties arise when considering non-zero tuples $\gamma_{i_j,1}$ for some $1<j\leqslant\mu_1-1$.  Since $l\leqslant 5$, either $h=2$, so that all such non-zero tuples are colinear (multiples of $(0,1)$), or $h\geqslant 3$, so $\mu_1\leqslant 3$. In the first case, base changes of the form $B_{(i_j,j), (i_{j'},j')}$ for $j>j'$ kill all $\gamma_{i_{j'},1}$ but the lowermost non-zero one (if exists). In the second case, we can apply a base change of the form $E_{1,i_{\mu_1}}$ to kill $\gamma_{i_{\mu_1},1}$. In the latter case, the only possible remaining non-zero tuple $\gamma_{i_{j},1}$ arises for $j=2$.

%
%

In a compatible way to Proposition \ref{prop_reduced_cases}, we define $i_*:=i_j$ where $j$ is the index such that $\gamma_{i_j,1}\neq 0$, if exists. We are thus left with a picture of the following form. 
\begin{equation}\label{first step} \begin{array}{c c} 
\scalebox{1.3}{\begin{Young}
\scalebox{0.6}{$(0,\underline{*})$}& $0$& $0$& $0$& $0$&$0$\cr
$0$&$0$& $0$&$0$&\scalebox{0.6}{(1,\underline{0})} &$0$ \cr
$0$& 0& \scalebox{0.6}{(1,\underline{0})}&$0$\cr
\scalebox{0.6}{$(0,\underline{*})$}&$0$ &$0$ \cr
\scalebox{0.6}{$(0,\underline{*})$}& \scalebox{0.6}{(1,\underline{0})}\cr
\scalebox{0.8}{$(\underline{*})$}&$0$\cr
\scalebox{0.8}{$(\underline{*})$}\cr
\end{Young}}&\scalebox{1.4}{$\begin{array}{l}\\i_5\\ i_3\\\\ i_2=i_*\\ \\\\\\\\\\\\\\\\\end{array}$}
\end{array}
\end{equation}
\vspace{-25ex}
$$\!\!\!\!\!\!\!\!\!\!\!\!\!\!\!\!\!\!\!\!\!\!\!\scalebox{1.4}{$\begin{array}{c c c cc}&&& & \mu_1 \end{array}$}$$

\underline{Second step of case b)}: 

We initialize $S$ to $\emptyset$ and we apply an iterative procedure for decreasing $i=g,\dots, 1$. After each step the above shape \eqref{first step} is preserved (with $\mu_2$=1) and for each $i'>i$, the tuple $\gamma_{i',1}(\in \{0\}\cup \{e_s\, | \,s\in S\cup\{2\} \})$ is preserved (up to a possible permutation of the coordinates when $i=i_*$). When $i_*>i$, we will have $\gamma_{i_*,1}\in \{0,e_2\}$ and this is the only possible case of an index $i'>i$ such that $\gamma_{i',1}=e_2$ and $2\notin S$. 
We will also have $\gamma^s_{i',1}=0$ for any $i'<i$ and any $s\in S$. In particular, we will have $\gamma_{i,1}\notin W_s\setminus W_{s-1}$ for every $s\in S$ at the beginning of every loop. 
\medskip

Given $i$, if $\gamma_{i,1}\neq 0$, we consider the following cases: \vspace{-1.5ex}
\begin{itemize}
\item Assume that ($i>i_*$) or ($i<i_*$ and $\gamma_{i,1}\notin K\gamma_{i_*,1}\subset K e_2$) or ($i_*$ does not exists).\\ Define $s$ as the index such that $\gamma_{i,1}\in W_s\setminus W_{s-1}$. We have $s\notin S$.\vspace{-1.5ex} 
\begin{itemize}
\item If $s=1$, apply $M_i$ to bring $\gamma_{i,1}$ to $e_1$.\vspace{-1ex}
\item If $s\neq 1$, a base change of the form $D_1$ brings $\gamma_{i,1}$ to $e_s$,
thereby fixing each tuple in the picture which is equal to $e_{s'}$ for some $s'\neq s$. Note that whenever $i<i_*$ and $\gamma_{i_*,1}=e_2$, we have $\gamma^{1}_{i,1}=0$, so $s\neq 2$ and $\gamma_{i_*,1}$ is preserved. \vspace{-1.5ex}
\end{itemize}
Then apply a base change of the form $C_{i,1}$ to kill each $\gamma_{i',1}^s$ ($i'<i$).\\ Set $S\leftarrow S\cup\{s\}$.\vspace{-1ex}
\item If $i=i_*$, then $\gamma_{i,1}\in W_s\setminus W_{s-1}$ for some $s\notin S\cup\{1\}$. A base change of the form $D_1$ turns $\gamma_{i,1}$ to $e_s$, fixing each $e_{s'}$ ($s'\neq s$).\\ 
Define the map $\sigma:[\![1,h]\!]\rightarrow [\![1,h]\!]$ via $\sigma(s)=2$, $\sigma(s')=s'+1$ for $2\leqslant s'<s$, and $\sigma(s)=s$ otherwise. An additional base change of the form $D_1$ sends each $e_{s'}$ to $e_{\sigma(s')}$. Set $S\leftarrow \sigma(S)$, then.\vspace{-1 ex}
\item If $i<i_*$ and $\gamma_{i,1}\in K \gamma_{i_*,1}$, then a base change of the form $M_i$ brings $\gamma_{i,1}$ to $e_2$. Apply $C_{i,1}$ to kill each $\gamma_{i',1}^2$ ($i'<i$), then. Set $S\leftarrow S\cup\{2\}$.
\end{itemize}
If $\gamma_{i,1}=0$, we do nothing.

\medskip

{\bf Case c)}: $\boldmu=(3,2)$.\\[1ex]
\underline{First step of case c)}:  Arguing as in case a) with $\boldmu=(2,1)$, we can reduce columns $3$ and $2$. Then there exists a single index $i_3$ such that $\gamma_{i_3,3}\neq 0$, and $\gamma_{i_3,3}=(1,0)$. There also exist at most $2$ indices $i=i_2,i_*$, such that $\gamma_{i,2}\neq 0$ and we obtain $\gamma_{i_2,2}=(0,1)$ and $\gamma_{i_*,2}=(1,0)$. Moreover, $i_3,i_2,i_*$ are distinct and $i_*>i_3$.

Next, we apply $E_{1,1}$ and $E_{2,1}$ to kill $\gamma_{i_3,1}$ and $\gamma_{i_2,1}$. We also apply $C_{i_3,3}$, and $C_{i_*,2}$ (resp. $C_{i_2,2}$) to set $\gamma^1_{i,1}=0$ (resp. $\gamma^2_{i,1}$) for $i\leqslant \max(i_3,i_ *)$ (resp. $i\leqslant i_2$). The picture then looks as follows
    
%
%

\begin{equation*} \begin{array}{c c}\scalebox{1.3}{
\begin{Young}
$0$ & $0$ & $0$  & $0$  \cr
$0$ & \scalebox{0.6}{$(0,1)$}&$0$& $0$ \cr
\scalebox{0.6}{$(0,*)$}&$0$& $0$ \cr
$0$& $0$ &\scalebox{0.6}{$(1,0)$}\cr
\scalebox{0.6}{$(0,*)$}&$0$&$0$\cr
\scalebox{0.6}{$(0,*)$}& \scalebox{0.6}{$(1,0)$}\cr
\scalebox{0.8}{$(\underline{*})$}&$0$\cr
\scalebox{0.8}{$(\underline{*})$}\cr
\end{Young}
}& \scalebox{1.4}{$\begin{array}{l} i_2\\\\ i_3\\\\i_*\\\\\\\\\\\\\\\\\\\end{array}$}
\end{array}
\end{equation*}
\vspace{-22ex}

\underline{Second step of case c)}:
Consider now the lowermost non-zero tuple of the first column $\gamma_{i_0,1}$. Here, the base changes used on $U$ are only of the form 
\begin{equation} u_1\leftarrow \alpha u_1, \quad u_2\leftarrow \beta u_2.\label{mult_only}\end{equation}

\begin{itemize}
\item Assume that $\gamma_{i_0,1}\notin Ke_1\cup Ke_2$. Then $i_0\notin\{i_3,i_2,i^*\}$ and $i_0> i_*$ (otherwise $\gamma^1_{i_0,1}=0$).\\
A base change of the form \eqref{mult_only} followed by base changes of the form $M_{i_1}$, $M_{i_2}$, $M_{i_3}$ brings $\gamma_{i_0,1}$ to $(1,1)$ without modifying columns $2$ and $3$.\\
 Apply a base change of the form $C_{i_0,1}$ to kill any $\gamma^2_{i',1}$ ($i'< i_0$) and note that the corresponding entries $\gamma^1_{i',1}$ are changed, thereby. But, in particular if $i_*$ exists, the whole tuple $\gamma_{i_*,1}$ is killed by an additional base change of the form $C_{i_*,2}$.\\
The next lowermost non-zero couple on the first column $\gamma_{i,1}$ can then be brought to $(1,0)$ by a base change $M_i$ and allows to kill any $\gamma_{i',1}$ ($i'<i$) by a base change of the form $C_{i,1}$.
\item Assume that $\gamma_{i_0,1}\in Ke_s$ for some ($s\in \{1,2\}$). We apply the following procedure for decreasing $i=i_0,\dots, 1$. At the beginning of each loop $i$, we will have $\gamma_{i,1}\in Ke_s$ for some $s\in \{1,2\}$.\\ 
Given $i$, assume that $\gamma_{i,1}\neq 0$ (otherwise do nothing and go over to the next smaller $i$):
\begin{itemize}
\item If $i\neq i_*$, then a base change of the form $M_i$ allows to set $\gamma_{i,1}$ to $e_s$ for some $s\in \{1,2\}$. \\Then apply $C_{i,1}$ to get $\gamma^s_{i',1}=0$ ($i'<i$). Note that this implies $\gamma_{i',1}\in Ke_{s'}$ ($i'<i$) where $s'$ is the index, such that $\{s,s'\}=\{1,2\}$.
\item If $i=i_*$, then $\gamma_{i,1}\in Ke_2$. A base change of the form \eqref{mult_only} on $u_2$ followed by base change $M_{i_2}$ yields $\gamma_{i,1}=e_2$ without modifying entries in columns $2$ and $3$. \qedhere 
\end{itemize}
\end{itemize}
\end{proof}
\begin{corollary}
The number of $P$-orbits in $\N_{\pp}$ is finite if $\bfp=(5,k)$ for some $k$.
\end{corollary}

\subsection{Main cases}\label{ssect:fin_further}
Relying on the results of the previous subsection, we now indicate how to deduce finiteness in the four maximal cases of diagram \ref{app:fin_case_diag}.

\begin{proof}[Proof of Lemma \ref{lem:fin_cases}]
Let $\bfp=(5,k,1)$. We consider quadruples of the form
$(U,V,f, \varphi)$ with $(U,V,f)$ giving rise to a representation of $\rep^{\inj}(Q_2,I_n)$ of dimension $(5,k+6)$ as in \eqref{fUV} and $\varphi\in V^*$, such that $U\subset \Ker(\varphi)$ and $\Ker(\varphi)$ is $f$-stable. 
The corresponding representation of $\rep^{\inj}(Q_3,I_n)$ is 
\begin{center}
\begin{tikzpicture}[descr/.style={fill=white,inner sep=2.5pt}]
\matrix (m) [matrix of math nodes, row sep=0.05em,
column sep=2em, text height=1.5ex, text depth=0.2ex]
{ U & \Ker(\varphi)& V \\ };
\draw [right hook-latex] (m-1-1) -- (m-1-2);
\draw [right hook-latex] (m-1-2) -- (m-1-3);
\path[->]
(m-1-1) edge [loop above] node{$f_{|U}$} (m-1-1)
(m-1-2) edge [loop above] node{$f_{|\mathrm{Ker}(\varphi)}$} (m-1-2)
(m-1-3) edge [loop above] node{$f$} (m-1-3)
;\end{tikzpicture}\end{center}
We consider these quadruples up to isomorphism, that is, up to an isomorphism in $\rep^{\inj}(Q_2,I_n)$ together with a scalar multiplication on $\varphi$. 
Given such $(U,V,f,\varphi)$, we first consider the triple $(V,f, \varphi)$. It follows from a dual statement to \cite[Lemma~5.3]{BE} that there exists a basis $(v_{i,j})_{i,j}$ of $V$, such that $f$ is in Jordan normal form in this basis 
and such that there exists $i_{\bullet}$ with $\lambda_{i}<\lambda_{i_{\bullet}}$ for any $i>i_{\bullet}$ and $\varphi(v_{i,j})=0$ unless $(i,j)=(i_{\bullet}, \lambda_{i_{\bullet}})$. 

We can carry out the whole reduction procedure of Subsection \ref{ssect_fin_reductions} applied to $(U,V,f)$ while parallely considering how the applied base changes modify $\varphi$. Within the used base changes on $V$ of Tools \ref{tool1}, (namely $M$, $C$ and $B$), only $M_{i}$ can modify $\varphi$ and this modification is just a scalar multiplication. In particular, $\Ker(\varphi)=\langle v_{i,j}\rangle_{(i,j) \neq(i_{\bullet}, \lambda_{i_{\bullet}})}$ and there are at most $\sharp\{\lambda_i\, |\, i\in [\![1,g]\!]\}$ different isomorphism classes $(U,V,f,\varphi)$ for each isomorphism class $(U,V,f)$. Finiteness follows for the case $\bfp=(5,k,1)$.

\medskip

We will now consider the cases $\bfp\in \{ (1,3,k) , (3,1,k) , (1,1,1,k)\}$. For these, we will reduce to a finite number of isomorphism classes by applying base changes on $U$ and on $V$, such that every base change on $V$ is one of the three tools \ref{tool1}. 

\medskip

Let us consider the case $\bfp=(1,l,k)$, $l\leqslant 3$. It amounts to classify quadruples $(u',U,V,f)$ with $(U,V,f)$ as in previous subsections, $\dim U=l+1$, $\dim V=k+l+1$ and $u'\in U\cap \Ker f$. The quadruples should be considered up to isomorphism which is given in the corresponding representation context and arises by base changes on $U$, $V$ and scalar multiplication on $u'$. Using the
notation introduced in \ref{ssect_fin_notation}, we distinguish 2 cases: either a) $\mu_1\leqslant 2$ or b) $\boldmu\in \{(4), (3),(3,1)\}$.\vspace{-1.5 ex}
\begin{itemize}
\item In case a), we carry out the reductions of Section \ref{ssect_fin_reductions} on $(U,V,f)$ in order to get a labeled Young diagram as in Proposition \ref{prop_reduced_cases}. Then, writing $u'=\sum \eta_m u_{m,1}$, we proceed to the following base change on $U$ \[u_m\leftarrow \eta_m u_m \textrm{ for each $m$, such that $\eta_m\neq 0$}.\]
Meanwhile, in the labeled Young diagram, each $e_m$ has been changed to $\eta_m e_m$. Since for each row $i$, there is at most one box $(i,j)$, such that $\gamma_{i,j}\neq0$, base changes on $V$ of the form $M_i$ allow to recover the original labeled Young diagram. We, thus, have our finiteness result: each situation can be reduced to a case where $(U,V,f)$ fulfills the conditions of Proposition \ref{prop_reduced_cases}  and $u'=\sum \epsilon_m u_m$ with $\epsilon_m\in \{0,1\}$ for all $m$.\vspace{-1.5ex}
\item In case b), we first assume that $u'\in f^{\mu_1-1}(U)$. This way, no infinite family can arise, since $\dim f^{\mu_1-1}(U)=1$ in all provided cases.\\ 
The only remaining case to consider is $\boldmu=(3,1)$ and $u'\notin f^{2}(U)$. Choosing $u_1\notin \Ker f^{2}$, $u_{1,t}:=f^{3-t}(u_1)$ and $u_2:=u'$, a basis of $U$ arises in which $f$ is in Jordan normal form. It is then possible to carry out the reductions of Section \ref{ssect_fin_reductions}, using as base changes on $U$ only those of the form \eqref{mult_only} (and accordingly $u'\leftarrow \beta u'$). Indeed, the second and third column of the labeled Young diagram can be treated as in the first step of case b) of Section \ref{ssect_fin_reductions} without the use of base change $E_{1,j}$. The first column can then be reduced as in the second step of case c) of Section \ref{ssect_fin_reductions}.
\end{itemize}


The case $(l,1,k)$ ($l\leqslant 3$) amounts to classify quadruples $(\varphi,U,V,f)$ with $(U,V,f)$ as in previous subsections, $\dim U=l+1$, $\dim V=k+l+1$ and $\varphi\in U^*$, such that $\Ker (\varphi)$ is $f$-stable. The isomorphisms are given by base changes on $U$, $V$ and scalar multiplication on $\varphi$. Since $f$ is nilpotent and $\dim_K U/\Ker\varphi =1$, we know that $f(U)\subset \Ker(\varphi)$ and, given a basis $(u_{m,t})_{m,t}$ in which $f$ is in Jordan normal form, $\varphi$ is completely determined by $(\varphi(u_m))_m$. 

We can then proceed as in the $(1,l,k)$-case. Namely, in case a) we argue in the same way, with base changes on $U$ given by $(u_m\leftarrow (\varphi(u_m))^{-1} u_m)_{\varphi(u_m)\neq 0}$; here $u_m:= u_{m,\mu_m}$. In case b), the generic cases happen with $\boldmu=(3,1)$ and $\Ker(\varphi)\neq \Ker (f_{|U})^2$. In this situation, we choose $u_1\in \Ker(\varphi)\setminus \Ker((f_{|U})^2)$  and obtain $\Ker(\varphi)=Ku_1\oplus Kf(u_1)\oplus Kf^2(u_1)$. Choose $u_2\in \Ker(f)\setminus Kf^2(u_1)$ in order to get a basis in which $f$ is in Jordan normal form. Then the same arguments as in the $(1,3,k)$-case apply.  

\medskip

The last case to consider is $\bfp:=(1,1,1,k)$. It amounts to classify quintuples $(U'',U',U,V,f)$ with $(U,V,f)$ as in previous subsections, $U''\subset U'\subset U$ all $f$-stable and $(\dim U'',\dim U',\dim U,\dim V)=(1,2,3,k+3)$. We have $3$ cases to consider:\vspace{-1ex}
\begin{itemize}
\item If $\boldmu=(3)$, by Proposition \ref{prop_reduced_cases} we can reduce to a finite number of choices for $(U,V,f)$. Then $U''=K u_{1,1}$ and $U'=U''\oplus Ku_{1,2}$ follow without a choice.\vspace{-1.5ex}
\item Assume that $\boldmu=(2,1)$. \vspace{-1.5ex}
\begin{itemize}
\item If $U'=\Ker f_{|U}$, then finiteness follows from the classification of quadruples $(U'',U,V,f)$ which has been achieved in the case $\bfp=(1,2,k)$.
\item If $U'\neq \Ker f_{|U}$, then $f(U')$ is its only $f$-stable subspace of dimension $1$. So $U''=f(U')=f(U)$ and finiteness follows from the classification of quadruples $(U',U,V,f)$ which has been achieved in the case $\bfp=(2,1,k)$.
\end{itemize}
\item Assume that $\boldmu=(1,1,1)$. Choose a basis $(u_1,u_2,u_3)$ of $U$, such that $U''=Ku_3$, $U'=Ku_3\oplus K u_2$. Without modifying these spaces, we can apply base changes of the form 
\begin{equation*}
u_1\leftarrow \omega_1^1 u_1+\omega_1^2 u_2+\omega_1^3 u_3, \quad u_2\leftarrow \omega_2^2 u_2+\omega_2^3 u_3, \quad u_3\leftarrow \omega_3^3 u_3.
\end{equation*}
as introduced in Subsection \ref{ssect_base_changes}.
The reduction of the first column of the labeled Young diagram can then be achieved as in the second step of case a) of Section \ref{ssect_fin_reductions}.
\end{itemize}
We have shown finiteness for $\bfp\in\{ (1,l,k), (l,1,k), (1,1,1,k)\mid l\leq 3, k\in \mathbf{N}\}$. Enlarging these tuples $\bfp$ by an extra $1$ on the right means that we have to add an extra data $\varphi\in V^*$ to the considered representations of $\rep^{\inj}(\Q_p,I_n)$ as explained in case $(5,k,1)$. These cases can be dealt with in the very same way as above, where we deduced the $(5,k,1)$-case from the $(5,k)$-case.
\end{proof}

\section{Generalization to an arbitrary infinite field}\label{sect:field}

Let $K$ be an arbitrary infinite field. Our aim is to explain how to recover our main results under this assumption.

First of all, the notion of representation of a quiver with relations makes sense over any base field. 
It is an easy matter to check that the bijection exhibited in Lemma \ref{bijection} between the set of $P$-orbits in $\N_{\pp}$ and the set of $\GL_{\dfp}$-orbits in $R_{\dfp}^{\inj}(\Q_p,I_n)$ remains valid over $K$. 
The other translations to quiver-representation-theoretic contexts in Subsection \ref{ssect:transl} are also valid.

The proofs of our reduction techniques in Subsection \ref{ssect:reductions} do not depend on $K$ being algebraically closed.

All results on finiteness of Levi-actions in Subsection \ref{ssect:Levi_result} remain true. 
Each infinite cases over the algebraic closure $\overline{K}$ of $K$ gives rise to an infinite family of non-conjugate (over $\overline{K}$) matrices $(x_{t})_{t\in \overline{K}}$, each with entries in $\{0,1,t\}$. 
If two matrices are non-conjugate over $\overline{K}$, then they are surely non-conjugate over $K$. So, considering the family $(x_{t})_{t\in K}$ is enough to show that the considered case is already infinite over $K$.

\subsubsection{Main Theorem}
We now explain how to recover our main theorem. 

Firstly, we can apply the constructions of Section \ref{sect:covering} to $\overline{K}$. The infinite families of representations in the covering quivers constructed in Figures \ref{fig:D4} and \ref{fig:E6} arise from infinite families of isomorphism classes of representations of the affine Dynkin quivers $\widetilde{D_4}$ and $\widetilde{E_6}$. It is a classical fact that representatives can be expressed as representations $M_t\in R_{\dfp}(\widehat Q_{p,n},\widehat I_{p,n})$ ($t\in \overline{K}$) where the transition matrices have entries in $\{0,1,t\}$. 
\begin{figure}
$\!\!\!\!\!\!\!\!\!$\begin{tabular}{c c c }
\scalebox{1}{\begin{tikzpicture}[->,>=stealth',shorten >=1pt,auto,node distance=3cm,
  thick,main node/.style={
  node distance={6ex}, minimum size=2cm,
  font=\sffamily\small\bfseries,minimum size=15mm}]
\node[main node] (1b){$K$};
\node[main node] (2b)[below of=1b]{$K^2$};
\node[main node] (3)[below of=2b]{$K^3$}; 
\node[main node] (2a)[left of=3]{$K^2$};
\node[main node] (1a)[left of=2a]{$K$};
\node[main node] (2c)[below of=3]{$K^2$};
\node[main node] (1c)[below of=2c]{$K$};

\path[<-, draw] (1a) -- node {\scalebox{0.5}{$\!\! a=\begin{pmatrix}1\\0 \end{pmatrix}$}} (2a);
\path[<-, draw] (2a)-- node{\scalebox{0.5}{$\!\! b=\begin{pmatrix}1&0\\0&1\\0&0 \end{pmatrix}$}} (3);
\path[->, draw] (2c) -- node {\scalebox{0.5}{$f=\begin{pmatrix}1&1&0\\0&1&1 \end{pmatrix}$}}(3);
\path[->, draw] (1c)-- node {\scalebox{0.5}{$e=\begin{pmatrix}t&1 \end{pmatrix}$}} (2c);
\path[->,draw] (2b)-- node {\scalebox{0.5}{$c=\begin{pmatrix}0\\1 \end{pmatrix}$}} (1b);
\path[->, draw] (3)--node {\scalebox{0.5}{$d=\begin{pmatrix}0&0\\1&0\\0&1 \end{pmatrix}$}} (2b);
\end{tikzpicture}}

& 

\begin{tikzpicture}[->,>=stealth',shorten >=1pt,auto,
main node/.style={
 node distance={6ex}, 
  font=\sffamily\small\bfseries,minimum size=15mm}]
\node[main node] (1b){$\langle e_{12}\rangle$};
\node[main node] (2b)[below of=1b]{\scalebox{0.7}{$\quad \langle e_{8},e_{11}\rangle$}};
\node[main node] (3)[below of=2b]{\scalebox{0.7}{$\qquad\; \; \;\langle e_{4},e_7,e_{10}\rangle$}}; 
\node[main node] (3p)[below of=3]{\scalebox{0.7}{$\qquad \;\; \langle e_3,e_6,e_9\rangle$}};
\node[main node] (2c)[below of=3p]{\scalebox{0.7}{$\qquad\langle e_2,e_5\rangle$}};
\node[main node] (1c)[below of=2c]{\scalebox{0.7}{$\quad\langle e_1\rangle$}};
\node[main node] (2c2)[left of=2c]{\scalebox{0.7}{$\langle e_2,e_5\rangle\qquad$}};
\node[main node] (1c2)[left of=1c]{\scalebox{0.7}{$\langle e_{1}\rangle\quad$}};
\node[main node] (1a)[left of=3]{\scalebox{0.7}{$\langle e_{4}\rangle\quad$}};
\node[main node] (2a)[left of=3p]{\scalebox{0.7}{$\langle e_{3},e_6\rangle \qquad$}};

\path[<-,draw] 
(1a) edge node {\scalebox{0.6}{$a$}} (2a)
(2a) edge node {\scalebox{0.6}{$b$}} (3p)
(2b) edge node[right] {\scalebox{0.6}{$d$}}(3)
(1b) edge node[right] {\scalebox{0.6}{$c$}} (2b)
(2c2) edge node {\scalebox{0.6}{$e$}} (1c2) 
(2c) edge node[right] {\scalebox{0.6}{$e$}} (1c) 
(2a) edge node{\scalebox{0.6}{$fb$}} (2c2)
(1a) edge node{$\scalebox{0.6}{$\;\varphi^{-1}ba$}$}(3)
(3) edge node[right] {\scalebox{0.5}{$\varphi=\begin{pmatrix}1&0&0\\0&1&0\\0&0&1\end{pmatrix}$}} (3p)
(3p) edge node[right] {\scalebox{0.6}{$f$}} (2c);
\path[-, double,draw] 
(2c2) -- (2c)
(1c2) --(1c);

\end{tikzpicture}

&
\begin{tikzpicture}[->,>=stealth',shorten >=1pt,auto
,main node/.style={
 node distance={6ex}, 
  font=\sffamily\small\bfseries,minimum size=15mm}]
\node[main node] (1b){
$x_t=\scalebox{0.7}{$\left(\begin{array}{c c c c c c|c c c c c c}
0 &t&0&0&1&0&0&0&0&0&0&0\\
0 &0&1&0&0&1&0&0&0&0&0&0\\
0 &0&0&1&0&0&0&0&0&0&0&0\\
0 &0&0&0&0&0&0&0&0&0&0&0\\
0 &0&0&0&0&1&0&0&1&0&0&0\\
0 &0&0&0&0&0&1&0&0&0&0&0\\
\hline
0 &0&0&0&0&0&0&1&0&0&0&0\\
0 &0&0&0&0&0&0&0&0&0&0&0\\
0 &0&0&0&0&0&0&0&0&1&0&0\\
0 &0&0&0&0&0&0&0&0&0&1&0\\
0 &0&0&0&0&0&0&0&0&0&0&1\\
0 &0&0&0&0&0&0&0&0&0&0&0\\
\end{array}\right)$}$};
\node[main node] (2b)[below of=1b]{};
\node[main node] (3)[below of=2b]{}; 
\end{tikzpicture}
\end{tabular}
\caption{An explicit infinite example for $\bv=(6,6)$}\label{fig:expl}
\end{figure}

Through $F_{\lambda}$, this yields infinite families $(x_{t})_{t\in \overline{K}}$ of non-conjugate (over ${\overline{K}}$) nilpotent matrices of $\mathfrak{p}\otimes \overline{K}$, always with entries in $\{0,1,t\}$. 
As before, we conclude that $(x_{t})_{t\in \overline{K}}$ is an infinite family of non-conjugate nilpotent matrices of $\mathfrak p$. 
In Figure \ref{fig:expl}, we illustrate the procedure yielding an explicit family $(x_t)_t$, in the case $\bv=(6,6)$.

Concerning the finite cases dealt with in Section \ref{sect:finite_cases}, only elementary techniques of linear algebra are used. So everything holds over an arbitrary field.

\section{Applications to Hilbert schemes and commuting varieties}\label{sect:hilb_comm}

In this section, we assume for simplicity that $K$ is algebraically closed of characteristic zero. 

A motivation to consider the classification of nilpotent $P$-orbits in $\pp$ comes up in the context of commuting varieties and nested punctual Hilbert schemes \cite{BE}.  

Consider the \textit{nilpotent commuting variety} of $\mathfrak{p}$
\[\mathcal{C}(\N_{\pp}):=\{(x,y)\in \N_{\pp}\times\N_{\pp}~|\, [x,y]=0\},\] 
an important subvariety of the \textit{commuting variety} $\mathcal{C}(\mathfrak{p}):=\{{(x,y)\in \pp\times\pp}|\; [x,y]=0\}$ of $\pp$. 
We refer to \cite{Rich, Pr}, (resp. \cite{Ke, GoRo}, resp. \cite{GoGo,BE})  for ground results on commuting varieties and nilpotent commuting varieties of semisimple algebras (resp. Borel subalgebras, resp. parabolic subalgebras). 

Clearly, $P$ acts diagonally via conjugation on both varieties.
For  $x\in \N_{\pp}$, we define $\pp^x=\{y\in \pp~|~  [x,y]=0\}$ and say that $x$ is   \textit{distinguished} if $\pp^x\cap \mathfrak{sl}_n\subset \N_{\pp}$. A $P$-orbit in $\N_{\pp}$ is said to be \textit{distinguished}, if its elements are distinguished.

\begin{proposition}\label{prop:dimC_dist}
\begin{enumerate}
\item\label{enu:fincase} If $P$ acts finitely on $\N_\pp$, then $\dim \mathcal{C}(\N_{\pp})=\dim \pp-1$. Moreover, the irreducible components of maximal dimension are in one-to-one correspondence with the distinguished orbits in $\N_{\pp}$.
\item\label{enu:infcase} If there are infinitely many distinguished $P$-orbits in $\N_{\pp}$, then $\dim \mathcal{C}(\N_{\pp})\geqslant \dim \pp$.
 \end{enumerate}
\end{proposition}
\begin{proof}
We begin by proving \ref{enu:fincase} and assume that the number of $P$-orbits in $\N_\pp$ is finite. Then we can decompose $\mathcal{C}(\N_{\pp})$ into finitely many disjoint subsets as follows: 
\begin{equation}\label{disj_union}
\mathcal{C}(\N_{\pp})=\bigcup_{P.x\in \N_{\pp}} P.(x, \pp^x\cap\N_{\pp})
\end{equation}
It follows from \cite[(14)]{BE}, that $\dim P.(x, \pp^x\cap\N_{\pp})=\dim \pp-\codim_{\pp^x} (\pp^x\cap\N_{\pp})$. Since $\pp^x=K\cdot \id\oplus (\pp^x\cap \mathfrak{sl}_n)$, we see that $\codim_{\pp^x} (\pp^x\cap\N_{\pp})\geqslant 1$ with equality if and only if $x$ is distinguished. In this case, $P.(x, \pp^x\cap\N_{\pp})=P.(x, \pp^x\cap \mathfrak{sl}_n)$ is irreducible. We still have to see that distinguished elements exist. An example is given by the regular nilpotent element in Jordan normal form.

We show \ref{enu:infcase} and assume that there are infinitely many distinguished $P$-orbits in $\N_{\pp}$. It follows from \eqref{disj_union} that $\mathcal{C}(\N_{\pp})$ contains an infinite union of $(\dim \pp-1)$-dimensional disjoint constructible subvarieties. The result follows.
\end{proof}

\begin{lemma}\label{lm:carac_dist}
Let $x\in \N_{\pp}$ and $M$ be a corresponding representation in $\rep^{\inj}(\Q_p, I_n)(\dfp)$. Then $x$ is distinguished if and only if $M$ is indecomposable.
\end{lemma}
\begin{proof}
Denote by $V_1\subset\dots\subset V_k$ the partial flag of $K^n$ defining $\pp$. That is $\pp(V_i)\subset V_i$ for all $i$. Given a semisimple element $s\in\pp$, we consider the decomposition into $s$-eigenspaces $K^n=\bigoplus_j W_j$ and we have $V_i=\bigoplus_j V_i\cap W_j$ for any $i$. Conversely, from a decomposition $K^n=\bigoplus_j W_j$ such that $V_i=\bigoplus_j V_i\cap W_j$ for all $i$, we can construct a semisimple element $s\in\pp$ such that the $W_j$ are the eigenspaces of $s$.

Thus $x\in \NN_{\pp}$ is not distinguished if and only if there exists a non-trivial decomposition $K^n=\bigoplus_j W_j$ such that $V_i=\bigoplus_j V_i\cap W_j$ for all $i$ with $x\in \bigoplus_j \End(W_j)$. The last statement means that any corresponding representation of $\Q_p$ is decomposable. 
\end{proof}

As a consequence, we see that in the critical cases of Proposition \ref{prop:class_inf}, namely $\bfp\in \{(6,6), (2,2,2), (4,1,4), (1,4,6), (1,4,4,1), (1,2,1,4), (1,2,1,2,1)\}$, we always have a one-parameter family of indecomposables of $\rep^{\inj}(\widehat{Q_p}, \widehat{I_n})$ and hence of $\rep^{\inj}(\Q_p, I_n)$. So $\dim \mathcal{C}(\N_{\pp})\geqslant \dim \pp$ holds true in these cases. It is unclear whether this last property holds whenever $P$ acts on $\N_{\pp}$ with infinitely many orbits (\textit{e.g.} if $\bv=(2,3,2)$).
However, it is sometimes easy to extend the indecomposables of Figures \ref{fig:D4} and \ref{fig:E6} to indecomposables of greater dimension. For instance, Figure \ref{fig:ext66} provides an indecomposable in the case $\bv=(k,k')$ with $k,k'\geqslant 6$. We can therefore state

\begin{figure}
\begin{center}
\scalebox{0.6}{\begin{tikzpicture}[->,>=stealth',shorten >=1pt,auto,
  thick,main node/.style={
 node distance={6ex}, 
  font=\sffamily\small\bfseries,minimum size=15mm}]
\node[main node] (1b){1};
\node[main node] (2b)[below of=1b]{2};
\node[main node] (3)[below of=2b]{3}; 
\node[main node] (3p)[below of=3]{3};
\node[main node] (2c)[below of=3p]{2};
\node[main node] (1c)[below of=2c]{1};
\node[main node] (2c2)[left of=2c]{2};
\node[main node] (1c2)[left of=1c]{1};
\node[main node] (1a)[left of=3]{1};
\node[main node] (2a)[left of=3p]{2};
\node[main node] (1g1)[below of=1c]{$\vdots$};
\node[main node] (1g2)[left of=1g1]{$\vdots$};
\node[main node] (1g3)[below of=1g1]{1};
\node[main node] (1g4)[left of=1g3]{1};
\node[main node] (1b0)[above of=1b]{$\vdots$};
\node[main node] (1b1)[above of=1b0]{1};

\path[<-,draw] 
(1a) edge node {$a$} (2a)
(2a) edge node {$b$} (3p)
(2b) edge node[right] {$d$}(3)
(1b) edge node[right] {$c$} (2b)
(2c2) edge node {$e$} (1c2) 
(2c) edge node[right] {$e$} (1c) 
(2a) edge node{$fb$} (2c2)
(1a) edge node{$ba$}(3)
(3p) edge node[right] {$f$} (2c);

\path[-, double,draw]
(1b1)--(1b0)
(1b0)--(1b) 
(2c2) -- (2c)
(3) -- (3p)
(1c2) --(1c)
(1c2)--(1g2)
(1c)--(1g1)
(1g3)--(1g4)
(1g2)--(1g4)
(1g1)--(1g3);
\end{tikzpicture}}
\caption{$\dff=(k,n)$ with $k, n-k\geqslant 6$}
\label{fig:ext66}
\end{center}
\end{figure}

\begin{proposition}\label{prop:dimC_max_par}
If $P$ is a maximal parabolic, then $\dim \mathcal{C}(\N_{\pp})=\dim \pp-1$ if and only if one of the blocks of $\pp$ is of size at most $5$. Otherwise, $\dim \mathcal{C}(\N_{\pp})\geqslant \dim \pp$.
\end{proposition}

%
%
In correspondence with these commuting varieties, one can study the so-called \textit{nested punctual Hilbert schemes}. These Hilbert schemes were introduced in \cite{Ch1}. We refer to \cite[Definition~3.7]{BE} for a scheme-theoretic definition. Such scheme depends on a non-decreasing sequence $\dff=(d_1,\dots, d_p)$ and we will focus on Hilbert schemes on the plane $\mathbb{A}^2$ which we denote by $\Hil{\dff}$. We recall that, set-theoretically, 
\[\Hil{\dff}=\left\{ z_1\subset z_2\subset \dots \subset z_p\, \left|\; \begin{array}{l}z_i \textrm{ is a subscheme of } \mathbb{A}^2 \textrm{ of length }d_i\\ \textrm{ supported at $(0,0)$}\end{array}\right\}\right..\]
Equivalently, we can consider $\Hil{\dff}$ as the set of sequences of inclusions $I_1\supset\dots \supset I_p$ with $I_i$ an ideal of codimension $d_i$ in $K[X,Y]$ containing $(X,Y)^{\alpha}$ for some $\alpha\in \bN$. 

Let us consider the open subvariety 
 \[ \mathcal{C}^{\mathrm{cyc}}(\N_{\pp}):=\{ (x,y)\in \mathcal{C}(\N_{\pp})| \, \exists v \in K^n \textrm{ s.t. } 
\big \langle x^iy^j.v\big \rangle_{i,j}=K^n\},\]
 of $\mathcal{C}(\N_{\pp})$, that is, the set of couples admitting a cyclic vector. We will use the following result of \cite[Proposition 3.13]{BE}. Recall that, given a parabolic $P$, we denote the block sizes of $\pp$ by $\bv_P=(b_1,\dots, b_p)$ and we define $\dff_P=(d_1,\dots d_p)$ via $d_i=\sum_{j\leqslant i} b_j$.
\begin{proposition}\label{prop:dimHvsdimC}
There is a one-to-one correspondence between the irreducible components of $\mathcal{C}^{\mathrm{cyc}}(\N_{\pp})$ of dimension $m+\dim \pp-n$ and the irreducible components of
$\Hil{\overline{\dff_P}}$ of dimension $m$, where $\overline{\dff_P}:=(n-d_{p-1}, \dots , n-d_1, n)$
\end{proposition}
A deeper connection between related schemes is expressed in \cite[Proposition 3.2]{BE}.
Let us mention that there always exists a component of $\Hil{\dff_P}$ of dimension $n-1$, the so-called \textit{curvilinear component}. Proposition \ref{prop:dimHvsdimC} is used in \cite[Theorem~7.5]{BE} to show that $\Hil{2,n}$ and $\Hil{n-2,n}$ are equidimensional of dimension $n-1$. With the results of the present paper, we are now able to state the following:
\begin{theorem}\label{prop:dimHilb}
\begin{enumerate}
\item If there are only finitely many nilpotent orbits for the action of $P$ on its Lie algebra $\pp$, then $\dim \Hil{\dff_P}=n-1$. In particular, this is the case for $\dff=(k,n)$ with $k\leqslant 5$ or $n-k\leqslant 5$.
\item If $\dff=(k,n)$ with $k\geqslant 6$ and $n-k\geqslant 6$ then $\dim \Hil{\dff}\geqslant n$.
\end{enumerate}
\end{theorem}
\begin{proof} 
Using the transposition ${}^t(\cdot)$ as in Subsection \ref{ssect:reductions}, we see that there is an isomorphism $\mathcal{C}(\N_{\pp})\cong \mathcal{C}(\N_{{}^t\!\pp})$. Since $\mathcal{C}^{cyc}(\N_{{}^t\!\pp})$ is an open subvariety of $\mathcal{C}(\N_{{}^t\!\pp})$, the first statement of the theorem is a consequence of Propositions \ref{prop:dimC_dist} and \ref{prop:dimHvsdimC}.

For the second assertion, for simplicity of notation, we will focus on $\pp$ instead of ${}^t\pp$. We need to show that a component of dimension at least $\dim \pp$ provided by Proposition \ref{prop:dimC_max_par} still appears in $\mathcal{C}^{cyc}(\N_{\pp})$. This is done as follows. From Figure \ref{fig:ext66}, we build a one-parameter family of non-conjugate elements which are distinguished by Lemma \ref{lm:carac_dist}. That is, we let $x_{t}\in \N_{\pp}$ ($t\in K^{\times}$) be defined by its action on the canonical basis $(e_i)_{i\in[\![1,n]\!] }$ of $K^n$ via 
\[\begin{array}{l}e_n\rightarrow \dots\rightarrow e_{k+5}\rightarrow e_{k+4} \rightarrow e_{k+3}\rightarrow e_{k-1}\rightarrow e_{k-5}\rightarrow e_{k-6} \rightarrow \dots\rightarrow e_1\rightarrow 0\\
e_{k+2}\rightarrow e_{k+1}\rightarrow e_{k}\rightarrow e_{k-1}+e_{k-4}\\
e_{k-2}\rightarrow e_{k-3}\rightarrow e_{k-4} \rightarrow t e_{k-5}
\end{array}\]
where an arrow $e_i\rightarrow z$ means that $x_{t}(e_i):=z$. In Jordan form for general $t\in K$, this gives
\[\begin{array}{c@{}c@{} c@{}c@{}c@{}c@{}c@{}c c} e_n&\rightarrow& e_{n-1}&\rightarrow& e_{n-2}&\rightarrow& e_{n-3} &\rightarrow & \dots\\[2ex] 
\begin{array}{c}e_{k+2}-\\(1+t) e_{k+5}\end{array} &\rightarrow &\begin{array}{c}e_{k+1}-\\(1+t) e_{k+4}\end{array}&\rightarrow &\begin{array}{c}e_{k}-\\(1+t) e_{k+3}\end{array}&\rightarrow& \begin{array}{c}e_{k-4}-\\ t e_{k-1}\end{array}&\rightarrow &0\\[2.5ex]
\begin{array}{c}e_{k+1}-\\e_{k-2}-e_{k+4}\end{array}&\rightarrow &\begin{array}{c}e_{k}-\\e_{k-3}-e_{k+3}\end{array}&\rightarrow& 0
\end{array}\]
where the first line is as before (note that $n-3\geqslant k+3$). Define $y_{t}$ in this new basis via 
\[ \begin{array}{l} e_n\rightarrow e_{k+2}-(1+t) e_{k+5}\rightarrow e_{k+1}-e_{k-2}-e_{k+4}\rightarrow \alpha (e_k-(1+t)e_{k+3})\\
e_{n-1}\rightarrow  e_{k+1}-(1+t)e_{k+4}\rightarrow e_{k}-e_{k-3}-e_{k+3}\rightarrow \alpha(e_{k-4}-t e_{k-1})\\
e_{n-2}\rightarrow (e_k-(1+t)e_{k+3})\rightarrow 0\\
e_{n-3}\rightarrow (e_{k-4}-t e_{k-1})\rightarrow 0\\
e_{i}\rightarrow 0 \quad \textrm{ if $n-4 \geq i\geq k+3$ or $i=k-1$ or $k-5 \geq i\geq1$  }
\end{array}\]
with $\alpha:=t$ if $n=k+6$ and $\alpha:=0$ if $n\geqslant k+7$. It is straightforward to check that $y(U)\subset U$, where $U=\langle e_i\rangle_{i\in [\![1,k]\!]}$.
%
Then $(x_{t},y_{t})$ is a nilpotent commuting pair of $\pp$ for general $t$,  which admits $e_n$ as a cyclic vector. Hence, $\bigcup_{t} \mathcal{C}^{cyc}(\N_{\pp})\cap \left(P.(x_{t},\pp^{x_{t}}\cap\N_{\pp})\right)$ is an infinite union of ($\dim \pp-1$)-dimensional disjoint constructible subvarieties of  $\mathcal{C}^{cyc}(\N_{\pp})$. Hence $\dim \mathcal{C}^{cyc}(\N_{\pp'})\geqslant \dim \pp$ and the result follows from Proposition \ref{prop:dimHvsdimC}.
\end{proof}

\begin{remark}
In the first case of the previous theorem, the same proof together with Lemma \ref{lm:carac_dist} yields a more precise result. Namely, the irreducible components of  $\Hil{\dff_P}$ of maximal dimension are in one-to-one correspondence with distinguished orbits $P.x\in \N_{\pp}$ satisfying the following property:
\[  \exists y\in \N_{\pp}\cap \pp^x \textrm{ such that } ({}^tx,{}^ty) \textrm{ admits a cyclic vector.}\]
\end{remark}

\begin{remark}
Since any prime is good for $GL_n$, it is plausible that the results of this section remain true when $K$ is algebraically closed of any characteristic. In particular, note that \cite{BE} considers fields of any characteristic. 
In positive characteristic, one should be careful when defining a distinguished nilpotent element. One can find clues about how to proceed in positive characteristic in \cite[Section~3]{Pr}. 
\end{remark}

In the study of the irreducibility of the (non-nilpotent) commuting varieties $\mathcal{C}(\pp)$, it is crucial to estimate the modality of the action of the group on the cone of nilpotent elements, see \emph{e.g.} the recent paper \cite[Theorem 1.1]{GoGo}. The known examples of reducible commuting variety  \cite[Theorem~1.3 and Section~8]{GoGo} have at least $15$ blocks and arise from the study of the modality of the action of $P$ on $\mathfrak{n}_{\pp}$. In particular, no example can arise with such method if $P$ has less than $6$ blocks.
 
The approach of the present paper might help to find examples of parabolics $P$ with few blocks and a big modality on $\N_{\pp}$. Indeed, using ordinary quiver theory, it is possible to find families of representations with a great number of parameters in $\rep^{\inj}(\widehat Q_p, \widehat I_n)$. For instance, if $\dff=(d_{i,j})_{i,j}$, such that $R^{\inj}_{\dff}(\widehat Q_p, \widehat I_n)\neq \emptyset$ and the irreducible components of $R_{\dff}(\widehat Q_p, \widehat I_n)$ are all of dimension at least $m+\dim \GL_{\dff}$ with $m\geqslant 0$, then there exists a $m+1$-parameter family in $R_{\dff}^{\inj}(\widehat Q_p, \widehat I_n)$. This family translates to a $m+1$-parameter family of $R_{\dff'}^{\inj}( Q_p,  I_n)$ via Proposition \ref{prop:F_lambda} and, for the corresponding parabolic $\pp$, the modality of $\N_{\pp}$ is at least $m+1$.

As an example, consider the following dimension vectors in $\widehat Q_2$ and $\widehat Q_5$ 
\[ 
\scalebox{0.8}{
$\begin{array}{ c c}
0&20\\
0&40\\
0&60\\
20&80 \\
60&80\\
60&60\\
40&40\\
20&20
\end{array}$}\qquad\textrm{ and }\qquad
\scalebox{0.6}{$
\begin{array}{ c c c c c}
0  &0 & 0 & 0 & 1\\
0 & 0 &  0 & 0 & 2\\
0 & 0 & 0&  1 & 4\\
0&  0&  1&  4&  7\\
0&  1&  4&  9&  11\\
0 &3 &9 &12& 12\\
2 & 7&  10& 11 &11\\
3&  6 & 7 & 7&  7\\
3  &4 & 4&  4&  4\\
2 & 2 & 2  &2 & 2\\
1 & 1 & 1 & 1 & 1
\end{array}$}\]
Here, $(m+1,n-1)$ is equal to $(401,399)$ in the first case and $(64,61)$ in the second case. It follows from \cite{GoGo} that the corresponding parabolic subalgebras $\pp$ of block sizes $\bv_{P}=(200,400)$ or $(11,13,14,13,11)$ have reducible commuting varieties $\mathcal{C}(\pp)$. 

\appendix
\section{Auslander-Reiten quivers}
\subsection[The Auslander-Reiten quiver of A(2,2)]{The Auslander-Reiten quiver of $\widehat{\A}(2,2)_n$}\label{app:p2x2}
\setlength{\unitlength}{0.53mm}
\begin{picture}(10,30)(10,-6)
\multiput(72,-30)(0,2){25}{\line(0,1){1}}
  \put(13,-12){\scalebox{0.43}{$\begin{array}{l}
0 1\end{array}^{\!\!\!\!(0)}$}}
 \put(20,-6){$\nearrow$}\put(20,-18){$\searrow$}
    \put(26,-2){\scalebox{0.43}{$\begin{array}{l} 
1 1  
  \end{array}^{\!\!\!\!(0)}$}}\put(32,-6){$\searrow$}
  \put(26,-22){\scalebox{0.43}{$\begin{array}{l}
0 1\\
0 1  \end{array}^{\!\!\!\!(0)}$}}\put(32,-18){$\nearrow$}
  \put(36.5,-12){\scalebox{0.43}{$\begin{array}{l} 
0 1 \\ 
1 1  \end{array}^{\!\!\!\!(0)}$}}\put(44,-6){$\nearrow$}\put(44,-18){$\searrow$}\put(43,-12){$\rightarrow$}
  \put(49,0){\scalebox{0.43}{$\begin{array}{l}
0 1  \end{array}^{\!\!\!\!(1)}$}}\put(56,-6){$\searrow$}
  \put(49,-12){\scalebox{0.43}{$\begin{array}{l}
1 1 \\ 
1 1  \end{array}^{\!\!\!\!(0)}$}}\put(55,-12){$\rightarrow$}
  \put(49,-22){\scalebox{0.43}{$\begin{array}{l}
1 0  \end{array}^{\!\!\!\!(0)}$}}\put(56,-18){$\nearrow$}\put(56,6){$\nearrow$}
  \put(60.5,-12){\scalebox{0.43}{$\begin{array}{l}
1 1 \\ 
1 0  \end{array}^{\!\!\!\!(0)}$}}\put(68,-6){$\nearrow$}\put(68,-18){$\searrow$}
\put(60.5,12){\scalebox{0.43}{$\begin{array}{l}
0 1 \\ 
0 1  \end{array}^{\!\!\!\!(1)}$}}\put(68,6){$\searrow$}
  \put(73.5,0){\scalebox{0.43}{$\begin{array}{l}
0 1 \\ 
1 1 \\ 
1 0  \end{array}^{\!\!\!\!(0)}$}}\put(81.5,6){$\nearrow$}\put(81.5,-6){$\searrow$}  
  \put(73,-24){\scalebox{0.43}{$\begin{array}{l}
1 1  \end{array}^{\!\!\!\!(1)}$}}\put(81.5,-18){$\nearrow$}
\end{picture}

\begin{picture}(10,30)(-100,-45)
\multiput(23,-40)(0,2){25}{\line(0,1){1}}
 \put(8,-30){$\searrow$}
\put(8,-6){$\searrow$}
 \put(8,-18){$\nearrow$}
\put(7,-12){$\rightarrow$}
\put(13,-36){\scalebox{0.43}{$\begin{array}{l}
0 1\\ 
0 1\end{array}^{\!\!\!\!(n-1)}$}} \put(20,-30){$\nearrow$}
  \put(13,-12){\scalebox{0.43}{$\begin{array}{l}
1 1\\ 
1 0\end{array}^{\!\!\!\!(n-2)}$}}
 \put(20,-6){$\nearrow$}\put(20,-18){$\searrow$}
    \put(25,-22){\scalebox{0.43}{$\begin{array}{l}
0 1 \\ 
1 1 \\ 
1 0  \end{array}^{\!\!\!\!(n-2)}$}}\put(32,-6){$\searrow$}\put(32,-30){$\searrow$}
  \put(24,0){\scalebox{0.43}{$\begin{array}{l}
1 1  \end{array}^{\!\!\!\!(n-1)}$}}\put(32,-18){$\nearrow$}
   \put(37,-36){\scalebox{0.43}{$\begin{array}{l}
1 0 \\ 
1 0  \end{array}^{\!\!\!\!(n-2)}$}}\put(44,-30){$\nearrow$}
  \put(36.5,-12){\scalebox{0.43}{$\begin{array}{l} 
0 1 \\ 
1 1  \end{array}^{\!\!\!\!(n-1)}$}}\put(44,-6){$\nearrow$}\put(44,-18){$\searrow$}\put(43,-12){$\rightarrow$}
  \put(48,0){\scalebox{0.43}{$\begin{array}{l}
0 1  \end{array}^{\!\!\!\!(n)}$}}\put(56,-6){$\searrow$}
  \put(48,-12){\scalebox{0.43}{$\begin{array}{l}
1 1 \\ 
1 1  \end{array}^{\!\!\!\!(n-1)}$}}\put(55,-12){$\rightarrow$}
  \put(48,-24){\scalebox{0.43}{$\begin{array}{l}
1 0  \end{array}^{\!\!\!\!(n-1)}$}}\put(56,-18){$\nearrow$}
  \put(60.5,-12){\scalebox{0.43}{$\begin{array}{l}
1 1 \\ 
1 0  \end{array}^{\!\!\!\!(n-1)}$}}\put(68,-6){$\nearrow$}\put(68,-18){$\searrow$}
  \put(73,0){\scalebox{0.43}{$\begin{array}{l}
1 0\\
1 0  \end{array}^{\!\!\!\!(n-1)}$}}\put(82,-6){$\searrow$}
  \put(73.5,-24){\scalebox{0.43}{$\begin{array}{l}
1 1  
  \end{array}^{\!\!\!\!(n)}$}}\put(82,-18){$\nearrow$}
  \put(86,-12){\scalebox{0.43}{$\begin{array}{l}
1 0 
  \end{array}^{\!\!\!\!(n)}$}}
\end{picture}

\begin{picture}(10,30)(-40,-10)
\multiput(23,-40)(0,2){30}{\line(0,1){1}}
\multiput(72,-40)(0,2){30}{\line(0,1){1}}
 \put(8,6){$\nearrow$}
\put(8,-6){$\searrow$}
 \put(8,-18){$\nearrow$}
\put(7,-12){$\rightarrow$}
\put(13,12){\scalebox{0.43}{$\begin{array}{l}
0 1\\ 
0 1\end{array}^{\!\!\!\!(i+1)}$}} \put(20,6){$\searrow$}
  \put(13,-12){\scalebox{0.43}{$\begin{array}{l}
1 1\\ 
1 0\end{array}^{\!\!\!\!(i)}$}}
 \put(20,-6){$\nearrow$}\put(20,-18){$\searrow$}
    \put(25,0){\scalebox{0.43}{$\begin{array}{l}
0 1 \\ 
1 1 \\ 
1 0  \end{array}^{\!\!\!\!(i)}$}}\put(32,-6){$\searrow$}\put(32,6){$\nearrow$}
  \put(24,-24){\scalebox{0.43}{$\begin{array}{l}
1 1  \end{array}^{\!\!\!\!(i+1)}$}}\put(32,-18){$\nearrow$}
   \put(37,12){\scalebox{0.43}{$\begin{array}{l}
1 0 \\ 
1 0  \end{array}^{\!\!\!\!(i)}$}}\put(44,6){$\searrow$}
  \put(36.5,-12){\scalebox{0.43}{$\begin{array}{l} 
0 1 \\ 
1 1  \end{array}^{\!\!\!\!(i+1)}$}}\put(44,-6){$\nearrow$}\put(44,-18){$\searrow$}\put(43,-12){$\rightarrow$}
  \put(48,0){\scalebox{0.43}{$\begin{array}{l}
1 0  \end{array}^{\!\!\!\!(i+1)}$}}\put(56,-6){$\searrow$}
  \put(48,-12){\scalebox{0.43}{$\begin{array}{l}
1 1 \\ 
1 1  \end{array}^{\!\!\!\!(i+1)}$}}\put(55,-12){$\rightarrow$}
  \put(48,-24){\scalebox{0.43}{$\begin{array}{l}
0 1  \end{array}^{\!\!\!\!(i+2)}$}}\put(56,-18){$\nearrow$}\put(56,-30){$\searrow$}
  \put(60.5,-12){\scalebox{0.43}{$\begin{array}{l}
1 1 \\ 
1 0  \end{array}^{\!\!\!\!(i+1)}$}}\put(68,-6){$\nearrow$}\put(68,-18){$\searrow$}
\put(60,-36){\scalebox{0.43}{$\begin{array}{l}
0 1 \\ 
0 1  \end{array}^{\!\!\!\!(i+2)}$}}\put(68,-30){$\nearrow$}
  \put(73,0){\scalebox{0.43}{$\begin{array}{l}
1 1  \end{array}^{\!\!\!\!(i+2)}$}}\put(82,-6){$\searrow$}
  \put(73.5,-24){\scalebox{0.43}{$\begin{array}{l}
0 1 \\ 
1 1 \\ 
1 0  \end{array}^{\!\!\!\!(i+1)}$}}\put(82,-18){$\nearrow$}\put(82,-30){$\searrow$}
\end{picture}

\rule{3mm}{0mm}\vspace{1cm} \newpage
\subsection[The Auslander-Reiten quiver of A(3,2)]{The Auslander-Reiten quiver of $\widehat{\A}(3,2)_n$}\label{app:p3x2}

\setlength{\unitlength}{0.68mm}
\begin{picture}(44,20)(-55,6)
\multiput(-20,-62)(0,2){44}{\line(0,1){1}}
\multiput(80,-62)(0,2){44}{\line(0,1){1}}
\put(-33,6){$\nearrow$}
\put(-33,-6){$\searrow$}\put(-33,-18){$\nearrow$}\put(-33,-12){$\rightarrow$}
\put(-33,-30){$\searrow$}
    \put(-29,12){\scalebox{0.4}{$\begin{array}{l}
0 1 1 \\ 
0 1 1 \end{array}^{\!\!\!\!(i+1)}$}}\put(-23,6){$\searrow$}
  \put(-29,-12){\scalebox{0.4}{$\begin{array}{l}
0 0 1 \\ 
1 2 1 \\ 
1 1 0  \end{array}^{\!\!\!\!(i)}$}}\put(-23,-6){$\nearrow$}\put(-23,-18){$\searrow$}\put(-23,-12){$\rightarrow$}
\put(-29,-36){\scalebox{0.4}{$\begin{array}{l}
1 1 1 \\ 
1 0 0  \end{array}^{\!\!\!\!(i)}$}}\put(-23,-30){$\nearrow$}\put(-23,-42){$\searrow$}
  \put(-19,0){\scalebox{0.4}{$\begin{array}{l}
0 1 1 \\ 
1 2 1 \\ 
1 1 0  \end{array}^{\!\!\!\!(i)}$}}\put(-13,6){$\nearrow$}\put(-13,-6){$\searrow$}
  \put(-19.5,-12){\scalebox{0.4}{$\begin{array}{l}
0 1 0  \end{array}^{\!\!\!\!(i+1)}$}}\put(-13,-12){$\rightarrow$}
  \put(-19,-24){\scalebox{0.4}{$\begin{array}{l}
0 0 1 \\ 
1 1 1 \\ 
1 0 0  \end{array}^{\!\!\!\!(i)}$}}\put(-13,-18){$\nearrow$}\put(-13,-30){$\searrow$}
\put(-19,-48){\scalebox{0.4}{$\begin{array}{l}
1 1 1  \end{array}^{\!\!\!\!(i+1)}$}}\put(-13,-42){$\nearrow$}
  \put(-9,12){\scalebox{0.4}{$\begin{array}{l}
1 1 0 \\ 
1 1 0  \end{array}^{\!\!\!\!(i)}$}}
  \put(-9,-12){\scalebox{0.4}{$\begin{array}{l}
0 1 1 \\ 
1 2 1 \\ 
1 0 0  \end{array}^{\!\!\!\!(i)}$}}
  \put(-9,-36){\scalebox{0.4}{$\begin{array}{l}
0 0 1 \\ 
1 1 1  \end{array}^{\!\!\!\!(i+1)}$}}\put(-3,-42){$\searrow$}
 \put(-3,6){$\searrow$}
\put(-3,-6){$\nearrow$}\put(-3,-18){$\searrow$}\put(-3,-12){$\rightarrow$}
\put(-3,-30){$\nearrow$}
  \put(1,0){\scalebox{0.4}{$\begin{array}{l}
1 1 0 \\ 
1 0 0  \end{array}^{\!\!\!\!(i)}$}}\put(7,-6){$\searrow$}
  \put(1,-12){\scalebox{0.4}{$\begin{array}{l}
0 1 1 \\ 
1 1 1 \\ 
1 0 0  \end{array}^{\!\!\!\!(i)}$}}\put(7,-12){$\rightarrow$}
  \put(1,-24){\scalebox{0.4}{$\begin{array}{l}
0 1 1 \\ 
1 2 1  \end{array}^{\!\!\!\!(i+1)}$}}\put(7,-18){$\nearrow$}\put(7,-30){$\searrow$}
 \put(1,-48){\scalebox{0.4}{$\begin{array}{l}
0 0 1  \end{array}^{\!\!\!\!(i+2)}$}}\put(7,-42){$\nearrow$}\put(7,-54){$\searrow$}
  \put(11,-12){\scalebox{0.4}{$\begin{array}{l}
0 1 1 \\ 
2 2 1 \\ 
1 0 0  \end{array}^{\!\!\!\!(i)}$}}
  \put(11,-36){\scalebox{0.4}{$\begin{array}{l}
0 1 1 \\ 
0 1 0  \end{array}^{\!\!\!\!(i+1)}$}}
  \put(11,-60){\scalebox{0.4}{$\begin{array}{l}
0 0 1 \\ 
0 0 1  \end{array}^{\!\!\!\!(i+2)}$}}
 \put(17,-6){$\nearrow$}\put(17,-18){$\searrow$}\put(17,-12){$\rightarrow$}
  \put(17,-30){$\nearrow$}\put(17,-42){$\searrow$}
 \put(17,-54){$\nearrow$}
  \put(21,0){\scalebox{0.4}{$\begin{array}{l}
0 1 1 \\ 
1 1 1  \end{array}^{\!\!\!\!(i+1)}$}}\put(27,6){$\nearrow$}\put(27,-6){$\searrow$}
  \put(20.5,-12){\scalebox{0.4}{$\begin{array}{l}
1 1 0  \end{array}^{\!\!\!\!(i+1)}$}}\put(27,-12){$\rightarrow$}
  \put(21,-24){\scalebox{0.4}{$\begin{array}{l}
0 1 1 \\ 
1 1 0 \\ 
1 0 0  \end{array}^{\!\!\!\!(i)}$}}\put(27,-18){$\nearrow$}\put(27,-30){$\searrow$}
  \put(21,-48){\scalebox{0.4}{$\begin{array}{l}
0 0 1 \\ 
0 1 1 \\ 
0 1 0  \end{array}^{\!\!\!\!(i+1)}$}}\put(27,-42){$\nearrow$}\put(27,-54){$\searrow$}
   \put(31,12){\scalebox{0.4}{$\begin{array}{l}
1 1 1 \\ 
1 1 1  \end{array}^{\!\!\!\!(i+1)}$}}\put(37,6){$\searrow$}
  \put(31,-12){\scalebox{0.4}{$\begin{array}{l}
0 1 1 \\ 
1 1 0  \end{array}^{\!\!\!\!(i+1)}$}}\put(37,-6){$\nearrow$}\put(37,-18){$\searrow$}\put(37,-12){$\rightarrow$}
  \put(31,-36){\scalebox{0.4}{$\begin{array}{l}
0 0 1 \\ 
0 1 1 \\ 
1 1 0 \\ 
1 0 0  \end{array}^{\!\!\!\!(i)}$}}\put(37,-30){$\nearrow$}\put(37,-42){$\searrow$}
  \put(31,-60){\scalebox{0.4}{$\begin{array}{l}
0 1 0 \\ 
0 1 0  \end{array}^{\!\!\!\!(i+1)}$}}\put(37,-54){$\nearrow$}
  \put(41,0){\scalebox{0.4}{$\begin{array}{l}
1 1 1 \\ 
1 1 0  \end{array}^{\!\!\!\!(i+1)}$}}\put(47,-6){$\searrow$}
  \put(40.5,-12){\scalebox{0.4}{$\begin{array}{l}
0 1 1  \end{array}^{\!\!\!\!(i+2)}$}}\put(47,-12){$\rightarrow$}
  \put(41,-24){\scalebox{0.4}{$\begin{array}{l}
0 0 1 \\ 
0 1 1 \\ 
1 1 0  \end{array}^{\!\!\!\!(i+1)}$}}\put(47,-18){$\nearrow$}\put(47,-30){$\searrow$}
  \put(41,-48){\scalebox{0.4}{$\begin{array}{l}
0 1 0 \\ 
1 1 0 \\ 
1 0 0  \end{array}^{\!\!\!\!(i)}$}}\put(47,-42){$\nearrow$}\put(47,-54){$\searrow$}
  \put(51,-12){\scalebox{0.4}{$\begin{array}{l}
0 0 1 \\ 
1 2 2 \\ 
1 1 0  \end{array}^{\!\!\!\!(i+1)}$}}\put(57,-6){$\nearrow$}\put(57,-18){$\searrow$}\put(57,-12){$\rightarrow$}
\put(51,-36){\scalebox{0.4}{$\begin{array}{l}
0 1 0 \\ 
1 1 0  \end{array}^{\!\!\!\!(i+1)}$}}\put(57,-30){$\nearrow$}\put(57,-42){$\searrow$}
  \put(51,-60){\scalebox{0.4}{$\begin{array}{l}
1 0 0 \\ 
1 0 0  \end{array}^{\!\!\!\!(i)}$}}\put(57,-54){$\nearrow$}
  \put(61,0){\scalebox{0.4}{$\begin{array}{l}
0 0 1 \\ 
0 1 1  \end{array}^{\!\!\!\!(i+2)}$}}\put(67,6){$\nearrow$}\put(67,-6){$\searrow$}
  \put(61,-12){\scalebox{0.4}{$\begin{array}{l}
0 0 1 \\ 
1 1 1 \\ 
1 1 0  \end{array}^{\!\!\!\!(i+1)}$}}\put(67,-12){$\rightarrow$}
  \put(61,-24){\scalebox{0.4}{$\begin{array}{l}
1 2 1 \\ 
1 1 0  \end{array}^{\!\!\!\!(i+1)}$}}\put(67,-18){$\nearrow$}\put(67,-30){$\searrow$}
  \put(61,-48){\scalebox{0.4}{$\begin{array}{l}
1 0 0  \end{array}^{\!\!\!\!(i+1)}$}}\put(67,-42){$\nearrow$}
   \put(71,12){\scalebox{0.4}{$\begin{array}{l}
0 1 1 \\ 
0 1 1  \end{array}^{\!\!\!\!(i+2)}$}}\put(77,6){$\searrow$}
  \put(71,-12){\scalebox{0.4}{$\begin{array}{l}
0 0 1 \\ 
1 2 1 \\ 
1 1 0  \end{array}^{\!\!\!\!(i+1)}$}}\put(77,-6){$\nearrow$}\put(77,-18){$\searrow$}\put(77,-12){$\rightarrow$}
\put(71,-36){\scalebox{0.4}{$\begin{array}{l}
1 1 1 \\ 
1 0 0  \end{array}^{\!\!\!\!(i+1)}$}}\put(77,-30){$\nearrow$}\put(77,-42){$\searrow$}
  \put(81,0){\scalebox{0.4}{$\begin{array}{l}
0 1 1 \\ 
1 2 1 \\ 
1 1 0  \end{array}^{\!\!\!\!(i+1)}$}}\put(87,6){$\nearrow$}\put(87,-6){$\searrow$}
  \put(80.5,-12){\scalebox{0.4}{$\begin{array}{l}
0 1 0  \end{array}^{\!\!\!\!(i+2)}$}}\put(87,-12){$\rightarrow$}
  \put(81,-24){\scalebox{0.4}{$\begin{array}{l}
0 0 1 \\ 
1 1 1 \\ 
1 0 0  \end{array}^{\!\!\!\!(i+1)}$}}\put(87,-18){$\nearrow$}\put(87,-30){$\searrow$}
\put(81,-48){\scalebox{0.4}{$\begin{array}{l}
1 1 1  \end{array}^{\!\!\!\!(i+2)}$}}\put(87,-42){$\nearrow$}
\end{picture}\label{ark33}\rule{3mm}{0mm}
 \vspace{5cm}


\subsection[The Auslander-Reiten quiver of A(2,3)]{The Auslander-Reiten quiver of $\widehat{\A}(2,3)_n$}\label{app:p2x3}

\setlength{\unitlength}{0.58mm}
\begin{picture}(44,15)(-55,6)
\multiput(-11,-75)(0,2){46}{\line(0,1){1}}
\multiput(69,-75)(0,2){46}{\line(0,1){1}}

   \put(-23,-18){$\searrow$}\put(-23,-6){$\nearrow$}
  \put(-23,-30){$\nearrow$}
\put(-23,-54){$\nearrow$}\put(-23,-18.2){\vector(1,-4){5.5}}
  \put(-18.5,0){\scalebox{0.4}{$\begin{array}{l}
0 1 \\ 
0 1 \\
1 1 \end{array}^{\!\!\!\!(i+1)}$}}\put(-13,-6){$\searrow$}\put(-13,-18){$\nearrow$}\put(-13,-30){$\searrow$}
  \put(-18.5,-24){\scalebox{0.4}{$\begin{array}{l}
1 1 \\ 
2 1 \\
1 0 \end{array}^{\!\!\!\!(i)}$}}\put(-13,-54){$\searrow$}\put(-13,-40){\vector(1,4){5.5}}
\put(-18.5,-48){\scalebox{0.4}{$\begin{array}{l}
0 1 \\ 
1 2 \\
1 1 \\
1 0 \end{array}^{\!\!\!\!(i)}$}}
  \put(-8.5,-12){\scalebox{0.4}{$\begin{array}{l}
0 1 \\ 
1 2 \\
2 1 \\
1 0 \end{array}^{\!\!\!\!(i)}$}}
  \put(-9,-36){\scalebox{0.4}{$\begin{array}{l}
1 1 \\ 
1 1 \end{array}^{\!\!\!\!(i+1)}$}}\put(-3,-18.2){\vector(1,-4){5.5}}
  \put(-8.5,-60){\scalebox{0.4}{$\begin{array}{l}
0 1 \\ 
1 1 \\
1 1 \\
1 0 \end{array}^{\!\!\!\!(i)}$}}\put(-3,-54){$\nearrow$}

\put(-3,-6){$\nearrow$}\put(-3,-18){$\searrow$}
\put(-3,-30){$\nearrow$} 
  \put(1.5,0){\scalebox{0.4}{$\begin{array}{l}
1 1 \\ 
1 0 \\
1 0 \end{array}^{\!\!\!\!(i)}$}}\put(7,-6){$\searrow$}
  \put(1.3,-24.5){\scalebox{0.4}{$\begin{array}{l}
0 1 \\ 
1 2 \\
1 1 \end{array}^{\!\!\!\!(i+1)}$}}\put(7,-18){$\nearrow$}\put(7,-30){$\searrow$}\put(7,-40){\vector(1,4){5.5}}
 \put(1.5,-48){\scalebox{0.4}{$\begin{array}{l}
0 1 \\ 
1 1 \\
2 1 \\
1 0 \end{array}^{\!\!\!\!(i)}$}}\put(7,-54){$\searrow$}
  \put(11.5,-12){\scalebox{0.4}{$\begin{array}{l}
0 1 \\ 
2 2 \\
2 1 \\
1 0 \end{array}^{\!\!\!\!(i)}$}}
  \put(11,-36){\scalebox{0.4}{$\begin{array}{l}
0 1 \\ 
0 1 \end{array}^{\!\!\!\!(i+2)}$}}\put(17,-43){\vector(1,-4){5.5}}
  \put(11,-60){\scalebox{0.4}{$\begin{array}{l}
1 0 \end{array}^{\!\!\!\!(i+1)}$}}
 \put(17,-6){$\nearrow$}\put(17,-18){$\searrow$}\put(17,-18.2){\vector(1,-4){5.5}}
  \put(17,-30){$\nearrow$}
 \put(17,-54){$\nearrow$}
  \put(21,-0.5){\scalebox{0.4}{$\begin{array}{l}
0 1 \\ 
1 1 \\
1 1 \end{array}^{\!\!\!\!(i+1)}$}}\put(27,6){$\nearrow$}\put(27,-6){$\searrow$}
  \put(21.5,-24){\scalebox{0.4}{$\begin{array}{l}
0 1 \\ 
1 1 \\
1 0 \\
1 0 \end{array}^{\!\!\!\!(i)}$}}
  \put(21,-48){\scalebox{0.4}{$\begin{array}{l}
1 1 \\ 
1 0 \end{array}^{\!\!\!\!(i+1)}$}}\put(27,-18){$\nearrow$}\put(27,-30){$\searrow$}\put(27,-40){\vector(1,4){5.5}}\put(27,-64.5){\vector(1,4){5.5}}
  \put(21.5,-72){\scalebox{0.4}{$\begin{array}{l}
0 1 \\ 
0 1 \\
0 1 \end{array}^{\!\!\!\!(i+2)}$}}\put(27,-54){$\searrow$}
   \put(31.5,12){\scalebox{0.4}{$\begin{array}{l}
1 1 \\ 
1 1 \\
1 1 \end{array}^{\!\!\!\!(i+1)}$}}\put(37,6){$\searrow$}\put(37,-18.2){\vector(1,-4){5.5}}
  \put(31,-12.5){\scalebox{0.4}{$\begin{array}{l}
0 1 \\ 
1 1 \\
1 0 \end{array}^{\!\!\!\!(i+1)}$}}\put(37,-6){$\nearrow$}\put(37,-18){$\searrow$}
  \put(31.5,-36){\scalebox{0.4}{$\begin{array}{l}
0 1 \\ 
0 1 \\
1 1 \\
1 0 \\
1 0 \end{array}^{\!\!\!\!(i)}$}}\put(37,-30){$\nearrow$}\put(37,-43){\vector(1,-4){5.5}}
  \put(31,-60){\scalebox{0.4}{$\begin{array}{l}
1 1 \end{array}^{\!\!\!\!(i+2)}$}}\put(37,-54){$\nearrow$}
  \put(41.5,0){\scalebox{0.4}{$\begin{array}{l}
1 1 \\
1 1 \\
1 0 \end{array}^{\!\!\!\!(i+1)}$}}\put(47,-6){$\searrow$}
  \put(41.3,-23){\scalebox{0.4}{$\begin{array}{l}
0 1 \\ 
0 1 \\
1 1 \\
1 0 \end{array}^{\!\!\!\!(i+1)}$}}
\put(47,-40){\vector(1,4){5.5}}
\put(47,-64.5){\vector(1,4){5.5}}
  \put(41.5,-48){\scalebox{0.4}{$\begin{array}{l}
0 1 \\ 
1 1 \end{array}^{\!\!\!\!(i+2)}$}}\put(47,-18){$\nearrow$}\put(47,-30){$\searrow$}
  \put(41.5,-72){\scalebox{0.4}{$\begin{array}{l}
1 0 \\ 
1 0 \\
1 0 \end{array}^{\!\!\!\!(i)}$}}\put(47,-54){$\searrow$}
  \put(51.5,-11){\scalebox{0.4}{$\begin{array}{l}
0 1 \\ 
1 2 \\
2 2 \\
1 0 \end{array}^{\!\!\!\!(i+1)}$}}\put(57,-6){$\nearrow$}\put(57,-18){$\searrow$}
\put(51,-36){\scalebox{0.4}{$\begin{array}{l}
1 0 \\
1 0 \end{array}^{\!\!\!\!(i+1)}$}}\put(57,-30){$\nearrow$}
\put(57,-18.2){\vector(1,-4){5.5}}
  \put(51,-60){\scalebox{0.4}{$\begin{array}{l}
0 1 \end{array}^{\!\!\!\!(i+3)}$}}\put(57,-54){$\nearrow$}
  \put(61.5,0){\scalebox{0.4}{$\begin{array}{l}
0 1 \\ 
0 1 \\
1 1 \end{array}^{\!\!\!\!(i+2)}$}}\put(67,-6){$\searrow$}
  \put(60.9,-24.5){\scalebox{0.4}{$\begin{array}{l}
1 1 \\ 
2 1 \\
1 0 \end{array}^{\!\!\!\!(i+1)}$}}\put(67,-18){$\nearrow$}\put(67,-30){$\searrow$}\put(67,-40){\vector(1,4){5.5}}
  \put(61.5,-48){\scalebox{0.4}{$\begin{array}{l}
0 1 \\ 
1 2 \\
1 1 \\
1 0 \end{array}^{\!\!\!\!(i+1)}$}}\put(67,-54){$\searrow$}
   \put(71.5,-11){\scalebox{0.4}{$\begin{array}{l}
0 1 \\
1 2 \\ 
2 1 \\
1 0 \end{array}^{\!\!\!\!(i+1)}$}}\put(77,-18){$\searrow$}\put(77,-6){$\nearrow$}
  \put(71,-36){\scalebox{0.4}{$\begin{array}{l}
1 1 \\ 
1 1 \end{array}^{\!\!\!\!(i+2)}$}}\put(77,-30){$\nearrow$}\put(77,-18){$\searrow$}\put(77,-18.2){\vector(1,-4){5.5}}
\put(71.5,-59){\scalebox{0.4}{$\begin{array}{l}
0 1 \\ 
1 1 \\
1 1 \\
1 0 \end{array}^{\!\!\!\!(i+1)}$}}
\put(77.4,-54){$\nearrow$}

\end{picture}\label{ark23}\rule{3mm}{0mm} \vspace{6cm}
\section{Case diagrams}\label{app:case_diagrams}
\subsection{Infinite cases}\label{app:inf_case_diag}

\begin{center}
\scalebox{0.5}{
$\!\!\!\!\!\!\!\!$\begin{tikzpicture}[->,>=stealth',shorten >=1pt,auto,node distance=3cm,
  thick,main node/.style={
draw, ellipse, node distance={20ex}, minimum size=2cm,
  font=\sffamily\small\bfseries,minimum size=15mm},descr node/.style={fill=blue!20,draw},first node/.style={fill=blue!20,draw,
draw, ellipse, node distance={20ex}, minimum size=2cm,
  font=\sffamily\small\bfseries,minimum size=15mm},
conj node/.style={fill=green!20,draw,
draw, ellipse, node distance={20ex}, minimum size=2cm,
  font=\sffamily\small\bfseries,minimum size=15mm},
node/.style={}]
\tikzstyle{split} = [draw, {-angle 90 reversed},]
\tikzstyle{plus} = [draw,{-o}]
  \node[first node] (222) {(2,2,2)};
  \node[descr node] (3bl)[left of =222] {3 blocks};
  \node[descr node] (4bl)[above=4cm of 3bl] {4 blocks};
  \node[descr node] (5bl)[above=4.5cm of 4bl] {5 blocks};
  \node[descr node] (6bl)[above=1.5cm of 5bl] {$\geq$ 6 blocks};
  \node[descr node] (2bl)[below=2cm of 3bl] {$\leq$ 2 blocks};

 \node[main node] (111111)[right of=6bl] {(1,1,1,1,1,1)};
  \node[main node] (6B) [right of=111111] {\mbox{\begin{tabular}{c} at least \\6 blocks \end{tabular}}};
 \node[main node] (222B)[above right=0.1 and 3.5 of 222] {\mbox{\begin{tabular}{c}at least \\ 3 blocks\\  of size $\geqslant$2\end{tabular}}};
  \node[main node] (11211) [right of=5bl] {(1,1,2,1,1)};
  \node[main node] (11k11) [right of=11211] {\mbox{\begin{tabular}{c}(1,1,k,1,1)\\  k $\geqslant$ 2\end{tabular}}};
  \node[main node] (11112) [below right=0.5 and -0.5 of 11211] {(1,1,1,1,2)};
  \node[main node] (1111k) [right of=11112] {\mbox{\begin{tabular}{c}(1,1,1,1,k)\\  k $\geqslant$ 2\end{tabular}}};
 \node[main node] (5B)[right of=1111k] {\scalebox{0.8}{\mbox{\begin{tabular}{c}$5$ blocks \\ with 2 blocks\\  of size $\geqslant$2  \end{tabular}}}}; 

  \node[first node] (12121) [below right=0.5 and -0.5 of 11112] {(1,2,1,2,1)};
  \node[main node] (1122) [right of =4bl] {(1,1,2,2)};
  \node[main node] (11kkp) [right of=1122] {\mbox{\begin{tabular}{c}(1,1,k,k')\\  k,k' $\geqslant$ 2\end{tabular}}};
  \node[first node] (1214) [right=2cm of 11kkp] {(1,2,1,4)};
  \node[main node] (2112) [below right=0.3 and 0.3 of 1122] {(2,1,1,2)};
  \node[main node] (k11kp) [right of=2112] {\mbox{\begin{tabular}{c}(k,1,1,k')\\  k,k' $\geqslant$ 2\end{tabular}}};
  \node[main node] (1k1kp) [right of=1214] {\mbox{\begin{tabular}{c}(1,k,1,k')\\  k $\geqslant$ 2, k'$\geqslant$ 4\end{tabular}}};
  \node[first node] (146) [right=1cm  of 222] {(1,4,6)};
  \node[first node] (1441) [right=2cm of k11kp] {(1,4,4,1)};
  \node[main node] (1kkp1) [right of=1441] {\mbox{\begin{tabular}{c}(1,k,k',1)\\  k,k' $\geqslant$ 4\end{tabular}}};
  \node[main node] (1kkp) [below right=0.1 and 0.9 of 146] {\mbox{\begin{tabular}{c}(1,k,k')\\  k $\geqslant$ 4, k'$\geqslant$ 6\end{tabular}}};
  \node[first node] (414) [right=3.2cm of 146] {(4,1,4)};
  \node[main node] (k1kp) [right of=414] {\mbox{\begin{tabular}{c}(k,1,k')\\  k,k' $\geqslant$ 4\end{tabular}}};
  \node[first node] (66) [right of=2bl] {(6,6)};
  \node[main node] (66plus) [right of=66] {2 blocks $\geqslant$ 6};

\node (leg2) [right=4.5cm of 6B] {Key:};
\node[first node] (leg3) [right=2cm of leg2] {\begin{tabular}{c}initial\\ infinite case\end{tabular}};
\node(leg3p) [right=-0.2cm of leg3]{\begin{tabular}{c}see Figures \ref{fig:D4}, \ref{fig:E6}\\ and Prop \ref{prop:class_inf}\end{tabular}};
\node[main node] (leg4) [below left= 3cm and -0.5cm of leg3] {\begin{tabular}{c}deduced\\ case\end{tabular}};
\node[main node] (leg5) [below right=3cm and -0.5cm of leg3] {\begin{tabular}{c}deduced\\ case\end{tabular}};
\path[every edge/.style={split}]
(222) edge[bend left=28]   (111111)
 edge[bend left=18]  (11211)
edge[bend left=38] (11112)
edge  (1122)
edge (2112)
;
\path[split, draw] (leg3) -- node[left]{Lemma \ref{lem:red_subgroup}} (leg4);
\path[plus, draw] (leg3) -- node[right]{Lemma \ref{lem:red_induction}} (leg5);
\path[every edge/.style={plus}]
(222) edge [bend left] (222B)
(111111) edge (6B)
(11211) edge (11k11)
(11112) edge (1111k)
(12121) edge[bend right] (5B)
(11k11) edge[bend left] (5B)
(1111k) edge (5B)
(1122) edge (11kkp)

(2112) edge (k11kp)
(1214) edge (1k1kp)
(1441) edge (1kkp1)
(146) edge (1kkp)
(414) edge (k1kp)
(66) edge (66plus);

\end{tikzpicture}}\end{center}

\subsection{Finite cases}\label{app:fin_case_diag}

\begin{center}
\scalebox{0.5}{
$\!\!\!$\begin{tikzpicture}[->,>=stealth',shorten >=1pt,auto,node distance=3cm,
  thick,main node/.style={
draw, ellipse, node distance={20ex}, minimum size=2cm,
  font=\sffamily\small\bfseries,minimum size=15mm},descr node/.style={fill=blue!20,draw},first node/.style={fill=blue!20,draw,
draw, ellipse, node distance={20ex}, minimum size=2cm,
  font=\sffamily\small\bfseries,minimum size=15mm},
conj node/.style={fill=green!20,draw,
draw, ellipse, node distance={20ex}, minimum size=2cm,
  font=\sffamily\small\bfseries,minimum size=15mm},
node/.style={}]
\tikzstyle{split} = [draw, {-angle 90 reversed},]
\tikzstyle{plus} = [draw,{-o}]

\node[descr node] (5bl) {5 blocks};
\node[descr node] (4bl)[below=2cm of 5bl] {4 blocks};
\node[descr node] (3bl)[below=4.5cm of 4bl] {3 blocks};
\node[descr node] (2bl)[below=5cm of 3bl] {2 blocks};

  \node[conj node] (111k1) [right of=5bl] {\mbox{\begin{tabular}{c}(1,1,1,k,1)\\  k$\geqslant$ 1\end{tabular}}};
  
   \node[main node] (111k) [right of=4bl] {\mbox{\begin{tabular}{c}(1,1,1,k)\\  k$\geqslant$ 1\end{tabular}}};
 
 \node[conj node] (13k1) [right=1.5cm  of 111k] {\mbox{\begin{tabular}{c}(1,3,k,1)\\  k$\geqslant$ 1\end{tabular}}};
  \node[main node] (12k1) [below right=0.7 and 0.1 of 13k1] {\mbox{\begin{tabular}{c}(1,2,k,1)\\  k$\geqslant$ 1\end{tabular}}};
  
   \node[conj node] (1k13) [right=4cm of 13k1] {\mbox{\begin{tabular}{c}(1,k,1,3)\\  k$\geqslant$ 1\end{tabular}}};
\node[main node] (1k12) [below left=0.7 and 0.1 of 1k13] {{\mbox{\begin{tabular}{c}(1,k,1,2)\\  k$\geqslant$ 1\end{tabular}}}};

 \node[main node] (11k1) [below left=0.7 and -0.5 of 1k12] {\mbox{\begin{tabular}{c}(1,1,k,1)\\  k$\geqslant$ 1\end{tabular}}}; 

\node[conj node] (1k5) [right of=3bl] {\mbox{\begin{tabular}{c}(1,k,5)\\k$\geqslant$ 1 \end{tabular}}};
  \node[main node] (1k4) [below right=0.3 and 0.1 of 1k5] {\mbox{\begin{tabular}{c}(1,k,4)\\k$\geqslant$ 1 \end{tabular}}};
 \node[main node] (1k3) [below right=0.3 and 0.1 of 1k4] {\mbox{\begin{tabular}{c}(1,k,3)\\k$\geqslant$ 1 \end{tabular}}}; 
  \node[main node] (1k2) [below right=0.3 and 0.1 of 1k3] {\mbox{\begin{tabular}{c}(1,k,2)\\ k $\geqslant$ 1\end{tabular}}};
\node[main node] (1k1) [below right=0.3 and 0.1 of 1k2] {\mbox{\begin{tabular}{c}(1,k,1)\\ k $\geqslant$ 1\end{tabular}}};

\node[main node] (13k) [right=2 of 1k5] {\mbox{\begin{tabular}{c}(1,3,k)\\ k $\geqslant$ 1\end{tabular}}}; 
 \node[main node] (12k) [below right=0.7 and 0.1 of 13k] {\mbox{\begin{tabular}{c}(1,2,k)\\ k $\geqslant$ 1\end{tabular}}};
 
  \node[main node] (31k) [right=4.5cm of 13k] {\mbox{\begin{tabular}{c}(3,1,k)\\ k $\geqslant$ 1\end{tabular}}};
 \node[main node] (21k) [below left=0.7 and 0.1 of 31k] {\mbox{\begin{tabular}{c}(2,1,k)\\ k $\geqslant$ 1\end{tabular}}};  
 \node[main node] (11k) [below left=0.7 and 0.1 of 21k] {\mbox{\begin{tabular}{c}(1,1,k)\\ k $\geqslant$ 1\end{tabular}}};
 
  \node[main node] (5k) [right of=2bl] {\mbox{\begin{tabular}{c}(5,k)\\ k$\geqslant$ 1\end{tabular}}};
    \node[main node] (kkp) [below right=0.4 and 0.1 of 5k] {\mbox{\begin{tabular}{c}(k,k')\\ k$\leqslant$ 4\\ k'$\geqslant$ 1 \end{tabular}}};

\node (leg2) [right=9.5cm of 111k1] {Key:};
\node[conj node] (leg3) [right=2cm of leg2] {\begin{tabular}{c}initial\\ finite case\end{tabular}};
\node(leg3p) [right=-0.2cm of leg3]{\begin{tabular}{c}see Lemma \ref{lem:fin_cases}\end{tabular}};

\node[main node] (leg4) [below left= 3cm and -0.5cm of leg3] {\begin{tabular}{c}deduced\\ case\end{tabular}};
\node[main node] (leg5) [below right=3cm and -0.5cm of leg3] {\begin{tabular}{c}deduced\\ case\end{tabular}};
\path[every edge/.style={split}]
(12k1) edge[bend right=40] (111k1)
(1k12) edge[bend right=30] (111k1)
(12k) edge[bend right=30] (111k)
(21k) edge (111k)
(1k2) edge[bend left=35] (11k1)
(1k3) edge[bend left=40] (12k1)
(1k4) edge[bend left=10] (13k1)
(1k4) edge[bend left=38] (1k13)
;
\path[split, draw] (leg4) -- node[left]{Lemma \ref{lem:red_subgroup}} (leg3);
\path[plus, draw] (leg5) -- node[right]{Lemma \ref{lem:red_induction}} (leg3);
\path[every edge/.style={plus}]
(1k12) edge (1k13)
(11k1) edge (1k12)
(31k) edge (1k13)
(13k) edge (13k1)
(12k) edge (12k1)
(11k) edge (11k1)
(21k) edge (1k12)
(31k) edge (1k13)
(12k1) edge (13k1)
(11k1) edge (12k1)
(111k) edge (111k1)
(11k) edge (21k)
(11k) edge (12k)
(12k) edge (13k)
(21k) edge (31k)
(1k1) edge (1k2)
(1k2) edge (1k3)
(1k3) edge (1k4)
(1k4) edge (1k5)
(5k) edge (1k5)
(kkp) edge (5k);
\end{tikzpicture}}
\end{center}

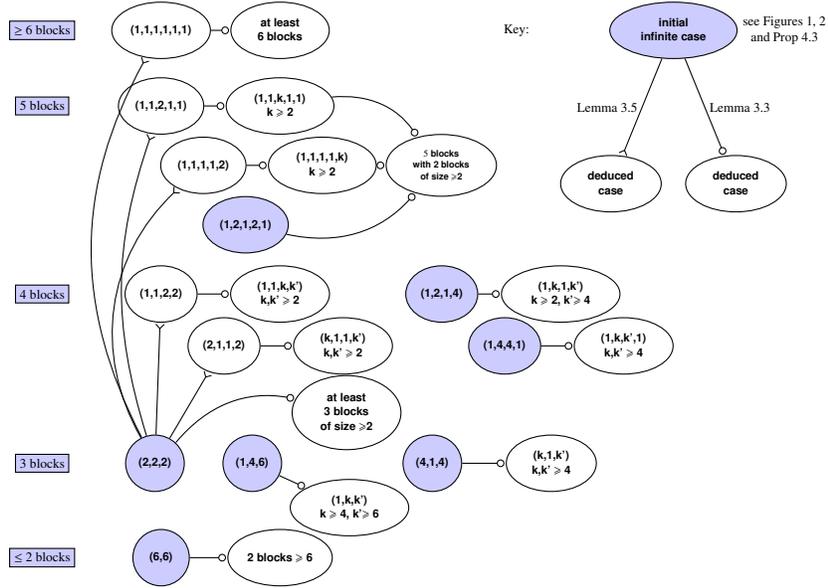
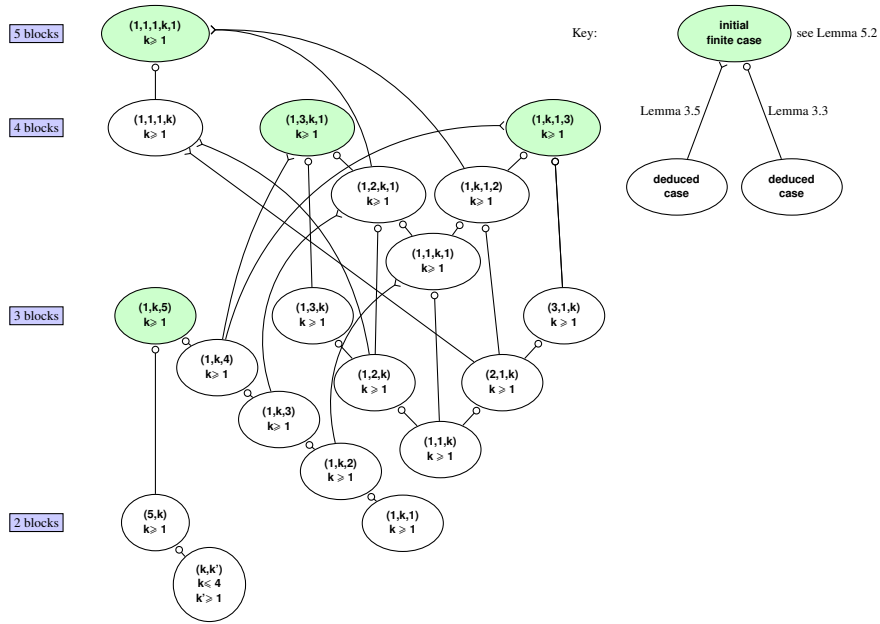
\end{document}